\documentclass[12pt,sloppy]{article}
\usepackage{graphicx}

\usepackage{multirow}
\usepackage{subfigure}
\usepackage{latexsym}
\usepackage{amsfonts}
\usepackage{amsthm}
\usepackage{amsmath}
\usepackage{caption}
\usepackage{chngpage}

\author{  {\bf Maya Mohsin Ahmed} \\  maya.ahmed@gmail.com}
\title{{\LARGE {\bf Namaste Franklin, 
from the magic squares of Narayana Pandita
}}}
\date{}

\newtheorem{prop}{Proposition}[section]
\newtheorem{lemma}{Lemma}[section]
\newtheorem{defn}{Definition}[section]
\newtheorem{exm}{Example}[section]
\newtheorem{algo}{Algorithm}[section]
\newtheorem{corollary}{Corollary}[section]

\newcounter{magicrownumbers}
\newcommand\rownumber{\stepcounter{magicrownumbers}\arabic{magicrownumbers}}

\begin{document}     
\maketitle           

\begin{abstract}
 Narayana Pandita constructed magic squares as a superimposition of
 two squares, folded together like palms in the Indian greeting,
 Namaste. In this article, we show how to construct Franklin squares
 of every order, as a superimposition of two squares. We also explore
 the myriad of similarities in construction and properties of Franklin
 and Narayana squares.
\end{abstract}

\section{Introduction.} \label{introsection}

A {\em magic square} is a square matrix whose entries are non-negative
integers, such that the sum of the numbers in every row, in every
column, and in each diagonal is the same number called the {\em magic
  sum}. Narayana Pandita, in the thirteenth century, showed how to
construct the magic squares in Table \ref{tabenarayanasquares8n16}
(\cite{datta}, \cite{narayana}).  The well-known squares F1, F2, and
F3, that appear in Table \ref{tabefranklinsquaresf1f3}, were
constructed by Benjamin Franklin, in the eighteenth century
(\cite{mathesis}, \cite{andrews}, \cite{pasles}).  We start with
looking at the mind blowing structure and properties of these squares.
\begin{table}
  {\scriptsize
    \caption{Narayana squares N1 and N2.}
    \label{tabenarayanasquares8n16} 
    \begin{adjustwidth}{-2cm}{-1cm} 
      \[ \begin{array}{ccccccc}
        \begin{array}{|c|c|c|c|c|c|c|c|} \hline
          60 & 53 & 44 & 37 & 4 & 13 & 20 & 29 \\ \hline
          3 & 14 & 19 & 30 & 59 & 54 & 43 & 38 \\ \hline
          58 & 55 & 42 & 39 & 2 & 5 & 18 & 31 \\ \hline
          1 & 16 & 17 & 32 & 57 & 56 & 41 & 40 \\ \hline
          61 & 52 & 45 & 36 & 5 & 12 & 21 & 28 \\ \hline
          6 & 11 & 22 & 27 & 62 & 51 & 46 & 35 \\ \hline
          63 & 50 & 47 & 34 & 7 & 10 & 23 & 26 \\ \hline
          8 & 9 & 24 & 25 & 64 & 49 & 48 & 33 \\ \hline
        \end{array} & & 
        \begin{array}{|c|c|c|c|c|c|c|c|c|c|c|c|c|c|c|c|} \hline
          248 & 233 & 216 & 201 & 184 & 169 & 152 & 137 & 8 & 25 & 40 & 57 & 72 & 89 & 104 & 121 \\ \hline
          7 & 26 & 39 & 58 & 71 & 90 & 103 &122 & 247 & 234 & 215 & 202 & 183 & 170 & 151 & 138 \\ \hline
          246 & 235 & 214 & 203 & 182 & 171 & 150 & 139 & 6 & 27 & 38 & 59 & 70 & 91 & 102 & 123 \\ \hline 
          5 & 28 & 37 & 60 & 69 & 92 & 101 & 124 & 245 & 236 & 213 & 204 & 181 & 172 & 149 & 140 \\ \hline
          244 & 237 & 212 & 205 & 180 & 173 & 148 & 141 & 4 & 29 & 36 & 61 & 68 & 93 & 100 & 125 \\ \hline
          3 & 30 & 35 & 62 & 67 & 94 & 99 & 126 & 243 & 238 & 211 & 206 & 179 & 174 & 147 & 142 \\ \hline
          242 & 239 & 210 & 207 & 178 & 175 & 146 & 143 & 2 & 31 & 34 & 63 & 66 & 95 & 98 & 127 \\ \hline
          1  & 32 & 33 & 64 & 65 & 96 & 97 & 128 & 241 & 240 & 209 & 208 & 177 & 176 & 145 & 144 \\ \hline
          249 & 232 & 217 & 200 & 185 & 168 & 153 & 136 & 9 & 24 & 41 & 56 & 73 & 88 & 105 & 120 \\ \hline
          10 & 23 & 42 & 55 & 74 & 87 & 106 & 119 & 250 & 231 & 218 & 199 & 186 & 167 & 154 & 135 \\ \hline
          251 & 230 & 219 & 198 & 187 & 166 & 155 & 134 & 11 & 22 & 43 & 54 & 75 & 86 & 107 & 118 \\ \hline
          12 & 21 & 44 & 53 & 76 & 85 & 108 & 117 & 252 & 229 & 220 & 197 & 188 & 165 & 156 & 133 \\ \hline
          253 & 228 & 221 & 196 & 189 & 164 & 157 & 132 & 13 & 20 & 45 & 52 & 77 & 84 & 109 & 116 \\ \hline
          14 & 19 & 46 & 51 & 78 & 83 & 110 & 115 & 254 & 227 & 222 & 195 & 190 & 163 & 158 & 131 \\ \hline
          255 & 226 & 223 & 194 & 191 & 162 & 159 & 130 & 15 & 18 & 47 & 50 & 79 & 82 & 111 & 114  \\ \hline
          16 & 17 & 48 & 49 & 80 & 81 & 112 & 113 & 256 & 225 & 224 & 193 & 192 & 161 &  160 & 129 \\ \hline
      \end{array} \end{array} 
      \]
    \end{adjustwidth}
  }
\end{table}

\begin{table} {\scriptsize
    \caption{ Franklin squares.}
    \label{tabefranklinsquaresf1f3} 
    \begin{adjustwidth}{-2cm}{-1cm} 
      \[ \begin{array}{ccccccccccc}

        \begin{array}{ccccccc}
          F1 & F2 \\
          \begin{array}{|c|c|c|c|c|c|c|c|} \hline
            52 & 61 & 4 & 13 & 20 & 29 & 36 & 45 \\ \hline
            14 & 3 & 62 & 51 & 46 & 35 & 30 & 19 \\ \hline
            53 & 60 & 5 & 12 & 21 & 28 & 37 & 44 \\ \hline
            11 & 6 & 59 & 54 & 43 & 38 & 27 & 22 \\ \hline
            55 & 58 & 7 & 10 & 23 & 26 & 39 & 42 \\ \hline
            9 & 8 & 57 & 56 & 41 & 40 & 25 & 24 \\ \hline
            50 & 63&  2 & 15 & 18 & 31 & 34 & 47 \\ \hline
            16 & 1 & 64 & 49 & 48 & 33 & 32 & 17 \\ \hline
          \end{array}
          &
          \begin{array}{|c|c|c|c|c|c|c|c|} \hline
            17 & 47 & 30 & 36 & 21 & 43 & 26 & 40 \\ \hline
            32 & 34 & 19 & 45 & 28 & 38 & 23 & 41 \\ \hline
            33 & 31 & 46 & 20 & 37 & 27 & 42 & 24 \\ \hline
            48 & 18 & 35 & 29 & 44 & 22 & 39 & 25 \\ \hline
            49 & 15 & 62 & 4 & 53 & 11 & 58 & 8 \\ \hline
            64 & 2 & 51 & 13 & 60 & 6 & 55 & 9 \\ \hline
            1 & 63 & 14 & 52 & 5 & 59 & 10 & 56 \\ \hline
            16 & 50 & 3 & 61 & 12 & 54 & 7 & 57 \\ \hline
          \end{array}
        \end{array} \\ \\
        \begin{array}{ccccccc}
          F3 \\
          \begin{array}{|c|c|c|c|c|c|c|c|c|c|c|c|c|c|c|c|} \hline
            200 & 217 & 232 & 249 & 8 & 25 & 40 & 57 & 72 & 89 & 104 & 121 & 136 & 153 & 168 & 185 \\ \hline
            58 & 39 & 26 & 7 & 250 & 231 & 218 & 199 & 186 & 167 & 154 & 135 & 122 & 103 & 90 & 71 \\ \hline
            198 & 219 & 230 & 251 & 6 & 27 & 38 & 59 & 70 & 91 & 102 & 123 & 134 & 155 & 166 & 187 \\ \hline
            60 & 37 & 28 &  5 & 252& 229 & 220 & 197 & 188 & 165 & 156 & 133 & 124 & 101 & 92 & 69  \\ \hline
            201 & 216 & 233 & 248 & 9  & 24  & 41  & 56 & 73 & 88 & 105 & 120 & 137 & 152 & 169 & 184  \\ \hline
            55 & 42 & 23 & 10 & 247 & 234 & 215 & 202 & 183 & 170 & 151 & 138 & 119 & 106 &  87 & 74  \\ \hline
            203 & 214 & 235 & 246 & 11 & 22 & 43 & 54 & 75 & 86 & 107 & 118 & 139 & 150 & 171 & 182  \\ \hline
            53 & 44 & 21 & 12 & 245 & 236 & 213 & 204 & 181 & 172 & 149 & 140 & 117 & 108 & 85 & 76  \\ \hline

            205 & 212 & 237 & 244 &  13 & 20 & 45 &  52 &  77 &  84 & 109 & 116 & 141 & 148 & 173 & 180 \\ \hline 
            51 & 46 & 19 &  14 &  243 & 238 & 211 & 206 & 179 & 174 & 147 & 142 & 115 & 110 & 83 &  78  \\ \hline
            207 & 210 & 239 & 242 & 15 & 18 &  47 &  50 &  79 &  82 &  111 & 114 & 143 & 146 & 175 & 178  \\ \hline
            49 &  48 & 17 & 16 &  241 & 240 & 209 & 208 & 177 & 176 & 145 & 144 & 113 & 112 & 81 &  80  \\ \hline
            196 &221 & 228 & 253 & 4 & 29 &  36 & 61 & 68 & 93 & 100 & 125 & 132 & 157 & 164 & 189  \\ \hline
            62  & 35 &  30 &  3 & 254 & 227 & 222 & 195 & 190 & 163 & 158 & 131 & 126 & 99 & 94 & 67  \\ \hline
            194 & 223 & 226 & 255 & 2 & 31 &  34 &  63 &  66 &  95 &  98 &  127 & 130 & 159 & 162 & 191  \\ \hline
            64  & 33  & 32 & 1 & 256 & 225 & 224 &  193 & 192 & 161 & 160 & 129 & 128 & 97 & 96 &  65  \\ \hline
          \end{array}
        \end{array}
      \end{array}
      \]
    \end{adjustwidth}
  }
\end{table}

\begin{figure}[h]
  \begin{center}
    \includegraphics[scale=0.3]{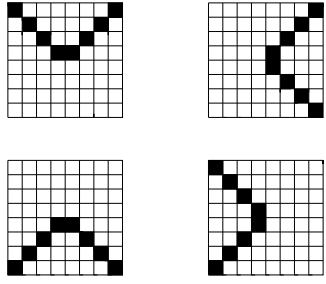}
    \caption{The four main bend diagonals \cite{pasles}} \label{bent_diags}
  \end{center}
\end{figure}
\begin{figure}[h]
  \begin{center}
    \includegraphics[scale=0.4]{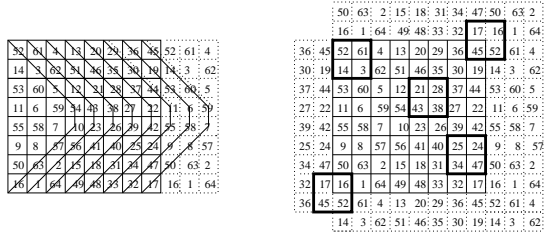}
    \caption{Continuous properties of bend diagonals and  $2 \times 2$ sub-squares.} \label{parallelbends}
  \end{center}
\end{figure}

\begin{figure}[h]
  \begin{center}
    \includegraphics[scale=0.5]{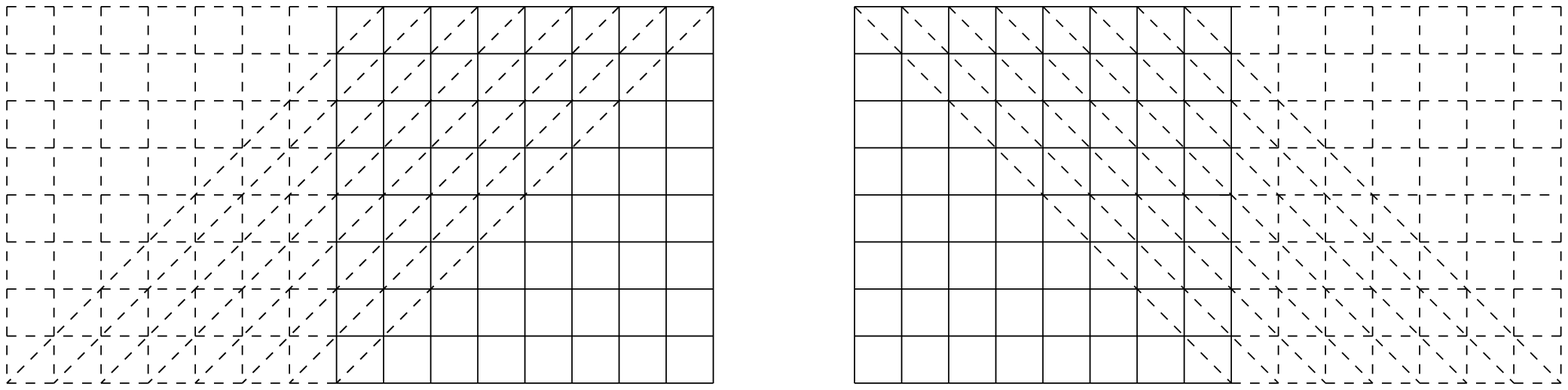}
    \caption{Right and left Pandiagonals.} \label{pandiagonas}
  \end{center}
\end{figure}

By a {\em continuous property}, we mean that if we imagine the square
as the surface of a torus (i.e., if we glue opposite sides of the
square together), then the property can be translated without effect on
the corresponding sums.  See Figure \ref{parallelbends} for
examples. From now on, row sum, column sum, or bend diagonal sum,
etc.  mean that we are adding the entries of those elements.  {\em
  Franklin squares} were defined in \cite{ahmed1} as follows.

\begin{defn}[Franklin Square]Consider
  an integer, $n=2^r$ such that $r \geq 3$. Let the magic sum
  be  denoted by  $M$ and  $N=n^2+1$.  We define  an {\em  $n
    \times n$  Franklin square} to  be a $n \times  n$ matrix
  with the following properties:
  \begin{enumerate}
  \item  Every  integer from  the  set  $\{1,2, \dots,  n^2\}$
    occurs  exactly once  in the  square. Consequently,  \[M =
    \frac{n}{2} N.\]
  \item All  the the half  rows, half columns add  to one-half
    the magic sum. Consequently,  all the rows and columns add
    to the magic sum.
  \item All the bend diagonals add to the magic sum, continuously (see Figures \ref{bent_diags} and \ref{parallelbends}).
  \item All  the $2  \times 2$ sub-squares  add to  $2N$,
    continuously.  
  \end{enumerate}
\end{defn}

Observe that Franklin squares are not traditional magic squares
because the main diagonals do not add to the magic sum.  In this
article, we restrict our discussion to Narayana squares of order
$2^r$, where $r \geq 3$. This is for comparison with Franklin
squares. For more general Narayana squares, see \cite{datta} and
\cite{narayana}.

\begin{defn}[Narayana Square]Consider
  an integer, $n=2^r$ such that $r \geq 3$. Let the magic sum
  be  denoted by  $M$ and  $N=n^2+1$.  We define  an {\em  $n
    \times n$  Narayana square} to  be a $n \times  n$ matrix
  with the following properties:
  \begin{enumerate}
  \item  Every  integer from  the  set  $\{1,2, \dots,  n^2\}$
    occurs  exactly once  in the  square. Consequently,  \[M =
    \frac{n}{2} N.\]
  \item All the rows and columns add to the magic sum.
  \item All the pandiagonals add to the magic sum  (see Figure \ref{pandiagonas}).
  \item All  the $2  \times 2$ sub-squares  add to  $2N$,
    continuously.  
  \end{enumerate}
\end{defn}

Thus, Narayana squares are magic squares with additional properties.
Though the squares were constructed by two different people, across
different centuries and continents, the similarities in properties and
construction are remarkable, and we address some of the properties in the
next proposition.

\begin{prop} \label{similarpropprop}
  Consider an integer, $n=2^r$ such that $r \geq 3$. Let the magic sum
  be denoted by $M$, $m=n^2/2+1$, and $N=n^2+1$. The similarities and
  differences in the defining properties of Franklin and Narayana
  squares are given below. 
  
  \begin{tabular}{|c|p{0.4\linewidth}|p{0.4\linewidth}|} \hline
    \multicolumn{2}{|c|}{Franklin Square} & \multicolumn{1}{|c|}{Narayana Square} \\ \hline
    \rownumber &$M = \frac{n}{2}{N}$ & $M = \frac{n}{2}{N}$ \\  \hline

    \rownumber &  All the rows and columns add to $M$ &  All the rows and columns add to $M$. \\ \hline

    \rownumber &  All  the $2  \times 2$ sub-squares  add to  $2N$,

    continuously. &  All  the $2  \times 2$ sub-squares  add to  $2N$,
    continuously. \\ \hline

    \rownumber & All  the the half row  sums add to $M/2$. & 
    Half row sums add either to $(n/4)m$ or $M-(n/4)m$.  \\ \hline

    \rownumber & Half column sums add to $M/2$. & Half column sums add to  $M/2 \pm n^2/8$. \\ \hline

    \rownumber & All Bend diagonal sums add to $M$. & Left and right
    bend diagonal sums add to $ M \pm n/2$. Top and bottom bend
    diagonal sums add to $ M \pm n^2/2$.  \\ \hline

    \rownumber & Pandiagonal sums add  to $M \pm n^2/2$. & All pandiagonals add to $M$.  \\ \hline
  \end{tabular}

\end{prop}

Proof of Proposition \ref{similarpropprop} is covered in Section
\ref{franklinsection} and Section \ref{newmethodsection}. Many other
properties of these squares are also revealed in these sections.

We describe Narayana's method of constructing magic squares.  For
constructing an $n \times n$ Narayana square, we start with two
sequences, namely,
\[
\mbox{{\em Mulapankti sequence:} }  1,2, \dots, n 
\] and \[  \mbox{\em{Ganapankti sequence:} } 0, n, 2n, 3n,  \dots, (n-1)n. \]

We construct arrays consisting of two rows, from these sequences, as
shown below (also see Examples \ref{8x8narayannamaste} and
\ref{16x16narayannamaste}).
\[ \begin{array}{cccccccccccccc}
  \mbox{{\em Mulapankti array}}  \\ \\
  \begin{array}{ccccccccccccccccccccc}
    1&  2 &  \dots  & \frac{n}{2}-1 & \frac{n}{2} \\ \\
    n & n-1 & \dots &   \frac{n}{2}+2 & \frac{n}{2}+1 
  \end{array}
\end{array} 
\]
\[ \begin{array}{cccccccccccccc}
  \mbox{{\em Ganapankti array}} \\ \\ 
  \begin{array}{cccccccccccccccccc}
    0 & n & 2n & \dots & n(\frac{n}{2}-2) & n(\frac{n}{2}-1) \\ \\
    n(n-1) & n(n-2) &n(n-3) &  \dots & n(\frac{n}{2}+1)& n\frac{n}{2}
  \end{array} 
\end{array} 
\]

Two squares, a covering one called {\em Chadaka}, and a square to be
covered called {\em Chadya} are formed as follows.  We start with the
last column of the Mulapankti array, working backwards. Each column of
the Mulapankti array is written horizontally, repeatedly, to form the
first $n/2$ rows of the Chadya. The two entries in each column are
inverted, and again written horizontally, to form the last $n/2$ rows
of the Chadya.  The Chadaka is constructed using the Ganapankti
array. This construction is similar to Chadya, except that, the
columns of the Ganapankti array are written vertically, repeatedly, to
form the first $n/2$ columns of the Chadaka. Again, the entries of the
Ganapankti columns are reversed to construct the last $n/2$
columns. The Chadya and the Chadaka are then superimposed like the
folding of palms in a Namaste to form the Narayana magic square. In
other words, the Chadaka is flipped about a vertical edge and added to
the Chadya to get the magic square. The square we get by flipping the
Chadaka about a vertical edge is called {\em flipped Chadaka} from now
on-wards. Example \ref{8x8narayannamaste} and Example
\ref{16x16narayannamaste} demonstrates the Narayana construction for
the $8 \times 8$ and $16 \times 16$ Narayana square, respectively.

\begin{exm} \label{8x8narayannamaste}
  {\em Narayana's method of constructing $8 \times 8$ magic squares. 

    We demonstrate how to construct the Narayana square N1 in Table
    \ref{tabenarayanasquares8n16}.
    \[\begin{array}{lllllllllll} \mbox{Mulapankti sequence:} & 1, 2, 3, 4, 5, 6, 7, 8. \end{array}\]
    \[\begin{array}{lllllllllll} \mbox{Ganapankti sequence:} & 0, 8, 16, 24, 32, 40, 48, 56. \end{array}\]
    \[ \begin{array}{cccccccccccccc}
      \mbox{Mulapankti array}  & &  \mbox{Ganapankti array} \\ 
      \begin{array}{ccccccccccccccccccccc}
        1 & 2 & 3 & 4 \\
        8 & 7 & 6 & 5  
      \end{array}
      &&
      \begin{array}{cccccccccccccccccc}
        0 & 8 & 16 & 24 \\
        56 & 48 & 40 & 32
      \end{array} 
    \end{array} 
    \]

    Each column of the Mulapankti array is written horizontally, repeatedly, to form the first $4$ rows of the Chadya. 
    {\scriptsize
      \[
      \begin{array}{|c|c|c|c|c|c|c|c|} \hline
        4 & 5 & 4 & 5 & 4 & 5 & 4 & 5 \\ \hline
        3&6&3&6&3&6&3&6 \\ \hline
        2&7&2&7&2&7&2&7 \\ \hline
        1 & 8 &1 & 8 &1 & 8 &1 & 8  \\ \hline
        &&&&&&& \\ \hline
        &&&&&&& \\ \hline
        &&&&&&& \\ \hline
        &&&&&&& \\ \hline
      \end{array} 
      \]
    } The entries of the Mulapankti array columns are flipped, and
    again written horizontally, repeatedly, to form the last four rows
    of the Chadya.  

    {\scriptsize
      \[ \begin{array}{c}
        \mbox{Chadya of N1} \\ 
        \begin{array}{|c|c|c|c|c|c|c|c|} \hline
          4 & 5 & 4 & 5 & 4 & 5 & 4 & 5 \\ \hline
          3&6&3&6&3&6&3&6 \\ \hline
          2&7&2&7&2&7&2&7 \\ \hline
          1 & 8 &1 & 8 &1 & 8 &1 & 8  \\ \hline
          5 & 4 & 5 & 4 &5 & 4 &5 & 4 \\ \hline 
          6&3&6&3&6&3&6 & 3 \\ \hline
          7&2&7&2&7&2&7 & 2 \\ \hline
          8 &1 & 8 &1 & 8 &1 & 8  & 1\\ \hline
        \end{array} 
      \end{array}
      \]
    }
    The columns of the Ganapankti array are written vertically,
    repeatedly, to form the first four columns of the Chadaka.
    {\scriptsize
      \[
      \begin{array}{|c|c|c|c|c|c|c|c|} \hline
        4 & 5 & 4 & 5 & & & & \\ \hline
        3&6&3&6&  & & &   \\ \hline
        2&7&2&7&  & & &   \\ \hline
        1 & 8 &1 & 8 &  & & &  \\ \hline
        5 & 4 & 5 & 4 &  & & &   \\ \hline 
        6&3&6&3&  & & &   \\ \hline
        7&2&7&2&  & & &  \\ \hline
        8 &1 & 8 &1 &  & & &   \\ \hline
      \end{array} 
      \]
    }
    The
    entries of the Ganapankti array columns are flipped to construct
    the last four columns.

    {\scriptsize
      \[ \begin{array}{ccccc}
        \mbox{Chadaka of N1} \\
        \begin{array}{|c|c|c|c|c|c|c|c|} \hline
          4 & 5 & 4 & 5 & 4 & 5 & 4 & 5 \\ \hline
          3&6&3&6&3&6&3&6 \\ \hline
          2&7&2&7&2&7&2&7 \\ \hline
          1 & 8 &1 & 8 &1 & 8 &1 & 8  \\ \hline
          5 & 4 & 5 & 4 &5 & 4 &5 & 4 \\ \hline 
          6&3&6&3&6&3&6 & 3 \\ \hline
          7&2&7&2&7&2&7 & 2 \\ \hline
          8 &1 & 8 &1 & 8 &1 & 8  & 1\\ \hline
        \end{array} 
        
      \end{array}
      \]
    }

    Finally, the Chadya and Chadaka are superimposed, like the folding
    of palms in a Namaste to form the Narayana magic square N1. That
    is, the Chadaka is flipped along a vertical edge and added to the
    Chadya to get N1.  

    {\scriptsize
      \[ \begin{array}{ccccccccccccccc}
        \mbox{Chadya} & &  \mbox{Flipped Chadaka} & & \mbox{Narayana square N1}\\ 
        \begin{array}{|c|c|c|c|c|c|c|c|} \hline
          4 & 5 & 4 & 5 & 4 & 5 & 4 & 5 \\ \hline
          3&6&3&6&3&6&3&6 \\ \hline
          2&7&2&7&2&7&2&7 \\ \hline
          1 & 8 &1 & 8 &1 & 8 &1 & 8  \\ \hline
          5 & 4 & 5 & 4 &5 & 4 &5 & 4 \\ \hline 
          6&3&6&3&6&3&6 & 3 \\ \hline
          7&2&7&2&7&2&7 & 2 \\ \hline
          8 &1 & 8 &1 & 8 &1 & 8  & 1\\ \hline
        \end{array} & + & 
        \begin{array}{|c|c|c|c|c|c|c|c|} \hline

          56 & 48 & 40 & 32 & 0 & 8 & 16 & 24 \\ \hline
          0 & 8 & 16 & 24 & 56 & 48 & 40 & 32 \\ \hline
          56 & 48 & 40 & 32 & 0 & 8 & 16 & 24 \\ \hline
          0 & 8 & 16 & 24 & 56 & 48 & 40 & 32 \\ \hline
          56 & 48 & 40 & 32 & 0 & 8 & 16 & 24 \\ \hline
          0 & 8 & 16 & 24 & 56 & 48 & 40 & 32 \\ \hline
          56 & 48 & 40 & 32 & 0 & 8 & 16 & 24 \\ \hline
          0 & 8 & 16 & 24 & 56 & 48 & 40 & 32 \\ \hline
        \end{array} &=& 

        \begin{array}{|c|c|c|c|c|c|c|c|} \hline
          60 & 53 & 44 & 37 & 4 & 13 & 20 & 29 \\ \hline
          3 & 14 & 19 & 30 & 59 & 54 & 43 & 38 \\ \hline
          58 & 55 & 42 & 39 & 2 & 5 & 18 & 31 \\ \hline
          1 & 16 & 17 & 32 & 57 & 56 & 41 & 40 \\ \hline
          61 & 52 & 45 & 36 & 5 & 12 & 21 & 28 \\ \hline
          6 & 11 & 22 & 27 & 62 & 51 & 46 & 35 \\ \hline
          63 & 50 & 47 & 34 & 7 & 10 & 23 & 26 \\ \hline
          8 & 9 & 24 & 25 & 64 & 49 & 48 & 33 \\ \hline
        \end{array}

      \end{array}
      \]
    }
} \end{exm}
\begin{exm} \label{16x16narayannamaste}
  {\em Narayana's method of constructing $16 \times 16$ magic squares.

    In this example, we construct the $16 \times 16$ Narayana square N2 in
    Table \ref{tabenarayanasquares8n16}.
    
    \[ \begin{array}{cccccccccccc}
      \mbox{Mulapankti sequence:}  &  1, 2, 3, 4, 5, 6, 7, 8, 9, 10 , 11, 12, 13, 14, 15, 16. 
    \end{array} 
    \]
    \[ \begin{array}{cccccccccccc}
      \mbox{Ganapankti sequence:} & 16, 32, 48, 64, 80, 96, 112, 128, 144, 160, 176, 192, 208, 224, 240. \\ \\
    \end{array} 
    \]
    \[ 
    \begin{array}{ccc}
      \mbox{Mulapankti array} && \mbox{Ganapankti array}\\

      \begin{array}{ccccccccccccccccccccccccccccc}
        1 & 2 & 3 & 4 & 5 & 6 & 7 & 8  \\
        16 & 15& 14 & 13 & 12 & 11 & 10 & 9  
      \end{array} 
      & & 
      \begin{array}{ccccccccccccccccccccccccccccc}
        0 & 16 & 32 & 48 & 64 & 80 & 96 & 112 \\
        240 & 224 & 208 & 192 & 176 & 160 & 144 & 128 \\

      \end{array} 
    \end{array}
    \]
    The Chadya and the  Chadaka of N2 is given below. The Chadya and the flipped Chadaka are added together to get the $16
    \times 16$ Narayana square N2.

    \begin{adjustwidth}{-2cm}{-1cm}
      {\scriptsize
        \[   \begin{array}{cccccccccccc}
          \mbox{Chadya of N2} & \mbox{Chadaka of N2} \\
          \begin{array}{|c|c|c|c|c|c|c|c|c|c|c|c|c|c|c|c|} \hline
            8 &  9 & 8 & 9 & 8 & 9 & 8 & 9 & 8 & 9 & 8 & 9 & 8 & 9 & 8 & 9\\  \hline
            7 & 10 & 7 & 10 &7 & 10 &7 & 10 &7 & 10 &7 & 10 &7 & 10 & 7 & 10 \\ \hline
            6 & 11 &6 & 11 &6 & 11 &6 & 11 &6 & 11 &6 & 11 &6 & 11 &6 & 11  \\ \hline
            5 & 12 &5 & 12 &5 & 12 &5 & 12 &5 & 12 &5 & 12 &5 & 12 &5 & 12  \\ \hline
            4 & 13 &4 & 13 &4 & 13 &4 & 13 &4 & 13 &4 & 13 &4 & 13 &4 & 13  \\ \hline
            3 & 14 &3 & 14 &3 & 14 &3 & 14 &3 & 14 &3 & 14 &3 & 14 &3 & 14  \\ \hline
            2 & 15 &2 & 15 & 2 & 15 &2 & 15 & 2 & 15 &2 & 15 & 2 & 15 &2 & 15 \\ \hline
            1 & 16 &1 & 16 &1 & 16 &1 & 16 &1 & 16 &1 & 16 &1 & 16 &1 & 16  \\ \hline
            9 &  8 & 9 &  8 & 9 &  8 & 9 &  8 & 9 &  8 & 9 &  8 &9 &  8 & 9 & 8 \\ \hline
            10 & 7 &10 & 7 &10 & 7 &10 & 7 &10 & 7 &10 & 7 & 10 & 7 & 10 & 7  \\ \hline
            11 & 6 & 11 & 6 &11 & 6 & 11 & 6 &11 & 6 & 11 & 6 &11 & 6 & 11 & 6   \\ \hline
            12& 5 &12& 5 &12& 5 &12& 5 &12& 5 &12& 5 &12& 5 &12& 5  \\ \hline
            13& 4 &13& 4 &13& 4 &13& 4 &13& 4 & 13& 4 &13& 4 &13& 4  \\ \hline
            14& 3  &14& 3  &14& 3  &14& 3  &14& 3  &14& 3  &14& 3  &14& 3   \\ \hline
            15& 2 &15& 2 &15& 2 &15& 2 &15& 2 &15& 2 & 15& 2 &15& 2 \\ \hline
            16 & 1 & 16 & 1 &16 & 1 & 16 & 1 &16 & 1 & 16 & 1 &16 & 1 & 16 & 1  \\ \hline
          \end{array} 
          &
          \begin{array}{|c|c|c|c|c|c|c|c|c|c|c|c|c|c|c|c|} \hline
            112  & 96 & 80 & 64 & 48 & 32 & 16 & 0 & 128 & 144 & 160 & 176 & 192 & 208 & 224 & 240 \\ \hline
            128 & 144 & 160 & 176 & 192 & 208 & 224 & 240 & 112  & 96 & 80 & 64 & 48 & 32 & 16 & 0\\ \hline
            112  & 96 & 80 & 64 & 48 & 32 & 16 & 0 & 128 & 144 & 160 & 176 & 192 & 208 & 224 & 240 \\ \hline
            128 & 144 & 160 & 176 & 192 & 208 & 224 & 240 & 112  & 96 & 80 & 64 & 48 & 32 & 16 & 0\\ \hline
            112  & 96 & 80 & 64 & 48 & 32 & 16 & 0 & 128 & 144 & 160 & 176 & 192 & 208 & 224 & 240 \\ \hline
            128 & 144 & 160 & 176 & 192 & 208 & 224 & 240 & 112  & 96 & 80 & 64 & 48 & 32 & 16 & 0\\ \hline
            112  & 96 & 80 & 64 & 48 & 32 & 16 & 0 & 128 & 144 & 160 & 176 & 192 & 208 & 224 & 240 \\ \hline
            128 & 144 & 160 & 176 & 192 & 208 & 224 & 240 & 112  & 96 & 80 & 64 & 48 & 32 & 16 & 0\\ \hline
            112  & 96 & 80 & 64 & 48 & 32 & 16 & 0 & 128 & 144 & 160 & 176 & 192 & 208 & 224 & 240 \\ \hline
            128 & 144 & 160 & 176 & 192 & 208 & 224 & 240 & 112  & 96 & 80 & 64 & 48 & 32 & 16 & 0\\ \hline
            112  & 96 & 80 & 64 & 48 & 32 & 16 & 0 & 128 & 144 & 160 & 176 & 192 & 208 & 224 & 240 \\ \hline
            128 & 144 & 160 & 176 & 192 & 208 & 224 & 240 & 112  & 96 & 80 & 64 & 48 & 32 & 16 & 0\\ \hline
            112  & 96 & 80 & 64 & 48 & 32 & 16 & 0 & 128 & 144 & 160 & 176 & 192 & 208 & 224 & 240 \\ \hline
            128 & 144 & 160 & 176 & 192 & 208 & 224 & 240 & 112  & 96 & 80 & 64 & 48 & 32 & 16 & 0\\ \hline
            112  & 96 & 80 & 64 & 48 & 32 & 16 & 0 & 128 & 144 & 160 & 176 & 192 & 208 & 224 & 240 \\ \hline
            128 & 144 & 160 & 176 & 192 & 208 & 224 & 240 & 112  & 96 & 80 & 64 & 48 & 32 & 16 & 0\\ \hline
          \end{array}
        \end{array}
        \]
      }
    \end{adjustwidth}

} \end{exm}

Franklin never revealed his method of constructing his squares. In
\cite{mathesis}, a method to construct Franklin squares using Hilbert
basis was developed. Later in \cite{ma2}, a method to construct
Franklin squares of every order, in particular, F1 and F3, was
derived.

In Section \ref{franklinsection}, we describe and prove the method
developed in \cite{ma2}. Let $N=n^2+1$. The strategy is to first place
the numbers $i$, where $i=1,2, \dots, n^2/2$, and then place the
numbers $N-i$, such that all the defining properties of the square are
satisfied. In this article, we call this method the $N-i$ {\em
  method}.  Then we show how the $N-i$ method produces two squares
Chadya and Chadaka which can be superimposed, like folded palms in
a Namaste, to construct Franklin squares. In other words, the method
of constructing Franklin squares in \cite{ma2} is only a slight
modification of Narayana's construction in \cite{narayana}.  See
Tables \ref{franklin8chadya} and \ref{chadya16x16franklin} for
examples. The $N-i$ method cannot be directly used to create F2.  See
\cite{ahmedotherfranklin} for a method to construct Franklin square
F2.

In Section \ref{newmethodsection}, we develop an $N-i$ method to
construct Narayana squares of every order. We also use the $N-i$
method to create new Narayana squares. Moreover, we show that the
$N-i$ method to construct Narayana squares, is the same as the
original Chadya-Chadaka method given by Narayana Pandita.

\begin{table} {\scriptsize
    \caption{ Chadya and Chadaka of Franklin square F1.} 
    \label{franklin8chadya} 
    \[
    \begin{array}{ccccccccccccc}
      \begin{array}{|c|c|c|c|c|c|c|c|} \hline
        4 & 5 & 4 & 5 & 4 & 5 & 4 & 5 \\ \hline
        6& 3&6&3&6&3&6&3 \\ \hline
        5 & 4 & 5 & 4 & 5 & 4 & 5 & 4 \\ \hline
        3&6&3&6&3&6&3 & 6\\ \hline
        7& 2&7 &2&7&2&7&2 \\ \hline
        1 & 8 &1 & 8 &1 & 8 &1 & 8  \\ \hline
        2& 7& 2&7 &2&7&2&7 \\ \hline
        8 &1 & 8 &1 & 8 &1 & 8 & 1 \\ \hline
      \end{array} &&
      \begin{array}{|c|c|c|c|c|c|c|c|} \hline
        40 & 32 & 24 & 16 & 8 & 0 & 56 & 48 \\ \hline
        16 & 24 & 32 & 40 & 48 & 56 & 0 &  8 \\ \hline
        40 & 32 & 24 & 16 & 8 & 0 & 56 & 48 \\ \hline
        16 & 24 & 32 & 40 & 48 & 56 & 0 &  8 \\ \hline 
        40 & 32 & 24 & 16 & 8 & 0 & 56 & 48 \\ \hline
        16 & 24 & 32 & 40 & 48 & 56 & 0 &  8 \\ \hline
        40 & 32 & 24 & 16 & 8 & 0 & 56 & 48 \\ \hline
        16 & 24 & 32 & 40 & 48 & 56 & 0 &  8 \\ \hline    
      \end{array}
    \end{array}
    \]}
\end{table}

\begin{table}  {\scriptsize
    \begin{adjustwidth}{-2cm}{-1cm} 
      \caption{Chadya and Chadaka of $16 \times 16$ Franklin Square F3.} 
      \label{chadya16x16franklin} 
      \[ 
      \begin{array}{ccccc} 
        \begin{array}{|c|c|c|c|c|c|c|c|c|c|c|c|c|c|c|c|} \hline
          8 & 9 & 8 & 9 & 8 & 9 & 8 & 9 \\ \hline
          10 & 7 & 10& 7 & 10 & 7 & 10 & 7 \\ \hline
          6 & 11 & 6 & 11 &  6 & 11 & 6 & 11 \\ \hline
          12 & 5 & 12 & 5 & 12 & 5 & 12 & 5 \\ \hline
          9 & 8 & 9 & 8 & 9 & 8 & 9 & 8 \\ \hline
          7 & 10 & 7 & 10 & 7 & 10 & 7 & 10 \\ \hline
          11 & 6 & 11 & 6 & 11 & 6 & 11 & 6 \\ \hline
          5 & 12 & 5 & 12 &  5 & 12 & 5 & 12 \\ \hline
          13 & 4 & 13 & 4 & 13 & 4 & 13 & 4 \\ \hline
          3 & 14 & 3& 14 & 3 & 14 & 3& 14 \\ \hline
          15 & 2 & 15 & 2& 15 & 2 & 15 & 2 \\ \hline
          1 & 16 & 1 & 16 & 1 & 16 & 1 & 16 \\ \hline
          4 & 13 & 4 & 13& 4 & 13 & 4 & 13 \\ \hline
          14 & 3 & 14 & 3 &    14 & 3 & 14 & 3 \\ \hline
          2 & 15 & 2 & 15& 2 & 15 & 2 & 15 \\ \hline
          16 & 1 & 16 & 1 &  16 & 1 & 16 & 1 \\ \hline
        \end{array} 
        \begin{array}{|c|c|c|c|c|c|c|c|c|c|c|c|c|c|c|c|} \hline
          8 & 9 & 8 & 9 & 8 & 9 & 8 & 9 \\ \hline
          10 & 7 & 10& 7 & 10 & 7 & 10 & 7 \\ \hline
          6 & 11 & 6 & 11 &  6 & 11 & 6 & 11 \\ \hline
          12 & 5 & 12 & 5 & 12 & 5 & 12 & 5 \\ \hline
          9 & 8 & 9 & 8 & 9 & 8 & 9 & 8 \\ \hline
          7 & 10 & 7 & 10 & 7 & 10 & 7 & 10 \\ \hline
          11 & 6 & 11 & 6 & 11 & 6 & 11 & 6 \\ \hline
          5 & 12 & 5 & 12 &  5 & 12 & 5 & 12 \\ \hline
          13 & 4 & 13 & 4 & 13 & 4 & 13 & 4 \\ \hline
          3 & 14 & 3& 14 & 3 & 14 & 3& 14 \\ \hline
          15 & 2 & 15 & 2& 15 & 2 & 15 & 2 \\ \hline
          1 & 16 & 1 & 16 & 1 & 16 & 1 & 16 \\ \hline
          4 & 13 & 4 & 13& 4 & 13 & 4 & 13 \\ \hline
          14 & 3 & 14 & 3 &    14 & 3 & 14 & 3 \\ \hline
          2 & 15 & 2 & 15& 2 & 15 & 2 & 15 \\ \hline
          16 & 1 & 16 & 1 &  16 & 1 & 16 & 1 \\ \hline
        \end{array} 
        &
        \begin{array}{|c|c|c|c|c|c|c|c|c|c|c|c|c|c|c|c|c|c|c|c|c|c|c|c|c|c|c|c|c|c|c|c|} \hline
          176 & 160 & 144 & 128 &  112 & 96 & 80& 64 & 48 & 32 & 16& 0 &   240  & 224 & 208 & 192   \\ \hline
          64 & 80 & 96 & 112 &  128 & 144 & 160 & 176  &   192 & 208 & 224& 240 & 0 & 16 & 32 & 48  \\ \hline
          176 & 160 & 144 & 128 &  112 & 96 & 80& 64 & 48 & 32 & 16& 0 &   240  & 224 & 208 & 192   \\ \hline
          64 & 80 & 96 & 112 &  128 & 144 & 160 & 176  &   192 & 208 & 224& 240 & 0 & 16 & 32 & 48  \\ \hline
          176 & 160 & 144 & 128 &  112 & 96 & 80& 64 & 48 & 32 & 16& 0 &   240  & 224 & 208 & 192   \\ \hline
          64 & 80 & 96 & 112 &  128 & 144 & 160 & 176  &   192 & 208 & 224& 240 & 0 & 16 & 32 & 48  \\ \hline
          176 & 160 & 144 & 128 &  112 & 96 & 80& 64 & 48 & 32 & 16& 0 &   240  & 224 & 208 & 192   \\ \hline
          64 & 80 & 96 & 112 &  128 & 144 & 160 & 176  &   192 & 208 & 224& 240 & 0 & 16 & 32 & 48  \\ \hline
          176 & 160 & 144 & 128 &  112 & 96 & 80& 64 & 48 & 32 & 16& 0 &   240  & 224 & 208 & 192   \\ \hline
          64 & 80 & 96 & 112 &  128 & 144 & 160 & 176  &   192 & 208 & 224& 240 & 0 & 16 & 32 & 48  \\ \hline
          176 & 160 & 144 & 128 &  112 & 96 & 80& 64 & 48 & 32 & 16& 0 &   240  & 224 & 208 & 192   \\ \hline
          64 & 80 & 96 & 112 &  128 & 144 & 160 & 176  &   192 & 208 & 224& 240 & 0 & 16 & 32 & 48  \\ \hline
          176 & 160 & 144 & 128 &  112 & 96 & 80& 64 & 48 & 32 & 16& 0 &   240  & 224 & 208 & 192   \\ \hline
          64 & 80 & 96 & 112 &  128 & 144 & 160 & 176  &   192 & 208 & 224& 240 & 0 & 16 & 32 & 48  \\ \hline
          176 & 160 & 144 & 128 &  112 & 96 & 80& 64 & 48 & 32 & 16& 0 &   240  & 224 & 208 & 192   \\ \hline
          64 & 80 & 96 & 112 &  128 & 144 & 160 & 176  &   192 & 208 & 224& 240 & 0 & 16 & 32 & 48  \\ \hline
        \end{array} 
      \end{array}
      \]
    \end{adjustwidth}

} \end{table}


\section{Franklin squares} \label{franklinsection}

In this section, we describe the $N-i$ method to construct Franklin
squares. We also show that the $N-i$ method can also be rewritten as a
Chadya-Chadaka method.

Let $N=n^2+1$. The strategy is to first place the numbers $i$, where
$i=1,2, \dots, n^2/2$, and then place the numbers $N-i$, such that,
all the defining properties of the square are satisfied. We start by
dividing the Franklin square in to two sides: the {\em left side}
consisting of the first $n/2$ columns and the {\em right side}
consisting of the last $n/2$ columns. The construction of the right
and left sides are largely independent of each other.  Each side is
further divided in to three parts: the {\em Top part} consisting of
the first $n/4$ rows, the {\em Middle part} consisting of the middle
$n/2$ rows, and the {\em Bottom part} consisting of the last $n/4$
rows.

{\em Distance of a column} in a given side is defined as the number of
columns between the given column and the center of the side. For
example, consider the left side. Here, the distance of the $n/4$ th
and the $n/4+1$ th column of the square (the middle columns of the
side) is zero whereas the distance of the first and last column of the
left side is $n/4$.

Each side is build, partially, two equidistant
columns at a time. For the left side of the Franklin square we start
from the middle two columns of the side, and navigate outwards, two
columns, at a time.  For the right side, we start with the first and
last columns, and navigate inward towards the center.

Given a side and a pair of equidistant columns, we denote the column
on the left of the center as $C_l$ and the column on the right of the
center as $C_r$.  Consider a given part with $r$ rows and a {\em
  starting number} $A$. There are only two operations involved for
such a part.  An {\em Up }operation where consecutive numbers from $A$
to $A+r$ are filled, in consecutive rows, starting from the bottom row
of the part, upwards, alternating between the columns $C_l$ and
$C_r$.   The only other operation is the {\em
  Down operation} where consecutive numbers from $A$ to $A+r$ are
filled, in consecutive rows, alternating between the columns $C_l$ and
$C_r$, starting from the top row of the part in a downward direction.
The {\em starting columns} of the parts are different for the
two sides and is given below.

\[
\begin{array}{|ccccc|} \hline
  & \mbox{Starting column} &&& \\ \hline
  \mbox{Operation}  & \mbox{Part}  & \mbox{Left side}  & \mbox{Right side} &   \\ \hline
  
  &  \mbox{Bottom}  & C_l  & C_r &\\
  \mbox{Up} & \mbox{Middle} & C_l & C_r & \\
  & \mbox{Top} & C_l & C_r  &\\ \hline
  &  \mbox{Bottom} & C_r  & C_l &\\
  \mbox{Down} & \mbox{Middle } & C_r & C_l & \\
  & \mbox{Top} & C_r & C_l & \\  \hline
\end{array} 
\]

For a given pair of equidistant columns, the sequence of operations
depends on the parity of the distance, and is as described below.
\begin{adjustwidth}{-2cm}{-1cm} 
  \[ \begin{array}{ccc}
    \mbox{Even distance} & & \mbox{Odd distance}  \\
    \begin{array}{|c|ccc|} \hline
      \mbox{Part}  & \mbox{Bottom} & \mbox{Top} & \mbox{Middle}\\ \hline
      \mbox{Operation} &  \mbox{Up} & \mbox{Up} & \mbox{Down} \\ \hline
      \mbox{Starting Number} & A & A+n/4 & A+n/2 \\ \hline
    \end{array}  & & 

    \begin{array}{|c|ccc|} \hline
      \mbox{Part}  & \mbox{Middle} & \mbox{Top} &  \mbox{Bottom}  \\ \hline
      \mbox{Operation} &  \mbox{Up} & \mbox{Down} & \mbox{Down} \\ \hline
      \mbox{Starting Number} & A & A+n/2 & A+3n/4 \\ \hline
    \end{array}
  \end{array}
  \]
\end{adjustwidth}

This sequence is same for the both the sides and will place $n$
consecutive numbers in the chosen two columns.  But the starting
number for the entire sequence of operations, depends on the distance
and the side.  For the left side, if the distance is $d$ then the
starting number is $nd+1$. Whereas for the right side, for columns at
a distance of $n/4-d$, the starting number is $nd+1+n^2/4$.

Finally, we complete the square as follows. For each side, the empty
cells in a row are filled with $N-i$, where $i$ is the entry in the
same row in the equidistant column.

\begin{exm}
  In Table \ref{tableleftf3steps}, the steps of partially filling the
  left side of the the $16 \times 16$ Franklin square F3 are
  illustrated. We start with the middle pair of columns. The distance of
  this pair from the center is zero, hence the starting number is $nd+1
  = 16\times 0 + 1 = 1$.  The sequence of operation is
  \[
  \begin{array}{|c|ccc|} \hline
    \mbox{Part}  & \mbox{Bottom} & \mbox{Top} & \mbox{Middle}\\ \hline
    \mbox{Operation} &  \mbox{Up} & \mbox{Up} & \mbox{Down} \\ \hline
    \mbox{Starting Number} & 1 &  5 & 9 \\ \hline
  \end{array}
  \]

  That is, we enter the numbers from $1$ to $4$ in the bottom part,
  starting from $C_l$, using the Up operation.  Next we enter the
  numbers from $5$ to $8$ in the top part, starting from $C_l$, using
  the Up operation. Finally, we enter the numbers from $8$ to $16$ in
  the middle part starting from column $C_r$, using the Down operation.

  In Step 2 of Table \ref{tableleftf3steps}, we consider equidistant
  columns of distance $1$ from the center. The starting number is $nd+1
  = 16 \times 1 + 1 = 17$.  Since the distance is odd, the sequence of
  operation is
  \[
  \begin{array}{|c|ccc|} \hline
    \mbox{Part}  & \mbox{Middle} & \mbox{Top} &  \mbox{Bottom}  \\ \hline
    \mbox{Operation} &  \mbox{Up} & \mbox{Down} & \mbox{Down} \\ \hline
    \mbox{Starting Number} & 17 &  25 & 29 \\ \hline
  \end{array}
  \]
  Thus, the numbers $17$ to $32$ are placed in the two columns using the
  above sequence of operations. Steps 3 and 4 demonstrate the placement
  of numbers from $33$ to $64$ in the rest of the columns of the left
  side of F3. See Table \ref{tabldistributeright} for the placement of
  the numbers from $65$ to $256$ in the right side of the $16 \times 16$
  Franklin square. See Table \ref{franklinstep1to4}, for the placement
  of numbers from $1$ to $32$ for the $8 \times 8$ Franklin Square F1.
  
  The filling of the empty cells with $N-i$, where $i$ is the entry in
  the same row in the equidistant column, for the left side of the
  square F3 is given in Table \ref{tableleftf3fillingN}. The completion
  of the right side of F3 is shown in Table
  \ref{rightfranklintablefillingwithN}. The final step of filling empty
  cells for the square F1 is given in Table \ref{franklinstepfinal}.
\end{exm}

The numbers $n^2/4+1$ to $n^2/2$ are entered in the right side of
Franklin square, starting from the last and first column, working
inwards. Consequently, we derive the right side from the left side as
follows. We swap the first $n/4$ columns with the last $n/4$ columns
of the partially filled left side and add $n^2/4$ to each entry. For
example, we swap the first four columns with the last four columns and
add $64$ to each element of the square in Step $4$ of Table
\ref{tableleftf3steps}. This gives us the square in Step 4 of Table
\ref{tabldistributeright}.  See Table \ref{firststeprightfromleft} for
an illustration. The last step of filling the empty squares with
$N-i$, involve subtractions by $n^2/4$. See Table
\ref{n-isteprightfromleft} for the example of F3.  This means that, once the
left side is build, we swap the first half columns with last half
columns and then add and subtract $n^2/4$, appropriately, to get the
right side. For example, the first eight columns of the square F3 is
swapped with the last eight columns of the left side of F3 and $64$ is
added and subtracted as shown in Table \ref{table64fromleft}.

\begin{table}{\scriptsize
    \caption{Constructing the Left side of the Franklin square F3.} 
    \label{tableleftf3steps} 
    \[ \begin{array}{cccccccccccccccccc}
      \mbox{Step 1} & \mbox{Step 2}  \\  
      \begin{array}{|c|c|c|c|c|c|c|c|c|c|c|c|c|c|c|c|} \hline
        - & - & - & - & 8 & - & - & - \\ \hline
        - & - & - & 7 & - & - & - & - \\ \hline
        - & - & - & - & 6 & - & - & - \\ \hline
        - & - & - & 5 & - & - & - & - \\ \hline
        - & - & - & - & 9 & - & - & - \\ \hline
        - & - & - & 10 & - & - & - & - \\ \hline
        - & - & - & - & 11 & - & - & - \\ \hline
        - & - & - & 12 & - & - & - & - \\ \hline
        - & - & - & - & 13 & - & - & - \\ \hline
        - & - & - & 14 & - & - & - & - \\ \hline
        - & - & - & - & 15 & - & - & - \\ \hline
        - & - & - & 16 & - & - & - & - \\ \hline
        - & - & - & - & 4 & - & - & - \\ \hline
        - & - & - & 3 & - & - & - & - \\ \hline
        - & - & - & - & 2 & - & - & - \\ \hline
        - & - & - & 1 & - & - & - & - \\ \hline
      \end{array}   & 
      \begin{array}{|c|c|c|c|c|c|c|c|c|c|c|c|c|c|c|c|} \hline
        - & - & - & - & 8 & 25 & - & - \\ \hline
        - & - & 26 & 7 & - & - & - & - \\ \hline
        - & - & - & - & 6 & 27& - & - \\ \hline
        - & - & 28 & 5 & - & - & - & - \\ \hline
        - & - & - & - & 9 & 24 & - & - \\ \hline
        - & - & 23 & 10 & - & - & - & - \\ \hline
        - & - & - & - & 11 & 22 & - & - \\ \hline
        - & - & 21 & 12 & - & - & - & - \\ \hline
        - & - & - & - & 13 & 20 & - & - \\ \hline
        - & - & 19 & 14 & - & - & - & - \\ \hline
        - & - & - & - & 15 & 18 & - & - \\ \hline
        - & - & 17 & 16 & - & - & - & - \\ \hline
        - & - & - & - & 4 & 29 & - & - \\ \hline
        - & - & 30 & 3 & - & - & - & - \\ \hline
        - & - & - & - & 2 & 31 & - & - \\ \hline
        - & - & 32 & 1 & - & - & - & - \\ \hline
      \end{array}   \\ \\
      \mbox{Step 3} & \mbox{Step 4}  \\  
      \begin{array}{|c|c|c|c|c|c|c|c|c|c|c|c|c|c|c|c|} \hline
        - & - & - & - & 8 & 25 & 40 & - \\ \hline
        - & 39 & 26 & 7 & - & - & - & - \\ \hline
        - & - & - & - & 6 & 27& 38 & - \\ \hline
        - & 37 & 28 & 5 & - & - & - & - \\ \hline
        - & - & - & - & 9 & 24 & 41 & - \\ \hline
        - & 42 & 23 & 10 & - & - & - & - \\ \hline
        - & - & - & - & 11 & 22 & 43 & - \\ \hline
        - & 44 & 21 & 12 & - & - & - & - \\ \hline
        - & - & - & - & 13 & 20 & 45 & - \\ \hline
        - & 46 & 19 & 14 & - & - & - & - \\ \hline
        - & - & - & - & 15 & 18 & 47 & - \\ \hline
        - & 48 & 17 & 16 & - & - & - & - \\ \hline
        - & - & - & - & 4 & 29 & 36 & - \\ \hline
        - & 35& 30 & 3 & - & - & - & - \\ \hline
        - & - & - & - & 2 & 31 & 34 & - \\ \hline
        - & 33 & 32 & 1 & - & - & - & - \\ \hline
      \end{array} &  
      \begin{array}{|c|c|c|c|c|c|c|c|c|c|c|c|c|c|c|c|} \hline
        - & - & - & - & 8 & 25 & 40 & 57 \\ \hline
        58 & 39 & 26 & 7 & - & - & - & - \\ \hline
        - & - & - & - & 6 & 27& 38 & 59 \\ \hline
        60 & 37 & 28 & 5 & - & - & - & - \\ \hline
        - & - & - & - & 9 & 24 & 41 & 56 \\ \hline
        55 & 42 & 23 & 10 & - & - & - & - \\ \hline
        - & - & - & - & 11 & 22 & 43 & 54 \\ \hline
        53 & 44 & 21 & 12 & - & - & - & - \\ \hline
        - & - & - & - & 13 & 20 & 45 & 52 \\ \hline
        51 & 46 & 19 & 14 & - & - & - & - \\ \hline
        - & - & - & - & 15 & 18 & 47 & 50 \\ \hline
        49 & 48 & 17 & 16 & - & - & - & - \\ \hline
        - & - & - & - & 4 & 29 & 36 & 61 \\ \hline
        62 & 35& 30 & 3 & - & - & - & - \\ \hline
        - & - & - & - & 2 & 31 & 34 & 63 \\ \hline
        64 & 33 & 32 & 1 & - & - & - & - \\ \hline
      \end{array}   
    \end{array}
    \]
  }
\end{table}
\begin{table}{\scriptsize
    \caption{Constructing the Right side of the Franklin square F3.} 
    \label{tabldistributeright} 
    \[
    \begin{array}{cccccccccccccc}
      \mbox{Step 1} & \mbox{Step 2}  \\ 
      \begin{array}{|c|c|c|c|c|c|c|c|c|c|c|c|c|c|c|c|} \hline
        72 &  & - & -  & - & - & - & - \\ \hline 
        - & - & - & - & - & - & - & 71 \\ \hline 
        70 & - & - & - & - & - & - & - \\ \hline
        - & - & - & - & - & - & - & 69 \\ \hline
        73 & - & -& -& - & - & - & - \\ \hline
        - & - & - & - & - & - & -& 74 \\ \hline
        75 & - & -& - & - & - & - & - \\ \hline
        - & - & - & - & - & - & - & 76 \\ \hline
        77 & - & - & - & - & - & - &  \\ \hline
        - & - & - & - & - & - & - & 78 \\ \hline
        79 & - & - & - & - & - & - & - \\ \hline
        - & - & - & - & - & - & - & 80 \\ \hline
        68 & - & -& -  & - & - & - & - \\ \hline
        - & - & - & - & - &- & - & 67 \\ \hline
        66 & - & - & -  & - & - & - & - \\ \hline
        - & - & - & - & - & - & - & 65 \\ \hline
      \end{array}  &
      \begin{array}{|c|c|c|c|c|c|c|c|c|c|c|c|c|c|c|c|} \hline
        72 & 89 & - & -  & - & - & - & - \\ \hline 
        - & - & - & - & - & - & 90 & 71 \\ \hline 
        70 & 91 & - & - & - & - & - & - \\ \hline
        - & - & - & - & - & - & 92 & 69 \\ \hline
        73 & 88 & -& -& - & - & - & - \\ \hline
        - & - & - & - & - & - & 87& 74 \\ \hline
        75 & 86 & -& - & - & - & - & - \\ \hline
        - & - & - & - & - & - & 85 & 76 \\ \hline
        77 & 84 & - & - & - & - & - &  \\ \hline
        - & - & - & - & - & - & 83 & 78 \\ \hline
        79 & 82 & - & - & - & - & - & - \\ \hline
        - & - & - & - & - & - & 81 & 80 \\ \hline
        68 & 93 & -& -  & - & - & - & - \\ \hline
        - & - & - & - & - &- & 94 & 67 \\ \hline
        66 & 95 & - & -  & - & - & - & - \\ \hline
        - & - & - & - & - & - & 96 & 65 \\ \hline
      \end{array}  \\ \\
      \mbox{Step 3} & \mbox{Step 4}  \\ 
      \begin{array}{|c|c|c|c|c|c|c|c|c|c|c|c|c|c|c|c|} \hline
        72 & 89 & 104 & -  & - & - & - & - \\ \hline 
        - & - & - & - & - & 103 & 90 & 71 \\ \hline 
        70 & 91 & 102 & - & - & - & - & - \\ \hline
        - & - & - & - & - & 101 & 92 & 69 \\ \hline
        73 & 88 & 105& -& - & - & - & - \\ \hline
        - & - & - & - & - & 106 & 87& 74 \\ \hline
        75 & 86 & 107 & - & - & - & - & - \\ \hline
        - & - & - & - & - & 108 & 85 & 76 \\ \hline
        77 & 84 & 109 & - & - & - & - &  \\ \hline
        - & - & - & - & - & 110 & 83 & 78 \\ \hline
        79 & 82 & 111 & - & - & - & - & - \\ \hline
        - & - & - & - & - & 112 & 81 & 80 \\ \hline
        68 & 93 & 100& -  & - & - & - & - \\ \hline
        - & - & - & - & - & 99 & 94 & 67 \\ \hline
        66 & 95 & 98 & -  & - & - & - & - \\ \hline
        - & - & - & - & - & 97 & 96 & 65 \\ \hline
      \end{array} &
      \begin{array}{|c|c|c|c|c|c|c|c|c|c|c|c|c|c|c|c|} \hline
        72 & 89 & 104 & 121 & - & - & - & - \\ \hline 
        - & - & - & - & 122 & 103 & 90 & 71 \\ \hline 
        70 & 91 & 102 & 123 & - & - & - & - \\ \hline
        - & - & - & - & 124 & 101 & 92 & 69 \\ \hline
        
        73 & 88 & 105& 120& - & - & - & - \\ \hline
        - & - & - & - & 119 & 106 & 87& 74 \\ \hline
        75 & 86 & 107 & 118 & - & - & - & - \\ \hline
        - & - & - & - & 117 & 108 & 85 & 76 \\ \hline
        77 & 84 & 109 & 116 & - & - & - &  \\ \hline
        - & - & - & - & 115 & 110 & 83 & 78 \\ \hline
        79 & 82 & 111 & 114 & - & - & - & - \\ \hline
        - & - & - & - & 113 & 112 & 81 & 80 \\ \hline
        
        68 & 93 & 100&  125  & - & - & - & - \\ \hline
        - & - & - & - & 126 & 99 & 94 & 67 \\ \hline
        66 & 95 & 98 & 127  & - & - & - & - \\ \hline
        - & - & - & - & 128& 97 & 96 & 65 \\ \hline
      \end{array}
    \end{array}
    \]
  }
\end{table}

\begin{table}{\scriptsize
    \caption{The left side of the Franklin square F3.} 
    \label{tableleftf3fillingN} 
    \[
    \begin{array}{|c|c|c|c|c|c|c|c|c|c|c|c|c|c|c|c|} \hline
      N-57 & N-40 & N-25 & N-8 & 8 & 25     & 40 & 57 \\ \hline
      58   & 39   & 26   & 7   & N-7 & N-26 & N-39 & N-58 \\ \hline
      N-59 & N-38 & N-27 & N-6  & 6    & 27   & 38   & 59 \\ \hline
      60   & 37   & 28   & 5    & N-5  & N-28 & N-37 & N-60 \\ \hline
      N-56 & N-41 & N-24 & N-9  & 9    & 24   & 41   & 56 \\ \hline
      55   & 42   & 23   & 10   & N-10 & N-23 & N-42 & N-55 \\ \hline
      N-54 & N-43 & N-22 & N-11 & 11   & 22   & 43   & 54 \\ \hline
      53   & 44   & 21   & 12  & N-12  & N-21 & N-44 & N-53 \\ \hline 
      N-52 & N-45 & N-20 & N-13 & 13   & 20   & 45   & 52 \\ \hline
      51   & 46   & 19   & 14  & N-14  & N-19 & N-46 & N-51 \\ \hline 
      N-50 & N-47 & N-18 & N-15 & 15   & 18   & 47   & 50 \\ \hline
      49   & 48   & 17   & 16  & N-16  & N-17 & N-48 & N-49 \\ \hline 
      N-61 & N-36 & N-29 & N-4 & 4     & 29   & 36   & 61 \\ \hline
      62   & 35   & 30   & 3   & N-3   & N-30 & N-35 & N-62 \\ \hline 
      N-63 & N-34 & N-31 & N-2 & 2     & 31   & 34   & 63 \\ \hline
      64   & 33   & 32   & 1   & N-1   & N-32 & N-33 & N-64 \\ \hline
    \end{array} \]
  }
\end{table}

\begin{table}{\scriptsize
    \caption{The right side of the Franklin square F3.} 
    \label{rightfranklintablefillingwithN}
    \[
    \begin{array}{|c|c|c|c|c|c|c|c|c|c|c|c|c|c|c|c|} \hline
      72   & 89   & 104   & 121   &N-121  & N-104 & N-89 & N-72 \\ \hline
      N-71 & N-90 & N-103 & N-122 &122    &103    &90    &71 \\ \hline
      70   & 91   & 102   & 123   &N-123  & N-102 & N-91 & N-70 \\ \hline
      N-69 & N-92 & N-101 & N-124 &124    &101    &92    &69 \\ \hline
      73   & 88   & 105   & 120   &N-120  & N-105 & N-88 & N-73 \\ \hline
      N-74 & N-87 & N-106 & N-119 &119    &106    &87    &74 \\ \hline
      75   & 86   & 107   & 118   &N-118  & N-107 & N-86 & N-75 \\ \hline
      N-76 & N-85 & N-108 & N-117 &117    &108    &85    &76 \\ \hline
      77   & 84   & 109   & 116   &N-116  & N-109 & N-84 & N-77 \\ \hline
      N-78 & N-83 & N-110 & N-115 &115    &110    &83    &78 \\ \hline
      79   & 82   & 111   & 114   &N-114  & N-111 & N-82 & N-79 \\ \hline
      N-80 & N-81 & N-112 & N-113 &113    &112    &81    &80 \\ \hline
      68   & 93   & 100   & 125   &N-125  & N-100 & N-93 & N-68 \\ \hline
      N-67 & N-94 & N-99  & N-126 &126    &99     &94    &67 \\ \hline
      66   & 95   & 98   & 127   &N-127  & N-98 & N-95 & N-66 \\ \hline
      N-65 & N-96 & N-97  & N-128 &128    &97     &96    &65 \\ \hline

    \end{array}
    \]
  }
\end{table}

\begin{table}{\scriptsize
    \caption{Construction of $8 \times 8$ Franklin Square F1.} 
    \label{franklinstep1to4} 

    \[ \begin{array}{ccccccccccccccccccccccccccc}
      \mbox{Step 1} &&&& \mbox{Step 2} \\ \\

      \begin{array}{cc}
        \begin{array}{|c|c|c|c|} \hline
          - & - & 4 & - \\ \hline
          - & 3 & - & -  \\ \hline
        \end{array} & 

        \begin{array}{|c|c|c|c|}  \hline
          - & - & - & - \\ \hline
          - & - & - & -  \\ \hline
        \end{array} \\ \\

        \begin{array}{|c|c|c|c|} \hline
          - & - & 5 & - \\ \hline
          - & 6 & - & -  \\ \hline
          - & - & 7 & -  \\ \hline
          - & 8 & - & -  \\ \hline
        \end{array} & 
        \begin{array}{|c|c|c|c|} \hline
          - & - & - & - \\ \hline
          - & - & - & -  \\ \hline
          - & - & - & -  \\ \hline
          - & - & - & -  \\ \hline
        \end{array}  \\ \\

        \begin{array}{|c|c|c|c|} \hline
          - & - & 2 & - \\ \hline
          - & 1 & - & -  \\ \hline
        \end{array} & 

        \begin{array}{|c|c|c|c|}  \hline
          - & - & - & - \\ \hline
          - & - & - & -  \\ \hline
        \end{array}
      \end{array}

      &&&&

      \begin{array}{cc}
        \begin{array}{|c|c|c|c|} \hline
          - & - & 4 & 13 \\ \hline
          14 & 3 & - & -  \\ \hline
        \end{array} & 

        \begin{array}{|c|c|c|c|}  \hline
          - & - & - & - \\ \hline
          - & - & - & -  \\ \hline
        \end{array} \\ \\

        \begin{array}{|c|c|c|c|} \hline
          - & - & 5 & 12 \\ \hline
          11 & 6 & - & -  \\ \hline
          - & - & 7 & 10  \\ \hline
          9 & 8 & - & -  \\ \hline
        \end{array} & 
        \begin{array}{|c|c|c|c|} \hline
          - & - & - & - \\ \hline
          - & - & - & -  \\ \hline
          - & - & - & -  \\ \hline
          - & - & - & -  \\ \hline
        \end{array}  \\ \\

        \begin{array}{|c|c|c|c|} \hline
          - & - & 2 & 15\\ \hline
          16 & 1 & - & -  \\ \hline
        \end{array} & 

        \begin{array}{|c|c|c|c|}  \hline
          - & - & - & - \\ \hline
          - & - & - & -  \\ \hline
        \end{array}

      \end{array} \\ \\ 
      \mbox{Step 3} &&&& \mbox{Step 4} \\ \\
      \begin{array}{cc}
        \begin{array}{|c|c|c|c|} \hline
          - & - & 4 & 13 \\ \hline
          14 & 3 & - & -  \\ \hline
        \end{array} & 

        \begin{array}{|c|c|c|c|}  \hline
          20 & - & - & - \\ \hline
          - & - & - & 19  \\ \hline
        \end{array} \\ \\

        \begin{array}{|c|c|c|c|} \hline
          - & - & 5 & 12 \\ \hline
          11 & 6 & - & -  \\ \hline
          - & - & 7 & 10  \\ \hline
          9 & 8 & - & -  \\ \hline
        \end{array} & 
        \begin{array}{|c|c|c|c|} \hline
          21 & - & - & - \\ \hline
          - & - & - & 22  \\ \hline
          23 & - & - & -  \\ \hline
          - & - & - & 24  \\ \hline
        \end{array}  \\ \\

        \begin{array}{|c|c|c|c|} \hline
          - & - & 2 & 15\\ \hline
          16 & 1 & - & -  \\ \hline
        \end{array} & 

        \begin{array}{|c|c|c|c|}  \hline
          18 & - & - & - \\ \hline
          - & - & - & 17  \\ \hline
        \end{array}
      \end{array} &&&&
      \begin{array}{cc}
        \begin{array}{|c|c|c|c|} \hline
          - & - & 4 & 13 \\ \hline
          14 & 3 & - & -  \\ \hline
        \end{array} & 

        \begin{array}{|c|c|c|c|}  \hline
          20 & 29 & - & - \\ \hline
          - & - & 30 & 19  \\ \hline
        \end{array} \\ \\

        \begin{array}{|c|c|c|c|} \hline
          - & - & 5 & 12 \\ \hline
          11 & 6 & - & -  \\ \hline
          - & - & 7 & 10  \\ \hline
          9 & 8 & - & -  \\ \hline
        \end{array} & 
        \begin{array}{|c|c|c|c|} \hline
          21 & 28 & - & - \\ \hline
          - & - & 27 & 22  \\ \hline
          23 & 26 & - & -  \\ \hline
          - & - & 25 & 24  \\ \hline
        \end{array}  \\ \\

        \begin{array}{|c|c|c|c|} \hline
          - & - & 2 & 15\\ \hline
          16 & 1 & - & -  \\ \hline
        \end{array} & 

        \begin{array}{|c|c|c|c|}  \hline
          18 & 31 & - & - \\ \hline
          - & - & 32 & 17  \\ \hline
        \end{array}
      \end{array}
    \end{array}
    \]
  }
\end{table}
\begin{table}{\scriptsize
    \caption{Final step in construction of F1} 
    \label{franklinstepfinal} 
    \[
    \begin{array}{cc}
      \begin{array}{|c|c|c|c|} \hline
        N-13 & N-4& 4 & 13 \\ \hline
        14 & 3 & N-3 & N-14  \\ \hline
      \end{array} & 

      \begin{array}{|c|c|c|c|}  \hline
        20 & 29 & N-29 & N-20 \\ \hline
        N-19 & N-30 & 30 & 19  \\ \hline
      \end{array} \\ \\

      \begin{array}{|c|c|c|c|} \hline
        N-12 & N-5 & 5 & 12 \\ \hline
        11 & 6 & N-6 & N-11  \\ \hline
        N-10 & N-7 & 7 & 10  \\ \hline
        9 & 8 & N-8 & N-9  \\ \hline
      \end{array} & 
      \begin{array}{|c|c|c|c|} \hline
        21 & 28 & N-28 & N-21 \\ \hline
        N-22 & N-27 & 27 & 22  \\ \hline
        23 & 26 & N-26 & N-23  \\ \hline
        N-24 &N -25 & 25 & 24  \\ \hline
      \end{array}  \\ \\

      \begin{array}{|c|c|c|c|} \hline
        N-15 & N-2 & 2 & 15\\ \hline
        16 & 1 & N-1 & N-16  \\ \hline
      \end{array} & 

      \begin{array}{|c|c|c|c|}  \hline
        18 & 31 & N-31 & N-18 \\ \hline
        N-17 & N-32 & 32 & 17  \\ \hline
      \end{array}
    \end{array}
    \]
  }
\end{table}

\begin{table}{\scriptsize
    \caption{Step 4 of  Table \ref{tabldistributeright} in
      construction of right side of F3 derived, alternately, by
      swapping columns of partially built left side, and adding 64.}
    \label{firststeprightfromleft}
    \[ \begin{array}{ccccccccccccccccc}
      \begin{array}{|c|c|c|c|c|c|c|c|c|c|c|c|c|c|c|c|} \hline
        8 & 25 & 40 & 57 & - & - & - & - \\ \hline 
        - & - & - & - & 58 & 39 & 26 & 7 \\ \hline 
        6 & 27 & 38 & 59 & - & - & - & - \\ \hline
        - & - & - & - & 60 & 37 & 28 & 5 \\ \hline
        
        9 & 24 & 41& 56& - & - & - & - \\ \hline
        - & - & - & - & 55 & 42 & 23& 10 \\ \hline
        11 & 22 & 43 & 54 & - & - & - & - \\ \hline
        - & - & - & - & 53 & 44 & 21 & 12 \\ \hline
        13 & 20 & 45 & 52 & - & - & - &  \\ \hline
        - & - & - & - & 51 & 46 & 19 & 14 \\ \hline
        15 & 18 & 47 & 50 & - & - & - & - \\ \hline
        - & - & - & - & 49 & 48& 17 & 16 \\ \hline
        
        4 & 29 & 36 &  61  & - & - & - & - \\ \hline
        - & - & - & - & 62 & 35 &  30 & 3 \\ \hline
        2 & 31 & 34 & 63  & - & - & - & - \\ \hline
        - & - & - & - & 64& 33 & 32 & 1 \\ \hline
      \end{array}&+&  
      \begin{array}{|c|c|c|c|c|c|c|c|c|c|c|c|c|c|c|c|} \hline
        64 & 64 & 64 & 64 & 0 & 0 & 0 & 0 \\ \hline
        0 & 0 & 0 & 0 & 64 & 64 & 64 & 64  \\ \hline
        64 & 64 & 64 & 64 & 0 & 0 & 0 & 0 \\ \hline
        0 & 0 & 0 & 0 & 64 & 64 & 64 & 64  \\ \hline
        64 & 64 & 64 & 64 & 0 & 0 & 0 & 0 \\ \hline
        0 & 0 & 0 & 0 & 64 & 64 & 64 & 64  \\ \hline
        64 & 64 & 64 & 64 & 0 & 0 & 0 & 0 \\ \hline
        0 & 0 & 0 & 0 & 64 & 64 & 64 & 64  \\ \hline
        64 & 64 & 64 & 64 & 0 & 0 & 0 & 0 \\ \hline
        0 & 0 & 0 & 0 & 64 & 64 & 64 & 64  \\ \hline
        64 & 64 & 64 & 64 & 0 & 0 & 0 & 0 \\ \hline
        0 & 0 & 0 & 0 & 64 & 64 & 64 & 64  \\ \hline
        64 & 64 & 64 & 64 & 0 & 0 & 0 & 0 \\ \hline
        0 & 0 & 0 & 0 & 64 & 64 & 64 & 64  \\ \hline
        64 & 64 & 64 & 64 & 0 & 0 & 0 & 0 \\ \hline
        0 & 0 & 0 & 0 & 64 & 64 & 64 & 64  \\ \hline
      \end{array} &=& 
      \begin{array}{|c|c|c|c|c|c|c|c|c|c|c|c|c|c|c|c|} \hline
        72 & 89 & 104 & 121 & - & - & - & - \\ \hline 
        - & - & - & - & 122 & 103 & 90 & 71 \\ \hline 
        70 & 91 & 102 & 123 & - & - & - & - \\ \hline
        - & - & - & - & 124 & 101 & 92 & 69 \\ \hline
        
        73 & 88 & 105& 120& - & - & - & - \\ \hline
        - & - & - & - & 119 & 106 & 87& 74 \\ \hline
        75 & 86 & 107 & 118 & - & - & - & - \\ \hline
        - & - & - & - & 117 & 108 & 85 & 76 \\ \hline
        77 & 84 & 109 & 116 & - & - & - &  \\ \hline
        - & - & - & - & 115 & 110 & 83 & 78 \\ \hline
        79 & 82 & 111 & 114 & - & - & - & - \\ \hline
        - & - & - & - & 113 & 112 & 81 & 80 \\ \hline
        
        68 & 93 & 100&  125  & - & - & - & - \\ \hline
        - & - & - & - & 126 & 99 & 94 & 67 \\ \hline
        66 & 95 & 98 & 127  & - & - & - & - \\ \hline
        - & - & - & - & 128& 97 & 96 & 65 \\ \hline
      \end{array}

    \end{array}
    \]
} \end{table}

\begin{table}{\scriptsize
    \caption{Constructing the right side of F3.}
    \label{n-isteprightfromleft}
    \[ \begin{array}{ccccccccccccccccc}
      \begin{array}{|c|c|c|c|c|c|c|c|c|c|c|c|c|c|c|c|} \hline
        8 & 25 & 40 & 57 &  N-57 & N-40 & N-25 & N-8 \\ \hline 
        N-7 & N-26 & N-39 & N-58 & 58 & 39 & 26 & 7 \\ \hline 
        6 & 27 & 38 & 59 &  N-59 & N-38 & N-27 & N-6 \\ \hline
        N-5  & N-28 & N-37 & N-60 & 60 & 37 & 28 & 5 \\ \hline
        9 & 24 & 41& 56&  N-56 & N-41 & N-24 & N-9 \\ \hline
        N-10 & N-23 & N-42 & N-55 & 55 & 42 & 23& 10 \\ \hline
        11 & 22 & 43 & 54 & N-54 & N-43 & N-22 & N-11 \\ \hline
        N-12  & N-21 & N-44 & N-53  & 53 & 44 & 21 & 12 \\ \hline
        13 & 20 & 45 & 52 &  N-52 & N-45 & N-20 & N-13 \\ \hline
        N-14  & N-19 & N-46 & N-51 & 51 & 46 & 19 & 14 \\ \hline
        15 & 18 & 47 & 50 &  N-50 & N-47 & N-18 & N-15 \\ \hline
        N-16  & N-17 & N-48 & N-49  & 49 & 48& 17 & 16 \\ \hline
        4 & 29 & 36 &  61  & N-61 & N-36 & N-29 & N-4 \\ \hline
        N-3   & N-30 & N-35 & N-62 & 62 & 35 &  30 & 3 \\ \hline
        2 & 31 & 34 & 63  &  N-63 & N-34 & N-31 & N-2 \\ \hline
        N-1   & N-32 & N-33 & N-64 & 64& 33 & 32 & 1 \\ \hline
      \end{array}&+&  
      \begin{array}{|c|c|c|c|c|c|c|c|c|c|c|c|c|c|c|c|} \hline
        64 & 64 & 64 & 64 & -64 & -64 & -64 & -64 \\ \hline
        -64 & -64 & -64 & -64 & 64 & 64 & 64 & 64   \\ \hline
        64 & 64 & 64 & 64 & -64 & -64 & -64 & -64 \\ \hline
        -64 & -64 & -64 & -64 & 64 & 64 & 64 & 64   \\ \hline
        64 & 64 & 64 & 64 & -64 & -64 & -64 & -64 \\ \hline
        -64 & -64 & -64 & -64 & 64 & 64 & 64 & 64   \\ \hline
        64 & 64 & 64 & 64 & -64 & -64 & -64 & -64 \\ \hline
        -64 & -64 & -64 & -64 & 64 & 64 & 64 & 64   \\ \hline
        64 & 64 & 64 & 64 & -64 & -64 & -64 & -64 \\ \hline
        -64 & -64 & -64 & -64 & 64 & 64 & 64 & 64   \\ \hline
        64 & 64 & 64 & 64 & -64 & -64 & -64 & -64 \\ \hline
        -64 & -64 & -64 & -64 & 64 & 64 & 64 & 64   \\ \hline
        64 & 64 & 64 & 64 & -64 & -64 & -64 & -64 \\ \hline
        -64 & -64 & -64 & -64 & 64 & 64 & 64 & 64   \\ \hline
        64 & 64 & 64 & 64 & -64 & -64 & -64 & -64 \\ \hline
        -64 & -64 & -64 & -64 & 64 & 64 & 64 & 64   \\ \hline
        
      \end{array}
    \end{array}
    \]
} \end{table}

\begin{table}
  {\scriptsize
    \caption{Deriving the right side of the Franklin square F3 from the left side.} 
    \label{table64fromleft} 
    \[ \begin{array}{ccccc} 
      \begin{array}{|c|c|c|c|c|c|c|c|c|c|c|c|c|c|c|c|} \hline
        8 & 25 & 40 & 57 &  200 & 217 & 232 & 249  \\ \hline
        250 & 231 & 218 & 199 &  58 & 39 & 26 & 7   \\ \hline
        6 & 27 & 38 & 59 & 198 & 219 & 230 & 251   \\ \hline
        252& 229 & 220 & 197 &   60 & 37 & 28 &  5   \\ \hline
        9  & 24  & 41  & 56 &   201 & 216 & 233 & 248  \\ \hline
        247 & 234 & 215 & 202 & 55 & 42 & 23 & 10   \\ \hline
        11 & 22 & 43 & 54 & 203 & 214 & 235 & 246   \\ \hline
        245 & 236 & 213 & 204 & 53 & 44 & 21 & 12  \\ \hline
        13 & 20 & 45 &  52 & 205 & 212 & 237 & 244 \\ \hline 
        243 & 238 & 211 & 206 & 51 & 46 & 19 &  14  \\ \hline
        15 & 18 &  47 &  50 & 207 & 210 & 239 & 242  \\ \hline
        241 & 240 & 209 & 208 & 49 &  48 & 17 & 16    \\ \hline
        4 & 29 &  36 & 61 & 196 &221 & 228 & 253  \\ \hline
        254 & 227 & 222 & 195 &  62  & 35 &  30 &  3  \\ \hline
        2 & 31 &  34 &  63  & 194 & 223 & 226 & 255 \\ \hline
        256 & 225 & 224 &  193 & 64  & 33  & 32 & 1  \\ \hline
      \end{array}
      & + &
      \begin{array}{|c|c|c|c|c|c|c|c|c|c|c|c|c|c|c|c|} \hline
        64 & 64 & 64 & 64 & -64 & -64 & -64 & -64 \\ \hline
        -64 & -64 & -64 & -64 & 64 & 64 & 64 & 64   \\ \hline
        64 & 64 & 64 & 64 & -64 & -64 & -64 & -64 \\ \hline
        -64 & -64 & -64 & -64 & 64 & 64 & 64 & 64   \\ \hline
        64 & 64 & 64 & 64 & -64 & -64 & -64 & -64 \\ \hline
        -64 & -64 & -64 & -64 & 64 & 64 & 64 & 64   \\ \hline
        64 & 64 & 64 & 64 & -64 & -64 & -64 & -64 \\ \hline
        -64 & -64 & -64 & -64 & 64 & 64 & 64 & 64   \\ \hline
        64 & 64 & 64 & 64 & -64 & -64 & -64 & -64 \\ \hline
        -64 & -64 & -64 & -64 & 64 & 64 & 64 & 64   \\ \hline
        64 & 64 & 64 & 64 & -64 & -64 & -64 & -64 \\ \hline
        -64 & -64 & -64 & -64 & 64 & 64 & 64 & 64   \\ \hline
        64 & 64 & 64 & 64 & -64 & -64 & -64 & -64 \\ \hline
        -64 & -64 & -64 & -64 & 64 & 64 & 64 & 64   \\ \hline
        64 & 64 & 64 & 64 & -64 & -64 & -64 & -64 \\ \hline
        -64 & -64 & -64 & -64 & 64 & 64 & 64 & 64   \\ \hline

      \end{array}
    \end{array}
    \]
  }
\end{table}

Summarizing, we derive the following algorithm for constructing Franklin squares.
\begin{algo}(Constructing $n \times n$ Franklin squares.) \label{franklincosntructalgo1}
  \begin{enumerate}
  \item Partial filling of the left side. \label{partialstep}

    Start with the two middle columns of left side, and then work
    outwards two equidistant columns at a time. Fill $n$ numbers at
    every step. We follow the sequence of operations, as described
    above, according to the parity of the distance of the columns.

  \item \label{subfronNfrankalgo} Subtractions from $N$ to complete the left side.

    Each empty cell of the left side is filled with the number
    $N-i$ where $i$ is the entry in the same row in the equidistant
    column. This gives us the left side of the Franklin square.

  \item \label{constructrightfromleftstep} Constructing the right side of the Franklin square from the left side.

    Swap the first $n/4$ columns with the last $n/4$ columns of the
    left side to build the right side of the Franklin square. For odd
    rows of this modified square, $n^2/4$ is added to every entry in the
    first half of the row, and $n^2/4$ is subtracted from every entry of
    the second half. For even rows, $n^2/4$ is subtracted  from every entry in the
    first half of the row, and $n^2/4$ is added to every entry of
    the second half.  This gives us the right side of the Franklin square.

  \end{enumerate}

\end{algo} 

We proceed to show that Algorithm \ref{franklincosntructalgo1} produces a Franklin square. We
start with the following lemma which describes many properties of the
Franklin square, constructed using Algorithm
\ref{franklincosntructalgo1}.

\begin{lemma} \label{prooffranklinalgolemma}
  Let $a_{i,j}$ denote the entries of a $n \times n$ Franklin square, where $i=1,2, \dots, n$ and $j=1,2, \dots, n$. 
  \begin{enumerate}

  \item  \label{consecuticeincolmite} Pair of entries of adjacent rows in a column add to $N \pm 1$
    except for the rows $n/4$ and $3n/4$, as follows.

    Consider a row  $i \in \{1,2, \dots, n\} \setminus
    \{n/4,  3n/4, n\}$.  In  the  top and  bottom  parts of  the
    square, that is, when $i \leq n/4$ or $n/2 < i <n$, we have
    \[
    a_{i,j}+a_{i+1,j} = \left \{
    \begin{array}{llllllllll}
      N+1, & \mbox{if $j$ is odd, and  $i$ is odd,} \\
      N-1, & \mbox{if $j$ is odd, and  $i$ is even,} \\
      N-1, & \mbox{if $j$ is even, and  $i$ is odd,} \\
      N+1, & \mbox{if $j$ is even, and  $i$ is even.} \\
    \end{array}
    \right .
    \]

    For the middle part, the situation is exactly opposite. That is, when
    $n/4 < i \leq n/2$, we have
    \[
    a_{i,j}+a_{i+1,j} =  \left \{ \begin{array}{llllllllll}
      N-1, & \mbox{if $j$ is odd, and  $i$ is odd,} \\
      N+1, & \mbox{if $j$ is odd, and  $i$ is even,} \\
      N+1, & \mbox{if $j$ is even, and  $i$ is odd,} \\
      N-1, & \mbox{if $j$ is even, and  $i$ is even.} \\
    \end{array}
    \right.
    \]
    Finally, we consider rows $n/4$ and $3n/4$.
    \[
    a_{n/4,j}+a_{n/4+1,j} =  \left \{ \begin{array}{llllllllll}
      N+n/4,  & \mbox{if $j$ is odd,} \\
      N-n/4,  & \mbox{if $j$ is even.} \\
    \end{array}
    \right.
    \]

    \[
    a_{3n/4,j}+a_{3n/4+1,j}
    =  \left \{ \begin{array}{llllllllll}
      N-3n/4,   & \mbox{if $j$ is odd,} \\
      N+3n/4,  & \mbox{if $j$ is even.} \\
    \end{array}
    \right.
    \]

  \item \label{colequi} Consider a row $i \leq n/2$ and let $m_i=n/2-1-2(i-1)$. Equidistant
    entries across the Horizontal axis add to either $N+m_i$ or $N
    -m_i$ as follows.
    \[a_{i,j} + a_{n-i+1,j} =  \left \{
    \begin{array}{lllllllll}
      N +  m_i, &  \mbox{ if $i$ is odd and $j$ is odd.} \\
      N -  m_i, &  \mbox{ if $i$ is odd and $j$ is even.} \\
      N -  m_i, &  \mbox{ if $i$ is even and $j$ is odd.} \\
      N +  m_i, &  \mbox{ if $i$ is even and $j$ is even.}
    \end{array}
    \right .  \]

  \item \label{rowequi} Let $m=n^2/2+1$. Equidistant entries across the vertical axis add to either $m$ or $2N-m$ as follows.

    If $i$ is odd 

    \[a_{i,j} + a_{i,n-j+1} =  \left \{
    \begin{array}{lllllllll}
      2N -  m,  &  \mbox{if } j \leq n/4,  \\
      m, &  \mbox{if }  n/4< j < n/2. 
    \end{array}
    \right .  \]

    If $i$ is even, the exact opposite is true. That is,
    \[a_{i,j} + a_{i,n-j+1} =  \left \{
    \begin{array}{lllllllll}
      m,  &  \mbox{if } j \leq n/4,  \\
      2N-m, &  \mbox{if }  n/4< j < n/2. 
    \end{array}
    \right .  \]

  \end{enumerate}
\end{lemma}
\begin{proof}
  The square inherits these properties by construction.
\end{proof}

\begin{prop} Algorithm \ref{franklincosntructalgo1} produces a Franklin square.

\end{prop}

\begin{proof}
  Let $a_{i,j}$ denote the entries of an $n \times n$ square
  constructed by Algorithm \ref{franklincosntructalgo1}. Let $N =
  n^2+1$ and $M$ denote the magic sum.

  \begin{enumerate}
  \item $2 \times 2$ sub-square sums add to $2N$ continuously.

    Consider a row $i \in \{1,2, \dots, n\} \setminus \{n/4, 3n/4, n\}$.
    By Part \ref{consecuticeincolmite} in Lemma
    \ref{prooffranklinalgolemma},
    \[
    \mbox{ if } a_{i,j} + a_{i+1,j} = N+1, \mbox{ then } a_{i,j+1} + a_{i+1,j+1} = N-1.
    \]
    On the other hand, 
    \[
    \mbox{ if } a_{i,j} + a_{i+1,j} = N-1, \mbox{ then } a_{i,j+1} + a_{i+1,j+1} = N+1.
    \]
    Consequently, for all $i \in \{1,2, \dots, n\} \setminus \{n/4, 3n/4,
    n\}$ and all $j$, \[a_{i,j}+a_{i+1,j} + a_{i+1,j}+ a_{i+1,j+1} = 2N.\]

    Now we consider the rows $n/4$ and $3n/4$. 
    By Part \ref{consecuticeincolmite} in Lemma
    \ref{prooffranklinalgolemma},
    \[ \begin{array}{lllllllll}
      \mbox{ if } a_{n/4,j} + a_{n/4+1,j} = N+n/4, \mbox{ then } a_{n/4,j+1} + a_{n/4,j+1} = N-n/4, \mbox{ and } \\
      \mbox{ if } a_{n/4,j} + a_{n/4+1,j} = N-n/4, \mbox{ then } a_{n/4,j+1} + a_{n/4,j+1} = N+n/4.
    \end{array} 
    \]
    Also,
    \[\begin{array}{lllllllll}
    \mbox{ if } a_{3n/4,j} + a_{3n/4+1,j} = N+3n/4, \mbox{ then } a_{3n/4,j+1} + a_{3n/4,j+1} = N-3n/4, \mbox{ and } \\
    \mbox{ if } a_{3n/4,j} + a_{3n/4+1,j} = N-3n/4, \mbox{ then } a_{3n/4,j+1} + a_{3n/4,j+1} = N+3n/4.
    \end{array} 
    \]

    Consequently, all the $2 \times 2$ sub-squares, within the Franklin
    square, add to $2N$. Next, we verify the continuity of this property.

    Part \ref{colequi} of Lemma \ref{prooffranklinalgolemma} implies
    \[a_{1,j}+a_{n,j} = \left \{ \begin{array}{llllllll} 
      N+m_1, & \mbox{ if $j$ is odd, } \\
      N-m_1, & \mbox{ if $j$ is even. }
    \end{array}
    \right.
    \]
    Consequently,  the $2 \times 2$ sub-squares formed
    by rows $1$ and $n$ add to $2N$.  Part \ref{rowequi} of Lemma
    \ref{prooffranklinalgolemma} implies 

    \[ a_{i,1}+a_{i,n} =  \left \{ \begin{array}{llllllll}
      2N-m, & \mbox{ if $i$ is odd, } \\
      m,  & \mbox{ if $i$ is even. }
    \end{array}
    \right.
    \]
    Thus, the $2 \times 2$ sub-squares formed by columns $1$ and $n$ add
    to $2N$. This proves that the continuity property of $2 \times 2$
    sub-squares hold for squares constructed by Algorithm
    \ref{franklincosntructalgo1}.

  \item  Half row and half column sums add to $M/2$. Row and column sums add to $M$.

    By Part \ref{consecuticeincolmite} of Lemma
    \ref{prooffranklinalgolemma}, quarter column sums are as follows.
    \[
    a_{1,j} + a_{2,j} + \cdots + a_{n/4,j} = \left \{ 
    \begin{array}{cccccc}  
      \frac{n}{8}(N+1), & \mbox{ when $j$ is odd, } \\ \\
      \frac{n}{8}(N-1),& \mbox{ when $j$ is even. } 
    \end{array} \right.
    \]

    \[
    a_{n/4+1,j} + a_{n/4+2,j} + \cdots + a_{n/2,j} = \left \{ 
    \begin{array}{cccccc}  
      \frac{n}{8}(N-1), & \mbox{ when $j$ is odd, } \\ \\
      \frac{n}{8}(N+1), & \mbox{ when $j$ is even. } 
    \end{array} \right.
    \]

    \[
    a_{n/2+1,j} + a_{n/2+2,j} + \cdots + a_{3n/4,j} = \left \{ 
    \begin{array}{cccccc}  
      \frac{n}{8}(N-1), & \mbox{ when $j$ is odd, } \\ \\
      \frac{n}{8}(N+1), & \mbox{ when $j$ is even. } 
    \end{array} \right.
    \]
    \[
    a_{3n/4+1,j} + a_{3n/4+2,j} + \cdots + a_{n,j} = \left \{ 
    \begin{array}{cccccc}  
      \frac{n}{8}(N+1), & \mbox{ when $j$ is odd, } \\ \\
      \frac{n}{8}(N-1),   & \mbox{ when $j$ is even. } 
    \end{array} \right.
    \]

    Consequently, for all $j$,

    \[\begin{array}{cccccc}  
    a_{1,j} + a_{2,j} + \cdots + a_{n/2,j} = \frac{n}{4}N. \\
    a_{n/2+1,j} + a_{n/2+1,j} + \cdots + a_{n,j} = \frac{n}{4}N.
    \end{array} 
    \]

    That is, all the half columns add to $(n/4)N$, which is half the
    magic sum. Therefore, all the columns add to the magic sum.  By
    construction, the way subtractions were done from $N$, (see Part
    \ref{subfronNfrankalgo} in Algorithm \ref{franklincosntructalgo1}),
    half rows add to $(n/4)N$. Hence full rows add to $(n/2)N$ which is
    the magic sum.

  \item Bend diagonals add to $M$.

    To prove that the left bend diagonals add to the magic sum, we add
    the entries, pairwise, where each pair is equidistant from the
    horizontal axis.  Let $1 \leq j < n$, then by Part \ref{colequi} of
    Lemma \ref{prooffranklinalgolemma}, if $1 \leq i < n/4$, and if,

    \[\begin{array}{lllllllllll}  
    a_{i,j} + a_{n-i+1,j} = N+m_i, \mbox{ then, } a_{i+1,j+1} + a_{n-i,j+1} = N+m_i, \mbox{ and if } \\
    a_{i,j} + a_{n-i+1,j} = N-m_i, \mbox{ then, } a_{i+1,j+1} + a_{n-i,j+1} = N-m_i. 
    \end{array}
    \]

    Observe that, if $n/4+1 \leq i < n/2$, then the signs for $m_i$ in
    the above sums are exactly opposite, by Part \ref{colequi} of Lemma
    \ref{prooffranklinalgolemma}.  Let $1 \leq j \leq n/2+1$, then, if $j$ is odd, the left
    bend diagonal sum starting with row $1$ and column $j$ adds as
    follows.

    \[ \begin{array}{llllllllllllll}
      \left [ (a_{1,j} + a_{n,j}) + (a_{2, j+1} + a_{n-1, j+1}) +  \cdots + (a_{\frac{n}{4}, j+\frac{n}{4}-1} + a_{\frac{3n}{4}+1, j+\frac{n}{4}-1}) \right ]\\ \\
      + \left [(a_{\frac{n}{4}+1, j+\frac{n}{4}}+a_{\frac{3n}{4}, j+\frac{n}{4}}) + \cdots + (a_{n/2,j+n/2-1}  + a_{n/2+1,j+n/2-1}) \right ]\\ \\
      =   \left [N+(\frac{n}{2} - 1 ) +  N -  (\frac{n}{2} - 3 ) -  N+(\frac{n}{2} - 5 ) + \cdots +  N-1\right ] \\ \\
      + \left [ N-(\frac{n}{2} - 1 ) +  N +  (\frac{n}{2} - 3 ) -  N-(\frac{n}{2} - 5 ) + \cdots +  N+1\right ]
      \\ \\  = \frac{n}{2}N =  M. 
    \end{array} 
    \]

    A similar argument gives us that the even bend diagonals also add to $M$.
    Consequently, the left bend diagonals, when $j=1,2, \dots, n/2+1$, add
    to the magic sum.

    Let  $n/2+1 < j \leq n$, the left bend diagonal sums are 
    \[ \begin{array}{llllllllllllll}
      (a_{1,j} + a_{n,j}) +  (a_{2,j} + a_{n-1, j+1})+ \cdots +  (a_{n-j+1, n} + a_{j,n}) \\ \\
      + (a_{n-j+2, 1} + a_{j-1,1}) + (a_{n-j+3, 2} + a_{j-2,2})+ \cdots +  (a_{\frac{n}{2}, j-\frac{n}{2}-1}+a_{\frac{n}{2}+1, j-\frac{n}{2}-1} ).
    \end{array} 
    \]
    By Part \ref{colequi} of Lemma \ref{prooffranklinalgolemma}, since the
    sums depend only on the parity of $j$, we get that the necessary
    cancellations happen, and these sums, also, add to $M$.

    For example, in the case of of F1 (see Table \ref{franklinstepfinal}) , the seventh bend diagonal sum is 
    \[\begin{array}{llllllllllllll}
    (a_{1,7}+a_{8,7}) + (a_{2,8}+a_{7,8})+(a_{3,1}+a_{6,1})+ (a_{4,2}+a_{5,2}) \\ \\
    = (N+3)+(N+1)+(N-3) + (N-1) = 4N = M.
    \end{array} 
    \]

    Thus, the left bend diagonals add to the magic sum,
    continuously. The proof that all right bend diagonals add to the magic
    sum, is similar to the case of left bend diagonals. The proof depends,
    mainly, on the fact that equidistant entries across the horizontal
    axis add to either $N+m_i$ or $N -m_i$, and all $m_i$ cancel in the
    final sum.

    Similar argument is used to prove that the top and bottom bend
    diagonals add to magic sum. By Part \ref{rowequi} of Lemma
    \ref{prooffranklinalgolemma}, pairs of equidistant entries across the
    vertical axis add to $m$ or $2N-m$. 

    For $1 \leq i \leq n/2+1$, let $i$ be odd, then the $i$-th top bend diagonal sum is given below.
    \[\begin{array}{llllllllll}  
    \left [ (a_{i,1} + a_{i,n})+ (a_{i+1,2} + a_{i+1,n-1})+ \cdots + (a_{i+\frac{n}{4}-1, \frac{n}{4}}+a_{i+\frac{n}{4}-1, \frac{3n}{4}+1}) \right ]  \\ \\
    + \left [ (a_{i+\frac{n}{4}, \frac{n}{4}+1} + a_{i+\frac{n}{4}, \frac{3n}{4}} ) + \cdots + (a_{i+\frac{n}{2}-1, \frac{n}{2}} + a_{i+\frac{n}{2}, \frac{n}{2}+1 }) \right ]  \\ \\
    = \left [(2N-m) + m  + \cdots +m  \right ] +  \left [ (2N-m) + m + \cdots + m  \right ] \\ \\
    = \frac{n}{2}N = M.
    \end{array}
    \]
    When $i$ is even, $m$ and  $2N-m$ are replaced with each other, wherever they appear in the above
    sum. Thus, the top bend diagonal sums add to $M$ when $1 \leq i \leq
    n/2+1$.

    For $n/2 + 1 < i \leq n$, the top bed diagonal sums are 
    \[\begin{array}{llllllllll}  
    (a_{i,1} + a_{i,n})+ (a_{i+1,2} + a_{i+1,n-1})+ \cdots (a_{n,n-i+1} + a_{n,i} )  \\ \\
    +(a_{1, n-i+2}+a_{1, i-1}) + (a_{2, n-i+3}+a_{2, i-2})+ \cdots +(a_{i-\frac{n}{2}-1, \frac{n}{2}}+ a_{i-\frac{n}{2}-1, \frac{n}{2}+1})
    \end{array}
    \]

    Again, by Part \ref{rowequi} of Lemma \ref{prooffranklinalgolemma}, it
    can be checked that the top bed diagonal sums add to $M$.

    For example, in the case of F1 (see Table \ref{franklinstepfinal}), the seventh top bend diagonal sum is
    \[\begin{array}{llllllllll}  
    (a_{7,1}+a_{7,8}) + (a_{8,2}+a_{8,7}) + (a_{1,3}+a_{1,6}) + (a_{2,4}+a_{2,5}) \\ \\
    = (2N-m)+ m + m + (2N-m) = 4N = M. 
    \end{array}
    \]

    Consequently, the top bend diagonals add to the magic sum,
    continuously. A similar proof, applying Part \ref{rowequi} of Lemma
    \ref{prooffranklinalgolemma}, shows that the bottom bend diagonals add
    to the magic sum, continuously.
    
    Thus, a square constructed by Algorithm \ref{franklincosntructalgo1} is
    a Franklin square.
  \end{enumerate}
\end{proof}

We now show that the Algorithm \ref{franklincosntructalgo1} is very
similar to Narayana's method.  That is, we show how a Franklin square
can be constructed as a superimposition of Chadya and Chadaka. We
illustrate the derivation of a Chadya and a Chadaka square using the
example of F3, before formulating an algorithm.  Since $n$ numbers are
placed at every step, the entries in a partially completed left side
of a Franklin square, before the final step of subtraction from $N$,
can be rewritten as multiples on $n$. For example, for F3, since $16$
numbers are placed at every step, the entries in the partially
completed left side of Franklin square F3, in Step 4 of Table
\ref{tableleftf3steps}, can be rewritten as multiples of $16$, as
follows.

{\scriptsize
  \[ 
  \begin{array}{|c|c|c|c|c|c|c|c|c|c|c|c|c|c|c|c|} \hline
    - & - & - & - & 8 & 9+16 & 8+32 & 9+48 \\ \hline
    10+48 & 7+32 & 10+16& 7 & - & - & - & - \\ \hline
    - & - & - & - & 6 & 11+16 & 6+32 & 11+48 \\ \hline
    12+48 & 5+32 & 12+16 & 5 & - & - & - & - \\ \hline
    - & - & - & - & 9 & 8+16 & 9+32 & 8+48 \\ \hline
    7+48 & 10+32 & 7+17 & 10 & - & - & - & - \\ \hline
    - & - & - & - & 11 & 6+16 & 11+32 & 6+48 \\ \hline
    5+48 & 12+32 & 5+16 & 12 & - & - & - & - \\ \hline
    - & - & - & - & 13 & 4+16 & 13+32 & 4+48 \\ \hline
    3+48 & 14+32 & 3+16& 14 & - & - & - & - \\ \hline
    - & - & - & - & 15 & 2+16 & 15+32 & 2+48 \\ \hline
    1+48 & 16+32 & 1+16 & 16 & - & - & - & - \\ \hline
    - & - & - & - & 4 & 13+16 & 4+32 & 13+48 \\ \hline
    14+48 & 3+32 & 14+16 & 3 & - & - & - & - \\ \hline
    - & - & - & - & 2 & 15+16 & 2+32 & 15+48 \\ \hline
    16+48 & 1+32 & 16+16 & 1 & - & - & - & - \\ \hline
  \end{array}
  \]
}

Since $N-rn-i = (n^2+1)-rn-i = (n^2-( r+1)n) +(n+1-i)$, the
subtractions from $N$ for F3, in Table \ref{tableleftf3fillingN}, can be
rewritten as follows.

{\scriptsize
  \[ 
  \begin{array}{|c|c|c|c|c|c|c|c|c|c|c|c|c|c|c|c|} \hline
    8+192 & 9+208& 8+224 & 9+240 & 8 & 9+16 & 8+32 & 9+48 \\ \hline
    10+ 48 & 7+32 & 10+16& 7 & 10+ 240 & 7+224 & 10+208 & 7+192 \\ \hline
    6+192& 11+208& 6+224 & 11+240 & 6 & 11+16 & 6+32 & 11+48 \\ \hline
    12+48 & 5+32 & 12+16 & 5 & 12+ 240 & 5+224 & 12+208 & 5+192 \\ \hline
    9+192 & 8+208& 9+224 & 8+240 & 9 & 8+16 & 9+32 & 8+48 \\ \hline
    7+48 & 10+32 & 7+16 & 10 & 7+ 240 & 10+224 & 7+208 & 10+192 \\ \hline
    11+192 & 6+208& 11+224 & 6+240 & 11 & 6+16 & 11+32 & 6+48 \\ \hline
    5+48 & 12+32 & 5+16 & 12 &  5+ 240 & 12+224 & 5+208 & 12+192 \\ \hline
    13+192 & 4+208& 13+224 & 4+240 & 13 & 4+16 & 13+32 & 4+48 \\ \hline
    3+48 & 14+32 & 3+16& 14 & 3+ 240 & 14+224 & 3+208 & 14+192 \\ \hline
    15+192 & 2+208& 15+224 & 2+240 & 15 & 2+16 & 15+32 & 2+48 \\ \hline
    1+48 & 16+32 & 1+16 & 16 &  1+ 240 & 16+224 & 1+208 & 16+192 \\ \hline
    13+192 & 4+208& 13+224 & 4+240 & 4 & 13+16 & 4+32 & 13+48 \\ \hline
    14+48 & 3+32 & 14+16 & 3 &  14+ 240 & 3+224 & 14+208 & 3+192 \\ \hline
    15+192& 2+208& 15+224 & 2+240 &  2 & 15+16 & 2+32 & 15+48 \\ \hline
    16+48 & 1+32 & 16+16 & 1 &  16+ 240 & 1+224 & 16+208 & 1+192 \\ \hline
  \end{array}
  \]
}

Consequently, the left side of the Franklin square can then be split
in to two squares. We call these squares {\em Chadya and flipped
  Chadaka of the left side of a Franklin square}. For example, the
left side of F3 is the sum of the Chadya and flipped Chadaka of the
left side of F3, as shown below.

{\scriptsize
  
  \[ \begin{array}{ccccccccccccc}
\mbox{Left side of F3} &=& 
   \mbox{Chadya of left side of F3 } & + & \mbox{Flipped Chadaka of left side of F3} \\ \\
&&
    \begin{array}{|c|c|c|c|c|c|c|c|c|c|c|c|c|c|c|c|} \hline
      8 & 9 & 8 & 9 & 8 & 9 & 8 & 9 \\ \hline
     10 & 7 & 10& 7 & 10 & 7 & 10 & 7 \\ \hline
      6 & 11 & 6 & 11 &  6 & 11 & 6 & 11 \\ \hline
      12 & 5 & 12 & 5 & 12 & 5 & 12 & 5 \\ \hline
      9 & 8 & 9 & 8 & 9 & 8 & 9 & 8 \\ \hline
      7 & 10 & 7 & 10 & 7 & 10 & 7 & 10 \\ \hline
      11 & 6 & 11 & 6 & 11 & 6 & 11 & 6 \\ \hline
      5 & 12 & 5 & 12 &  5 & 12 & 5 & 12 \\ \hline
      13 & 4 & 13 & 4 & 13 & 4 & 13 & 4 \\ \hline
      3 & 14 & 3& 14 & 3 & 14 & 3& 14 \\ \hline
      15 & 2 & 15 & 2& 15 & 2 & 15 & 2 \\ \hline
      1 & 16 & 1 & 16 & 1 & 16 & 1 & 16 \\ \hline
      4 & 13 & 4 & 13& 4 & 13 & 4 & 13 \\ \hline
      14 & 3 & 14 & 3 &    14 & 3 & 14 & 3 \\ \hline
      2 & 15 & 2 & 15& 2 & 15 & 2 & 15 \\ \hline
      16 & 1 & 16 & 1 &  16 & 1 & 16 & 1 \\ \hline
    \end{array} & \bf{+} &
    \begin{array}{|c|c|c|c|c|c|c|c|c|c|c|c|c|c|c|c|} \hline
      192 & 208 & 224& 240 & 0 & 16 & 32 & 48 \\ \hline
      48 & 32 & 16& 0 & 240 & 224 & 208 & 192 \\ \hline
      192 & 208 & 224& 240 & 0 & 16 & 32 & 48 \\ \hline
      48 & 32 & 16 & 0 &  240 & 224 & 208 & 192 \\ \hline
      
      192 & 208 & 224& 240 & 0 & 16 & 32 & 48 \\ \hline
      48 & 32 & 16 & 0 &  240  & 224 & 208 & 192 \\ \hline
      192 & 208 & 224& 240 &  0 & 16 & 32 & 48 \\ \hline
      48 & 32 & 16 & 0 &240 & 224 & 208 & 192 \\ \hline
      192 & 208 & 224& 240 & 0 & 16 & 32 & 48 \\ \hline
      48 & 32 & 16 & 0 &  240  & 224 & 208 & 192 \\ \hline
      192 & 208 & 224& 240 &  0 & 16 & 32 & 48 \\ \hline
      48 & 32 & 16 & 0 &240 & 224 & 208 & 192 \\ \hline

      192 & 208 & 224& 240 & 0 & 16 & 32 & 48 \\ \hline
      48 & 32 & 16 & 0 &  240  & 224 & 208 & 192 \\ \hline
      192 & 208 & 224& 240 &  0 & 16 & 32 & 48 \\ \hline
      48 & 32 & 16 & 0 &240 & 224 & 208 & 192 \\ \hline
    \end{array} 
  \end{array}
  \]
}

We now apply Step \ref{constructrightfromleftstep} of Algorithm
\ref{franklincosntructalgo1}.  To get the right side from the left
side, we swap the first $n/4$ columns with the last $n/4$ columns of
both the Chadya and Chadaka of the left side of the Franklin square,
and then add and subtract $n^2/4$. In case of F3, we get

{\scriptsize
  
  \[ \begin{array}{ccccc} 
    \begin{array}{|c|c|c|c|c|c|c|c|c|c|c|c|c|c|c|c|} \hline
      8 & 9 & 8 & 9 & 8 & 9 & 8 & 9 \\ \hline
      10 & 7 & 10& 7 & 10 & 7 & 10 & 7 \\ \hline
      6 & 11 & 6 & 11 &  6 & 11 & 6 & 11 \\ \hline
      12 & 5 & 12 & 5 & 12 & 5 & 12 & 5 \\ \hline
      9 & 8 & 9 & 8 & 9 & 8 & 9 & 8 \\ \hline
      7 & 10 & 7 & 10 & 7 & 10 & 7 & 10 \\ \hline
      11 & 6 & 11 & 6 & 11 & 6 & 11 & 6 \\ \hline
      5 & 12 & 5 & 12 &  5 & 12 & 5 & 12 \\ \hline
      13 & 4 & 13 & 4 & 13 & 4 & 13 & 4 \\ \hline
      3 & 14 & 3& 14 & 3 & 14 & 3& 14 \\ \hline
      15 & 2 & 15 & 2& 15 & 2 & 15 & 2 \\ \hline
      1 & 16 & 1 & 16 & 1 & 16 & 1 & 16 \\ \hline
      4 & 13 & 4 & 13& 4 & 13 & 4 & 13 \\ \hline
      14 & 3 & 14 & 3 &    14 & 3 & 14 & 3 \\ \hline
      2 & 15 & 2 & 15& 2 & 15 & 2 & 15 \\ \hline
      16 & 1 & 16 & 1 &  16 & 1 & 16 & 1 \\ \hline
    \end{array} & + &
    \begin{array}{|c|c|c|c|c|c|c|c|c|c|c|c|c|c|c|c|} \hline
      0 & 16 & 32 & 48 &  192 & 208 & 224& 240 \\ \hline
      240 & 224 & 208 & 192 &  48 & 32 & 16& 0 \\ \hline
      0 & 16 & 32 & 48 &  192 & 208 & 224& 240 \\ \hline
      240 & 224 & 208 & 192 &  48 & 32 & 16& 0 \\ \hline
      
      0 & 16 & 32 & 48 &  192 & 208 & 224& 240 \\ \hline
      240 & 224 & 208 & 192 &  48 & 32 & 16& 0 \\ \hline
      0 & 16 & 32 & 48 &  192 & 208 & 224& 240 \\ \hline
      240 & 224 & 208 & 192 &  48 & 32 & 16& 0 \\ \hline
      0 & 16 & 32 & 48 &  192 & 208 & 224& 240 \\ \hline
      240 & 224 & 208 & 192 &  48 & 32 & 16& 0 \\ \hline
      0 & 16 & 32 & 48 &  192 & 208 & 224& 240 \\ \hline
      240 & 224 & 208 & 192 &  48 & 32 & 16& 0 \\ \hline
      
      0 & 16 & 32 & 48 &  192 & 208 & 224& 240 \\ \hline
      240 & 224 & 208 & 192 &  48 & 32 & 16& 0 \\ \hline
      0 & 16 & 32 & 48 &  192 & 208 & 224& 240 \\ \hline
      240 & 224 & 208 & 192 &  48 & 32 & 16& 0 \\ \hline
    \end{array}  & + &
    \begin{array}{|c|c|c|c|c|c|c|c|c|c|c|c|c|c|c|c|} \hline
      64 & 64 & 64 & 64 & -64 & -64 & -64 & -64 \\ \hline
      -64 & -64 & -64 & -64 & 64 & 64 & 64 & 64   \\ \hline
      64 & 64 & 64 & 64 & -64 & -64 & -64 & -64 \\ \hline
      -64 & -64 & -64 & -64 & 64 & 64 & 64 & 64   \\ \hline
      64 & 64 & 64 & 64 & -64 & -64 & -64 & -64 \\ \hline
      -64 & -64 & -64 & -64 & 64 & 64 & 64 & 64   \\ \hline
      64 & 64 & 64 & 64 & -64 & -64 & -64 & -64 \\ \hline
      -64 & -64 & -64 & -64 & 64 & 64 & 64 & 64   \\ \hline
      64 & 64 & 64 & 64 & -64 & -64 & -64 & -64 \\ \hline
      -64 & -64 & -64 & -64 & 64 & 64 & 64 & 64   \\ \hline
      64 & 64 & 64 & 64 & -64 & -64 & -64 & -64 \\ \hline
      -64 & -64 & -64 & -64 & 64 & 64 & 64 & 64   \\ \hline
      64 & 64 & 64 & 64 & -64 & -64 & -64 & -64 \\ \hline
      -64 & -64 & -64 & -64 & 64 & 64 & 64 & 64   \\ \hline
      64 & 64 & 64 & 64 & -64 & -64 & -64 & -64 \\ \hline
      -64 & -64 & -64 & -64 & 64 & 64 & 64 & 64   \\ \hline
    \end{array}
  \end{array}
  \]
}

Observe that the Chadya, is unchanged by the column swaps. Thus the
Chadya of the left side and right side of the Franklin square are the
same. The Flipped Chadaka of the right side is obtained by adding the last
two squares in the equation above.

{\scriptsize
  
  \[ \begin{array}{ccccc} 
   \mbox{Right side of F3} &=&  \mbox{Chadya of right side of F3} & + & \mbox{Flipped Chadaka of right side of F3} \\ \\
&&
    \begin{array}{|c|c|c|c|c|c|c|c|c|c|c|c|c|c|c|c|} \hline
      8 & 9 & 8 & 9 & 8 & 9 & 8 & 9 \\ \hline
      10 & 7 & 10& 7 & 10 & 7 & 10 & 7 \\ \hline
      6 & 11 & 6 & 11 &  6 & 11 & 6 & 11 \\ \hline
      12 & 5 & 12 & 5 & 12 & 5 & 12 & 5 \\ \hline
      9 & 8 & 9 & 8 & 9 & 8 & 9 & 8 \\ \hline
      7 & 10 & 7 & 10 & 7 & 10 & 7 & 10 \\ \hline
      11 & 6 & 11 & 6 & 11 & 6 & 11 & 6 \\ \hline
      5 & 12 & 5 & 12 &  5 & 12 & 5 & 12 \\ \hline
      13 & 4 & 13 & 4 & 13 & 4 & 13 & 4 \\ \hline
      3 & 14 & 3& 14 & 3 & 14 & 3& 14 \\ \hline
      15 & 2 & 15 & 2& 15 & 2 & 15 & 2 \\ \hline
      1 & 16 & 1 & 16 & 1 & 16 & 1 & 16 \\ \hline
      4 & 13 & 4 & 13& 4 & 13 & 4 & 13 \\ \hline
      14 & 3 & 14 & 3 &    14 & 3 & 14 & 3 \\ \hline
      2 & 15 & 2 & 15& 2 & 15 & 2 & 15 \\ \hline
      16 & 1 & 16 & 1 &  16 & 1 & 16 & 1 \\ \hline
    \end{array} & + & 
    \begin{array}{|c|c|c|c|c|c|c|c|c|c|c|c|c|c|c|c|} \hline
      64 & 80 & 96 & 112 &  128 & 144 & 160 & 176 \\ \hline
      176 & 160 & 144 & 128 &  112 & 96 & 80& 64 \\ \hline
      64 & 80 & 96 & 112 &  128 & 144 & 160 & 176 \\ \hline
      176 & 160 & 144 & 128 &  112 & 96 & 80& 64 \\ \hline
      64 & 80 & 96 & 112 &  128 & 144 & 160 & 176 \\ \hline
      176 & 160 & 144 & 128 &  112 & 96 & 80& 64 \\ \hline
      64 & 80 & 96 & 112 &  128 & 144 & 160 & 176 \\ \hline
      176 & 160 & 144 & 128 &  112 & 96 & 80& 64 \\ \hline
      64 & 80 & 96 & 112 &  128 & 144 & 160 & 176 \\ \hline
      176 & 160 & 144 & 128 &  112 & 96 & 80& 64 \\ \hline
      64 & 80 & 96 & 112 &  128 & 144 & 160 & 176 \\ \hline
      176 & 160 & 144 & 128 &  112 & 96 & 80& 64 \\ \hline
      64 & 80 & 96 & 112 &  128 & 144 & 160 & 176 \\ \hline
      176 & 160 & 144 & 128 &  112 & 96 & 80& 64 \\ \hline
      64 & 80 & 96 & 112 &  128 & 144 & 160 & 176 \\ \hline
      176 & 160 & 144 & 128 &  112 & 96 & 80& 64 \\ \hline
    \end{array} 
  \end{array}
  \]
}

Putting the left and right sides together, we get the Chadya and the
Flipped Chadaka of the Franklin square. The flipped Chadaka is flipped
again to get the Chadaka of the Franklin square.  See Table
\ref{chadya16x16franklin} for the Chadya and Chadaka of F3.

We, now, describe a new Algorithm to construct Franklin squares as
superimposition of Chadya and Chadaka squares.  This algorithm, as we
have seen, is just a rewriting of Algorithm
\ref{franklincosntructalgo1}.

\begin{algo}
  \begin{enumerate}
  \item \label{step1algo2franklin} We start with placing  the numbers $1$ to  $n$
    in the $n/4$ and $n/4+1$ th columns,  using the
    following operation sequence.

    \[ \begin{array}{|c|ccc|} \hline
      \mbox{Part}  & \mbox{Bottom} & \mbox{Top} & \mbox{Middle}\\ \hline
      \mbox{Operation} &  \mbox{Up} & \mbox{Up} & \mbox{Down} \\ \hline
      \mbox{Starting Number} & 1 & 1+n/4 & 1+n/2 \\ \hline
    \end{array}
    \]

    For example, in the case of the Franklin square F3, we get
    {\scriptsize
      \[
      \begin{array}{|c|c|c|c|c|c|c|c|c|c|c|c|c|c|c|c|c|c|c|c|c|c|c|c|c|c|c|c|c|c|c|c|} \hline
        - & - & - & - & 8 & - & - & - &   - & - & - & - & - & - & - & - \\ \hline
        - & - & - & 7 & - & - & - & - &   - & - & - & - & - & - & - & -\\ \hline
        - & - & - & - & 6 & - & - & - &   - & - & - & - & - & - & - & - \\ \hline
        - & - & - & 5 & - & - & - & - &   - & - & - & - & - & - & - & -\\ \hline
        - & - & - & - & 9 & - & - & - &   - & - & - & - & - & - & - & -\\ \hline
        - & - & - & 10 & - & - & - & - &   - & - & - & - & - & - & - & - \\ \hline
        - & - & - & - & 11 & - & - & - &   - & - & - & - & - & - & - & -\\ \hline
        - & - & - & 12 & - & - & - & - &   - & - & - & - & - & - & - & -\\ \hline
        - & - & - & - & 13 & - & - & - &   - & - & - & - & - & - & - & -\\ \hline
        - & - & - & 14 & - & - & - & - &   - & - & - & - & - & - & - & -\\ \hline
        - & - & - & - & 15 & - & - & - &   - & - & - & - & - & - & - & -\\ \hline
        - & - & - & 16 & - & - & - & - &   - & - & - & - & - & - & - & - \\ \hline
        - & - & - & - & 4 & - & - & - &   - & - & - & - & - & - & - & - \\ \hline
        - & - & - & 3 & - & - & - & - &   - & - & - & - & - & - & - & -\\ \hline
        - & - & - & - & 2 & - & - & - &   - & - & - & - & - & - & - & -\\ \hline
        - & - & - & 1 & - & - & - & - &   - & - & - & - & - & - & - & -\\ \hline
      \end{array}
      \]
    }
    Observe that this is Step 1 of Table \ref{tableleftf3steps}.

  \item Let $a$ be the entry in a row in Step \ref{step1algo2franklin}. Then the row is filled with $a$
    and $n+1-a$ in alternate columns. This gives us the Chadya of the square. 
    In the case of F3, we get
    {\scriptsize
      \[
      \begin{array}{|c|c|c|c|c|c|c|c|c|c|c|c|c|c|c|c|c|c|c|c|c|c|c|c|c|c|c|c|c|c|c|c|} \hline
        8 & 9 & 8 & 9 & 8 & 9 & 8 & 9 & 8 & 9 & 8 & 9 & 8 & 9 & 8 & 9  \\ \hline
        10 & 7 & 10 & 7 & 10 & 7 & 10 & 7 & 10 & 7 & 10 & 7 & 10 & 7 & 10 & 7\\ \hline
        6 & 11 & 6 & 11 & 6 & 11 & 6 & 11 & 6 & 11 & 6 & 11 & 6 & 11 & 6 & 11 \\ \hline
        12 & 5 & 12 & 5 & 12 & 5 & 12 & 5 &  12 & 5 & 12 & 5 & 12 & 5 & 12 & 5 \\ \hline
        9 & 8 & 9 & 8 & 9 & 8 & 9 & 8 & 9 & 8 & 9 & 8 & 9 & 8 & 9 & 8 \\ \hline
        7 & 10 & 7 & 10 & 7 & 10 & 7 & 10 &  7 & 10 & 7 & 10 & 7 & 10 & 7 & 10 \\ \hline
        11 & 6 & 11 & 6 & 11 & 6 & 11 & 6 & 11 & 6 & 11 & 6 & 11 & 6 & 11 & 6 \\ \hline
        5 & 12 & 5 & 12 & 5 & 12 & 5 & 12 &  5 & 12 & 5 & 12 & 5 & 12 & 5 & 12 \\ \hline
        13 & 4 & 13 & 4& 13 & 4 & 13 & 4 & 13 & 4 & 13 & 4& 13 & 4 & 13 & 4  \\ \hline
        3 & 14 & 3& 14 & 3 & 14 & 3 & 14 & 3 & 14 & 3& 14 & 3 & 14 & 3 & 14 \\ \hline
        15 & 2 & 15 & 2 & 15 & 2 & 15 & 2 & 15 & 2 & 15 & 2 & 15 & 2 & 15 & 2 \\ \hline
        1 & 16 & 1 & 16 & 1 & 16 & 1 & 16 &  1 & 16 & 1 & 16 & 1 & 16 & 1 & 16 \\ \hline
        4 & 13 & 4 & 13& 4 & 13 & 4 & 13 &   4 & 13 & 4 & 13& 4 & 13 & 4 & 13 \\ \hline
        14 & 3 & 14 & 3 & 14 & 3 & 14& 3 & 14 & 3 & 14 & 3 & 14 & 3 & 14& 3\\ \hline
        2 & 15 & 2 & 15 & 2 & 15 & 2 & 15 & 2 & 15 & 2 & 15 & 2 & 15 & 2 & 15 \\ \hline
        16 & 1 &16 & 1 & 16 & 1 & 16 & 1 &  16 & 1 &16 & 1 & 16 & 1 & 16 & 1  \\ \hline
      \end{array}      
      \]
    }

  \item To construct the Chadaka, we start with the left side of the
    Franklin square.  We place zeroes in the two middle rows, in the
    same sequence, as in Step \ref{partialstep} of Algorithm
    \ref{franklincosntructalgo1}.  This places zeroes, alternating
    between the middle two columns, starting with the $n/4+1$ th
    column, in a downward direction.

    In the case of F3, this step will produce

    {\scriptsize
      \[
      \begin{array}{|c|c|c|c|c|c|c|c|c|c|c|c|c|c|c|c|} \hline
        - & - & - & - & 0 & - & - & -  \\ \hline
        - & - & - & 0 & - & - & - & - \\ \hline
        - & - & - & - & 0 & - & - & - \\ \hline
        - & - & - & 0 & - & - & - & - \\ \hline
        - & - & - & - & 0 & - & - & - \\ \hline
        - & - & - & 0 & - & - & - & - \\ \hline
        - & - & - & - & 0 & - & - & - \\ \hline
        - & - & - & 0 & - & - & - & - \\ \hline
        - & - & - & - & 0 & - & - & - \\ \hline
        - & - & - & 0 & - & - & - & - \\ \hline
        - & - & - & - & 0 & - & - & - \\ \hline
        - & - & - & 0 & - & - & - & -  \\ \hline
        - & - & - & - & 0 & - & - & - \\ \hline
        - & - & - & 0 & - & - & - & - \\ \hline
        - & - & - & - & 0 & - & - & - \\ \hline
        - & - & - & 0 & - & - & - & - \\ \hline
      \end{array} 
      \]
    }

  \item  Next, we fill the empty cells in the middle two columns with $n^2-n$.

    Filling the empty cells with $256-16 =240$ for F3, we get
    {\scriptsize
      \[
      \begin{array}{|c|c|c|c|c|c|c|c|c|c|c|c|c|c|c|c|} \hline
        - & - & - & 240 & 0 & - & - & -  \\ \hline
        - & - & - & 0 & 240 & - & - & - \\ \hline
        - & - & - & 240 & 0 & - & - & - \\ \hline
        - & - & - & 0 & 240 & - & - & - \\ \hline
        - & - & - & 240 & 0 & - & - & - \\ \hline
        - & - & - & 0 & 240 & - & - & - \\ \hline
        - & - & - & 240 & 0 & - & - & - \\ \hline
        - & - & - & 0 & 240 & - & - & - \\ \hline
        - & - & - & 240 & 0 & - & - & - \\ \hline
        - & - & - & 0 & 240 & - & - & - \\ \hline
        - & - & - & 240 & 0 & - & - & - \\ \hline
        - & - & - & 0 & 240 & - & - & -  \\ \hline
        - & - & - & 240 & 0 & - & - & - \\ \hline
        - & - & - & 0 & 240 & - & - & - \\ \hline
        - & - & - & 240 & 0 & - & - & - \\ \hline
        - & - & - & 0 & 240 & - & - & - \\ \hline
      \end{array}
      \]
    }

  \item We place $n \times i $ where $i=1, \dots, n/4-1$, to the right or left of
    zero depending on which side of zero is empty. Finally, we place
    $n^2-(i+1)n$, where $i=1, \dots, n/4-1$ to the right or left of
    $n^2-n$, depending on which side of $n^2-n$ is empty. This gives us
    the flipped Chadaka of the left side of the Franklin square.

    In case of F3, we get

    {\scriptsize
      \[
      \begin{array}{|c|c|c|c|c|c|c|c|c|c|c|c|c|c|c|c|} \hline
        192 & 208 & 224& 240 & 0 & 16 & 32 & 48 \\ \hline
        48 & 32 & 16& 0 & 240 & 224 & 208 & 192 \\ \hline
        192 & 208 & 224& 240 & 0 & 16 & 32 & 48 \\ \hline
        48 & 32 & 16 & 0 &  240 & 224 & 208 & 192 \\ \hline
        
        192 & 208 & 224& 240 & 0 & 16 & 32 & 48 \\ \hline
        48 & 32 & 16 & 0 &  240  & 224 & 208 & 192 \\ \hline
        192 & 208 & 224& 240 &  0 & 16 & 32 & 48 \\ \hline
        48 & 32 & 16 & 0 &240 & 224 & 208 & 192 \\ \hline
        192 & 208 & 224& 240 & 0 & 16 & 32 & 48 \\ \hline
        48 & 32 & 16 & 0 &  240  & 224 & 208 & 192 \\ \hline
        192 & 208 & 224& 240 &  0 & 16 & 32 & 48 \\ \hline
        48 & 32 & 16 & 0 &240 & 224 & 208 & 192 \\ \hline

        192 & 208 & 224& 240 & 0 & 16 & 32 & 48 \\ \hline
        48 & 32 & 16 & 0 &  240  & 224 & 208 & 192 \\ \hline
        192 & 208 & 224& 240 &  0 & 16 & 32 & 48 \\ \hline
        48 & 32 & 16 & 0 &240 & 224 & 208 & 192 \\ \hline
      \end{array} 
      \]
    }

  \item The flipped Chadaka of the right side of the Franklin
    square from the right side is constructed by applying  Step
    \ref{constructrightfromleftstep} of Algorithm
    \ref{franklincosntructalgo1}.  Swap the first $n/4$ columns with the
    last $n/4$ columns of the Chadaka of the left side to build the
    Chadaka of the right side of the Franklin square. For odd rows of
    this modified square, $n^2/4$ is added to every entry in the first
    half of the row, and $n^2/4$ is subtracted from every entry of the
    second half. For even rows, $n^2/4$ is subtracted from every entry
    in the first half of the row, and $n^2/4$ is added to every entry of
    the second half.  This gives us the flipped Chadaka of the  right side of the Franklin
    square.

    Thus, we get the  flipped Chadaka  of the right side of the Franklin square F3, as shown below.

    {\scriptsize
      \[ \begin{array}{ccccccccccccc}
        \begin{array}{|c|c|c|c|c|c|c|c|c|c|c|c|c|c|c|c|} \hline
          0 & 16 & 32 & 48 &  192 & 208 & 224& 240  \\ \hline
          240 & 224 & 208 & 192 & 48 & 32 & 16& 0   \\ \hline
          0 & 16 & 32 & 48 & 192 & 208 & 224& 240  \\ \hline
          240 & 224 & 208 & 192 &  48 & 32 & 16 & 0  \\ \hline
          0 & 16 & 32 & 48 & 192 & 208 & 224& 240  \\ \hline
          240  & 224 & 208 & 192 &  48 & 32 & 16 & 0 \\ \hline
          0 & 16 & 32 & 48 & 192 & 208 & 224& 240 \\ \hline
          240 & 224 & 208 & 192 & 48 & 32 & 16 & 0  \\ \hline
          0 & 16 & 32 & 48 & 192 & 208 & 224& 240   \\ \hline
          240  & 224 & 208 & 192 & 48 & 32 & 16 & 0  \\ \hline
          0 & 16 & 32 & 48 & 192 & 208 & 224& 240  \\ \hline
          240 & 224 & 208 & 192 & 48 & 32 & 16 & 0  \\ \hline
          0 & 16 & 32 & 48 & 192 & 208 & 224& 240  \\ \hline
          240  & 224 & 208 & 192 &  48 & 32 & 16 & 0  \\ \hline
          0 & 16 & 32 & 48 & 192 & 208 & 224& 240  \\ \hline
          240 & 224 & 208 & 192 & 48 & 32 & 16 & 0  \\ \hline
        \end{array}  & + &
        \begin{array}{|c|c|c|c|c|c|c|c|c|c|c|c|c|c|c|c|} \hline
          64 & 64 & 64 & 64 & -64 & -64 & -64 & -64 \\ \hline
          -64 & -64 & -64 & -64 & 64 & 64 & 64 & 64   \\ \hline
          64 & 64 & 64 & 64 & -64 & -64 & -64 & -64 \\ \hline
          -64 & -64 & -64 & -64 & 64 & 64 & 64 & 64   \\ \hline
          64 & 64 & 64 & 64 & -64 & -64 & -64 & -64 \\ \hline
          -64 & -64 & -64 & -64 & 64 & 64 & 64 & 64   \\ \hline
          64 & 64 & 64 & 64 & -64 & -64 & -64 & -64 \\ \hline
          -64 & -64 & -64 & -64 & 64 & 64 & 64 & 64   \\ \hline
          64 & 64 & 64 & 64 & -64 & -64 & -64 & -64 \\ \hline
          -64 & -64 & -64 & -64 & 64 & 64 & 64 & 64   \\ \hline
          64 & 64 & 64 & 64 & -64 & -64 & -64 & -64 \\ \hline
          -64 & -64 & -64 & -64 & 64 & 64 & 64 & 64   \\ \hline
          64 & 64 & 64 & 64 & -64 & -64 & -64 & -64 \\ \hline
          -64 & -64 & -64 & -64 & 64 & 64 & 64 & 64   \\ \hline
          64 & 64 & 64 & 64 & -64 & -64 & -64 & -64 \\ \hline
          -64 & -64 & -64 & -64 & 64 & 64 & 64 & 64   \\ \hline
        \end{array}
      \end{array}
      \]
    }
    We put the flipped Chadakas of the  two sides to get the flipped Chadaka of F3.

    {\scriptsize

      \[ \begin{array}{ccccc} 
        \mbox{Flipped Chadaka of the Franklin square F3} \\
        \begin{array}{|c|c|c|c|c|c|c|c|c|c|c|c|c|c|c|c|} \hline
          192 & 208 & 224& 240 & 0 & 16 & 32 & 48 \\ \hline
          48 & 32 & 16& 0 & 240 & 224 & 208 & 192 \\ \hline
          192 & 208 & 224& 240 & 0 & 16 & 32 & 48 \\ \hline
          48 & 32 & 16 & 0 &  240 & 224 & 208 & 192 \\ \hline
          
          192 & 208 & 224& 240 & 0 & 16 & 32 & 48 \\ \hline
          48 & 32 & 16 & 0 &  240  & 224 & 208 & 192 \\ \hline
          192 & 208 & 224& 240 &  0 & 16 & 32 & 48 \\ \hline
          48 & 32 & 16 & 0 &240 & 224 & 208 & 192 \\ \hline
          192 & 208 & 224& 240 & 0 & 16 & 32 & 48 \\ \hline
          48 & 32 & 16 & 0 &  240  & 224 & 208 & 192 \\ \hline
          192 & 208 & 224& 240 &  0 & 16 & 32 & 48 \\ \hline
          48 & 32 & 16 & 0 &240 & 224 & 208 & 192 \\ \hline

          192 & 208 & 224& 240 & 0 & 16 & 32 & 48 \\ \hline
          48 & 32 & 16 & 0 &  240  & 224 & 208 & 192 \\ \hline
          192 & 208 & 224& 240 &  0 & 16 & 32 & 48 \\ \hline
          48 & 32 & 16 & 0 &240 & 224 & 208 & 192 \\ \hline
        \end{array} 
        \begin{array}{|c|c|c|c|c|c|c|c|c|c|c|c|c|c|c|c|} \hline
          64 & 80 & 96 & 112 &  128 & 144 & 160 & 176 \\ \hline
          176 & 160 & 144 & 128 &  112 & 96 & 80& 64 \\ \hline
          64 & 80 & 96 & 112 &  128 & 144 & 160 & 176 \\ \hline
          176 & 160 & 144 & 128 &  112 & 96 & 80& 64 \\ \hline
          64 & 80 & 96 & 112 &  128 & 144 & 160 & 176 \\ \hline
          176 & 160 & 144 & 128 &  112 & 96 & 80& 64 \\ \hline
          64 & 80 & 96 & 112 &  128 & 144 & 160 & 176 \\ \hline
          176 & 160 & 144 & 128 &  112 & 96 & 80& 64 \\ \hline
          64 & 80 & 96 & 112 &  128 & 144 & 160 & 176 \\ \hline
          176 & 160 & 144 & 128 &  112 & 96 & 80& 64 \\ \hline
          64 & 80 & 96 & 112 &  128 & 144 & 160 & 176 \\ \hline
          176 & 160 & 144 & 128 &  112 & 96 & 80& 64 \\ \hline
          64 & 80 & 96 & 112 &  128 & 144 & 160 & 176 \\ \hline
          176 & 160 & 144 & 128 &  112 & 96 & 80& 64 \\ \hline
          64 & 80 & 96 & 112 &  128 & 144 & 160 & 176 \\ \hline
          176 & 160 & 144 & 128 &  112 & 96 & 80& 64 \\ \hline
        \end{array} 

      \end{array}
      \]
    }

    Because the Chadya of the left side and the right side are the same, we can also construct the two sides separately.
    For example, the left side of Franklin square F1 is constructed as follows.
    {\scriptsize   
      \[
      \begin{array}{ccccccccccccc}
        \mbox{Chadya} && \mbox{Flipped Chadaka} \\
        \begin{array}{|c|c|c|c|c|c|c|c|} \hline
          4 & 5 & 4 & 5  \\ \hline
          6& 3&6&3 \\ \hline
          5 & 4 & 5 & 4  \\ \hline
          3&6&3&6\\ \hline
          7& 2&7 &2 \\ \hline
          1 & 8 &1 & 8   \\ \hline
          2& 7& 2&7  \\ \hline
          8 &1 & 8 &1  \\ \hline
        \end{array} &+&
        \begin{array}{|c|c|c|c|c|c|c|c|} \hline
          
          48 & 56 & 0 &  8 \\ \hline
          8 & 0 & 56 & 48 \\ \hline
          48 & 56 & 0 &  8 \\ \hline
          8 & 0 & 56 & 48 \\ \hline
          48 & 56 & 0 &  8 \\ \hline
          8 & 0 & 56 & 48 \\ \hline
          48 & 56 & 0 &  8 \\ \hline
          8 & 0 & 56 & 48 \\ \hline
          
        \end{array}
      \end{array}
      \]
    }

    The right side of Franklin square F1 is derived below.
    {\scriptsize    
      \[
      \begin{array}{ccccccccccccc}
        \mbox{Chadya} && \mbox{Flipped Chadaka} \\ 
        \begin{array}{|c|c|c|c|c|c|c|c|} \hline
          4 & 5 & 4 & 5  \\ \hline
          6& 3&6&3 \\ \hline
          5 & 4 & 5 & 4  \\ \hline
          3&6&3&6\\ \hline
          7& 2&7 &2 \\ \hline
          1 & 8 &1 & 8   \\ \hline
          2& 7& 2&7  \\ \hline
          8 &1 & 8 &1  \\ \hline
        \end{array} &+& \left  \{ \hspace{.025in}
        \begin{array}{|c|c|c|c|c|c|c|c|} \hline
          
          0 &  8 & 48 & 56  \\ \hline
          56 & 48 & 8 & 0   \\ \hline
          0 &  8 & 48 & 56  \\ \hline
          56 & 48 & 8 & 0   \\ \hline
          0 &  8 & 48 & 56  \\ \hline
          56 & 48 & 8 & 0   \\ \hline
          0 &  8 & 48 & 56  \\ \hline
          56 & 48 & 8 & 0   \\ \hline
          
        \end{array} +
        \begin{array}{|c|c|c|c|c|c|c|c|} \hline
          16 & 16 & -16 & -16 \\ \hline
          -16 & -16 & 16 & 16 \\ \hline
          16 & 16 & -16 & -16 \\ \hline
          -16 & -16 & 16 & 16 \\ \hline
          16 & 16 & -16 & -16 \\ \hline
          -16 & -16 & 16 & 16 \\ \hline
          16 & 16 & -16 & -16 \\ \hline
          -16 & -16 & 16 & 16 \\ \hline
        \end{array} \hspace{.025in} \right \}
      \end{array}
      \]
    }

  \end{enumerate}

\end{algo}
 See Tables \ref{franklin8chadya} and \ref{chadya16x16franklin} for the Chadya and Chadaka of F1 and F3, respectively.

\begin{lemma} \label{franklinpandiaohlemma}
  Consider an $n \times n$ Franklin square. Let the magic sum be denoted by $M$, then pandiagonals add to $M \pm n^2/2$.
\end{lemma}
\begin{proof}

  We first look at left pandiagonals (see Figure \ref{pandiagonas}). 
  Since every pandiagonal starts from the first row, let $P_{1,c}$
  denote a pandiagonal that starts from column $c$.  Let $a_{i,j}$
  denote the entries of a $n \times n$ Franklin square.

  Let $a_{i,j}$ belong to a pandiagonal and let $1 \leq i \leq n/2$.
  If $1 \leq j \leq n/2$, then $a_{n/2+i, n/2+j}$ also belong to the
  pandiagonal.  On the other hand if $j > n/2$, then $a_{n/2+i,
    j-n/2}$ belongs to the pandiagonal. Thus, every pandiagonal
  $P_{1,c}$ is made up of $n/2$ paired entries.

  Let $1 \leq i \leq n/2$, define $s_{i,j}$ to be 
  \[s_{i,j} = \left \{ 
  \begin{array}{lllll}
    a_{i,j} + a_{\frac{n}{2} + i, \frac{n}{2}+j} & \mbox{ if } 1 \leq j \leq n/2 \\ \\
    a_{i,j} + a_{i,j-\frac{n}{2}} & \mbox{ if } n/2 < j \leq n.
  \end{array}
  \right .
  \]

  Observe that $s_{i,j} = s_{i, j+\frac{n}{2}}$. Therefore, by the
  continuity property of pandiagonals, it is sufficient to consider
  $s_{i,j}$, where $1 \leq i,j \leq \frac{n}{2}$, to derive pandiagonal
  sums. Each pandiagonal sum contains $n/2$ such paired sums.  Let
  $ch_{i,j}$ and $cd_{i,j}$ denote the entries of the Chadya and
  Chadaka, respectively. Let $y_{i,j} = ch_{i,j}+ ch_{\frac{n}{2} + i,
    \frac{n}{2}+j} $ and $d_{i,j} = cd_{i,j}+ cd_{\frac{n}{2} + i,
    \frac{n}{2}+j}$.  Then $s_{i,j} = y_{i,j}+d_{i,j}$.

  By construction, for $1 \leq i \leq n/4$ and $1 \leq j \leq n/2$,

  \[ y_{i,j} = \left \{ \begin{array}{lllllllllll}  
    \frac{5n}{4} + 1 & \mbox{ if $i$ is odd and $j$ is odd.} \\ \\
    \frac{3n}{4} + 1 & \mbox{ if $i$ is odd and $j$ is even.} \\ \\
    \frac{3n}{4} + 1 & \mbox{ if $i$ is even and $j$ is odd.} \\ \\
    \frac{5n}{4} + 1 & \mbox{ if $i$ is even and $j$ is even.} \\ \\
  \end{array}
  \right.
  \]

  Let $n/4 < i \leq n/2$ and $1 \leq j \leq n/2$. Then 
  \[ y_{i,j} = \left \{ \begin{array}{lllllllllll}  
    \frac{3n}{4} + 1 & \mbox{ if $i$ is odd and $j$ is odd.} \\ \\
    \frac{5n}{4} + 1 & \mbox{ if $i$ is odd and $j$ is even.} \\ \\
    \frac{5n}{4} + 1 & \mbox{ if $i$ is even and $j$ is odd.} \\ \\
    \frac{3n}{4} + 1 & \mbox{ if $i$ is even and $j$ is even.} \\ \\
  \end{array}
  \right.
  \]

  For example, in the case of F1, 
  {\scriptsize   
    \[ [y_{i,j}] = 
    \begin{array}{ccccccccccccc}
      \begin{array}{|c|c|c|c|c|c|c|c|} \hline
        4 & 5 & 4 & 5  \\ \hline
        6& 3&6&3 \\ \hline
        5 & 4 & 5 & 4  \\ \hline
        3&6&3&6\\ \hline
      \end{array} &+&
      \begin{array}{|c|c|c|c|c|c|c|c|} \hline
        7& 2&7 &2 \\ \hline
        1 & 8 &1 & 8   \\ \hline
        2& 7& 2&7  \\ \hline
        8 &1 & 8 &1  \\ \hline
      \end{array} &=&
      \begin{array}{|c|c|c|c|c|c|c|c|} \hline
        11 & 7 &  11 & 7 \\ \hline
        7 & 11 & 7 & 11 \\ \hline
        7 & 11 & 7 & 11 \\ \hline
        11 & 7 &  11 & 7 \\ \hline
      \end{array} = 
      \begin{array}{|c|c|c|c|c|c|c|c|c|c|c|c|c|c|c|c|} \hline
        \frac{5n}{4} + 1 & \frac{3n}{4} + 1 &  \frac{5n}{4} + 1 & \frac{3n}{4} + 1  \\ \hline
        \frac{3n}{4} + 1 &  \frac{5n}{4} + 1 & \frac{3n}{4} + 1 & \frac{5n}{4} + 1  \\ \hline
        \frac{3n}{4} + 1 &  \frac{5n}{4} + 1 & \frac{3n}{4} + 1 & \frac{5n}{4} + 1  \\ \hline
        \frac{5n}{4} + 1 & \frac{3n}{4} + 1 &  \frac{5n}{4} + 1 & \frac{3n}{4} + 1  \\ \hline
      \end{array}
    \end{array}
    \]
  }

  Check that,  in the case of the Franklin square F3,

  {\scriptsize
    \[ [y_{i,j}] = 
    \begin{array}{|c|c|c|c|c|c|c|c|c|c|c|c|c|c|c|c|} \hline
      \frac{5n}{4} + 1 & \frac{3n}{4} + 1 &  \frac{5n}{4} + 1 & \frac{3n}{4} + 1 & \frac{5n}{4} + 1 & \frac{3n}{4} + 1 &  \frac{5n}{4} + 1 & \frac{3n}{4} + 1 \\ \hline
      \frac{3n}{4} + 1 &  \frac{5n}{4} + 1 & \frac{3n}{4} + 1 & \frac{5n}{4} + 1 & \frac{3n}{4} + 1 &  \frac{5n}{4} + 1 & \frac{3n}{4} + 1  & \frac{5n}{4} + 1 \\ \hline
      \frac{5n}{4} + 1 & \frac{3n}{4} + 1 &  \frac{5n}{4} + 1 & \frac{3n}{4} + 1 & \frac{5n}{4} + 1 & \frac{3n}{4} + 1 &  \frac{5n}{4} + 1 & \frac{3n}{4} + 1 \\ \hline
      \frac{3n}{4} + 1 &  \frac{5n}{4} + 1 & \frac{3n}{4} + 1 & \frac{5n}{4} + 1 & \frac{3n}{4} + 1 &  \frac{5n}{4} + 1 & \frac{3n}{4} + 1  & \frac{5n}{4} + 1 \\ \hline
      \frac{3n}{4} + 1 &  \frac{5n}{4} + 1 & \frac{3n}{4} + 1 & \frac{5n}{4} + 1 & \frac{3n}{4} + 1 &  \frac{5n}{4} + 1 & \frac{3n}{4} + 1  & \frac{5n}{4} + 1 \\ \hline
      \frac{5n}{4} + 1 & \frac{3n}{4} + 1 &  \frac{5n}{4} + 1 & \frac{3n}{4} + 1 & \frac{5n}{4} + 1 & \frac{3n}{4} + 1 &  \frac{5n}{4} + 1 & \frac{3n}{4} + 1 \\ \hline
      \frac{3n}{4} + 1 &  \frac{5n}{4} + 1 & \frac{3n}{4} + 1 & \frac{5n}{4} + 1 & \frac{3n}{4} + 1 &  \frac{5n}{4} + 1 & \frac{3n}{4} + 1  & \frac{5n}{4} + 1 \\ \hline
      \frac{5n}{4} + 1 & \frac{3n}{4} + 1 &  \frac{5n}{4} + 1 & \frac{3n}{4} + 1 & \frac{5n}{4} + 1 & \frac{3n}{4} + 1 &  \frac{5n}{4} + 1 & \frac{3n}{4} + 1 \\ \hline
    \end{array}
    \]
  }
  Consequently, when we add all the $y_{i,j}$ along any left pandiagonal, we always get the sum to be 
  \[
  \frac{n}{2} \left \{ \frac{5n}{4} + 1 +  \frac{3n}{4} + 1  \right \} = \frac{n^2}{2} + \frac{n}{2}.
  \]

  Let $1 \leq j \leq \frac{n}{4}$.  When $i$ is odd, we get
  \[ \begin{array}{lllllllllll}  
    d_{i,j} =   n^2 + 2(j-1)n, \\ 
    d_{i,\frac{n}{4}+j} = n^2 - \left ( \frac{n}{2} - 2(j-1) \right )n.
  \end{array}
  \]

  When $i$ is even, we get
  \[ \begin{array}{lllllllllll}  
    d_{i,j} =   n^2 - 2jn, \\ 
    d_{i,\frac{n}{4}+j} = n^2 + \left ( \frac{n}{2} - 2(j-1) \right )n. 
  \end{array}
  \]

  For example, in the case of F3, we get

  {\scriptsize
    
    \[ \begin{array}{ccccccccccccc}
      [d_{i,j}]=  \\ \\
      
      \begin{array}{|c|c|c|c|c|c|c|c|c|c|c|c|c|c|c|c|} \hline
        192 & 208 & 224& 240 & 0 & 16 & 32 & 48 \\ \hline
        48 & 32 & 16& 0 & 240 & 224 & 208 & 192 \\ \hline
        192 & 208 & 224& 240 & 0 & 16 & 32 & 48 \\ \hline
        48 & 32 & 16 & 0 &  240 & 224 & 208 & 192 \\ \hline
        
        192 & 208 & 224& 240 & 0 & 16 & 32 & 48 \\ \hline
        48 & 32 & 16 & 0 &  240  & 224 & 208 & 192 \\ \hline
        192 & 208 & 224& 240 &  0 & 16 & 32 & 48 \\ \hline
        48 & 32 & 16 & 0 &240 & 224 & 208 & 192 \\ \hline
      \end{array} 
      +
      \begin{array}{|c|c|c|c|c|c|c|c|c|c|c|c|c|c|c|c|} \hline
        0 & 16 & 32 & 48 &  192 & 208 & 224& 240 \\ \hline
        240 & 224 & 208 & 192 &  48 & 32 & 16& 0 \\ \hline
        0 & 16 & 32 & 48 &  192 & 208 & 224& 240 \\ \hline
        240 & 224 & 208 & 192 &  48 & 32 & 16& 0 \\ \hline
        
        0 & 16 & 32 & 48 &  192 & 208 & 224& 240 \\ \hline
        240 & 224 & 208 & 192 &  48 & 32 & 16& 0 \\ \hline
        0 & 16 & 32 & 48 &  192 & 208 & 224& 240 \\ \hline
        240 & 224 & 208 & 192 &  48 & 32 & 16& 0 \\ \hline
        
      \end{array}+
      \begin{array}{|c|c|c|c|c|c|c|c|c|c|c|c|c|c|c|c|} \hline
        64 & 64 & 64 & 64 & -64 & -64 & -64 & -64 \\ \hline
        -64 & -64 & -64 & -64 & 64 & 64 & 64 & 64   \\ \hline
        64 & 64 & 64 & 64 & -64 & -64 & -64 & -64 \\ \hline
        -64 & -64 & -64 & -64 & 64 & 64 & 64 & 64   \\ \hline
        64 & 64 & 64 & 64 & -64 & -64 & -64 & -64 \\ \hline
        -64 & -64 & -64 & -64 & 64 & 64 & 64 & 64   \\ \hline
        64 & 64 & 64 & 64 & -64 & -64 & -64 & -64 \\ \hline
        -64 & -64 & -64 & -64 & 64 & 64 & 64 & 64   \\ \hline
      \end{array}

      \\ \\
      = 
      \begin{array}{|c|c|c|c|c|c|c|c|c|c|c|c|c|c|c|c|} \hline
        256 & 288 & 320 & 352 & 128 & 160 & 192 & 224 \\ \hline
        224 & 192 & 160 & 128 & 352 & 320 & 288 & 256 \\ \hline
        256 & 288 & 320 & 352 & 128 & 160 & 192 & 224 \\ \hline
        224 & 192 & 160 & 128 & 352 & 320 & 288 & 256 \\ \hline
        256 & 288 & 320 & 352 & 128 & 160 & 192 & 224 \\ \hline
        224 & 192 & 160 & 128 & 352 & 320 & 288 & 256 \\ \hline
        256 & 288 & 320 & 352 & 128 & 160 & 192 & 224 \\ \hline
        224 & 192 & 160 & 128 & 352 & 320 & 288 & 256 \\ \hline
        
      \end{array}

      =  \begin{array}{|c|c|c|c|c|c|c|c|c|c|c|c|c|c|c|c|} \hline
        n^2 & n^2+2n & n^2+4n & n^2+6n  & n^2-8n & n^2-6n & n^2-4n & n^2 - 2n \\ \hline
        n^2 -2n & n^2-4n &  n^2-6n & n^2-8n &  n^2+6n & n^2+4n &  n^2+2n & n^2 \\ \hline
        n^2 & n^2+2n & n^2+4n & n^2+6n  & n^2-8n & n^2-6n & n^2-4n & n^2 - 2n \\ \hline
        n^2 -2n & n^2-4n &  n^2-6n & n^2-8n &  n^2+6n & n^2+4n &  n^2+2n & n^2 \\ \hline
        n^2 & n^2+2n & n^2+4n & n^2+6n  & n^2-8n & n^2-6n & n^2-4n & n^2 - 2n \\ \hline
        n^2 -2n & n^2-4n &  n^2-6n & n^2-8n &  n^2+6n & n^2+4n &  n^2+2n & n^2 \\ \hline
        n^2 & n^2+2n & n^2+4n & n^2+6n  & n^2-8n & n^2-6n & n^2-4n & n^2 - 2n \\ \hline
        n^2 -2n & n^2-4n &  n^2-6n & n^2-8n &  n^2+6n & n^2+4n &  n^2+2n & n^2 \\ \hline
        
      \end{array}

    \end{array}
    \]
  }

  Check that in the case of F1, we get

  {\scriptsize
    \[ [d_{i,j}] = 
    \begin{array}{|c|c|c|c|c|c|c|c|c|c|c|c|c|c|c|c|} \hline
      n^2 & n^2+2n & n^2-4n & n^2 - 2n \\ \hline
      n^2 -2n & n^2-4n &   n^2+2n & n^2 \\ \hline
      n^2 & n^2+2n & n^2-4n & n^2 - 2n \\ \hline
      n^2 -2n & n^2-4n &   n^2+2n & n^2 \\ \hline
    \end{array}
    \]
  }
  
  Consequently, when we add all the $d_{i,j}$ along a left pandiagonal,  $P_{1,c}$,  we get the sum to be
  \[ \begin{array}{lllllll}
    \frac{n}{2}n^2 -n^2,  & \mbox{if $c$ is odd, and }, \\ \\
    \frac{n}{2}n^2  & \mbox{if $c$ is even}. 
  \end{array} 
  \]

  Recall that the magic sum $M = \frac{n}{2}(n^2+1)$. The pandiagonal
  sum is the sum of all $s_{i,j}$ along a pandiagonal. Since $s_{i,j} =
  y_{i,j} + d_{i,j}$, for a left pandiagonal $P_{1,c}$ when $c$ is odd,
  the pandiagonal sum add to
  \[
  \frac{n^2}{2} + \frac{n}{2} + \frac{n}{2}n^2 -n^2 = \frac{n}{2}(n^2+1) - \frac{n^2}{2} = M  - \frac{n^2}{2}. 
  \]

  On the other hand, for a left pandiagonal  $P_{1,c}$, when $c$ is even, the pandiagonal sum  add to 
  \[
  \frac{n^2}{2} + \frac{n}{2} + \frac{n}{2}n^2  = \frac{n}{2}(n^2+1) + \frac{n^2}{2} =  M  + \frac{n^2}{2}. 
  \]

  For example, in the case of F1,  the entries in the pandiagonal sum of $P_{1,1}$ is shown in bold below. Observe that the pandiagonal sum is $M-\frac{n^2}{2}$.

  \[\begin{array}{ccccccccccc}
      [y_{i,j}] & & [d_{i,j}] \\ 
      \begin{array}{|c|c|c|c|c|c|c|c|c|c|c|c|c|c|c|c|} \hline
        \bf{\frac{5n}{4}} + 1 & \frac{3n}{4}  + 1    &  \frac{5n}{4} + 1 & \frac{3n}{4} + 1  \\ \hline
        \frac{3n}{4} + 1 & \bf{\frac{5n}{4} + 1} & \frac{3n}{4} +  & \frac{5n}{4} + 1  \\ \hline
        \frac{3n}{4} + 1 &  \frac{5n}{4} + 1 & \bf{\frac{3n}{4} + 1} & \frac{5n}{4} +   \\ \hline
        \frac{5n}{4} + 1 & \frac{3n}{4} + 1 &  \frac{5n}{4} + 1 & \bf{\frac{3n}{4} + 1}  \\ \hline
      \end{array} && 
      \begin{array}{|c|c|c|c|c|c|c|c|c|c|c|c|c|c|c|c|} \hline
        \bf{n^2} & n^2+2n & n^2-4n & n^2 - 2n \\ \hline
        n^2 -2n & \bf{n^2-4n} &   n^2+2n & n^2 \\ \hline
        n^2 & n^2+2n & \bf{n^2-4n} & n^2 - 2n \\ \hline
        n^2 -2n & n^2-4n &   n^2+2n &\bf{ n^2} \\ \hline
      \end{array} \\ \\
      
  \end{array}
  \]

  The entries of the pandiagonal sum of $P_{1,2}$ is shown in bold below. Check that the pandiagonal sum is $M+\frac{n^2}{2}$.
  \[\begin{array}{ccccccccccc}
      [y_{i,j}] & & [d_{i,j}] \\ 
      \begin{array}{|c|c|c|c|c|c|c|c|c|c|c|c|c|c|c|c|} \hline
        \frac{5n}{4} + 1 & \bf{\frac{3n}{4}  + 1}    &  \frac{5n}{4} + 1 & \frac{3n}{4} + 1  \\ \hline
        \frac{3n}{4} + 1 & \frac{5n}{4} + 1 & \bf{\frac{3n}{4} + 1} & \frac{5n}{4} + 1  \\ \hline
        \frac{3n}{4} + 1 &  \frac{5n}{4} + 1 & \frac{3n}{4} + 1 & \bf{\frac{5n}{4} + 1}  \\ \hline
        \bf{\frac{5n}{4}} + 1 & \frac{3n}{4} + 1 &  \frac{5n}{4} + 1 & \frac{3n}{4} + 1  \\ \hline
      \end{array} && 
      \begin{array}{|c|c|c|c|c|c|c|c|c|c|c|c|c|c|c|c|} \hline
        n^2 & \bf{n^2+2n} & n^2-4n & n^2 - 2n \\ \hline
        n^2 -2n & n^2-4n &   \bf{n^2+2n} & n^2 \\ \hline
        n^2 & n^2+2n & n^2-4n & \bf{n^2 - 2n} \\ \hline
        \bf{n^2 -2n} & n^2-4n &   n^2+2n & n^2 \\ \hline
      \end{array} \\ \\
      
  \end{array}
  \]

  The proof is similar for right pandiagonals.  Thus, all the pandigonals add to $M \pm n^2/2$.
\end{proof}

Even though the Franklin square F2 is not constructed using the $N-i$
method, observe that all the pandiagonals of F2, also, add to $M \pm
n^2/2$.


\section{New method to construct Narayana square.} \label{newmethodsection}

In this section, we develop the $N-i$ method to construct Narayana
squares. We then show that the $N-i$ method is the same as the Chadya
and Chadaka method of Section \ref{introsection}. Finally, we modify the $N-i$
method to create a new Narayana square.

We start by dividing the Narayana square in to two sides: the {\em
  left side} consisting of the first $n/2$ columns and the {\em right side}
consisting of the last $n/2$ columns.  Each side is further divided
in to two parts: the {\em Top part} consisting of the first $n/2$ rows, and
the {\em Bottom part} consisting of the last $n/2$ rows.

Let $N=n^2+1$. The strategy is to first place the numbers $i$, where
$i=1,2, \dots, n^2/2$, and then place the numbers $N-i$, such that
all the defining properties of the square are satisfied.

{\em Distance of a column} in a given side is defined as the number of
columns between the given column and the left edge of the
side. Consequently, the distance of the $j$ th column of a side is
$j-1$.

The square is build, partially, two equidistant columns, one from each
side of the square, at a time. Given a pair of equidistant columns, we
denote the column in left side as $C_l$ and the column in the right
side as $C_r$.

Consider a given part with $r$ rows and a {\em starting number}
$A$. There are only two operations involved for such a part.  An {\em
  Up }operation where consecutive numbers from $A$ to $A+r$ are
filled, in consecutive rows, starting from the bottom row of the part,
and column $C_l$, upwards, alternating between the columns $C_l$ and
$C_r$.  The only other operation is the {\em Down operation} where
consecutive numbers from $A$ to $A+r$ are filled, in consecutive rows,
alternating between the columns $C_l$ and $C_r$, starting from the top
row of the part, and column $C_r$, in a downward direction.

For a given pair of equidistant columns, the sequence of operations
depends on the parity of the distance, and is as described below.
\begin{adjustwidth}{-2cm}{-1cm} 
  \[ \begin{array}{ccc}
    \mbox{Even distance} & & \mbox{Odd distance}  \\
    \begin{array}{|c|cc|} \hline
      \mbox{Part}  & \mbox{Top} & \mbox{Bottom} \\ \hline
      \mbox{Operation} &  \mbox{Up} & \mbox{Down}  \\ \hline
      \mbox{Starting Number} & A & A+n/2  \\ \hline
    \end{array}  & & 

    \begin{array}{|c|cc|} \hline
      \mbox{Part}  & \mbox{Bottom} &  \mbox{Top}  \\ \hline
      \mbox{Operation} &  \mbox{Up} & \mbox{Down} \\ \hline
      \mbox{Starting Number} & A & A+n/2 \\ \hline
    \end{array}
  \end{array}
  \]
\end{adjustwidth}

This sequence will place $n$ consecutive numbers in the chosen two
columns. If the distance is $d$, then the starting number for the
entire sequence of operations, is $nd+1$.

Finally, we complete the square by placing the numbers $N-i$, $i=1,
2, \dots n^2/2$ as follows.  At this stage, the odd rows of the left
part and the even rows of the right part are empty. Subtractions from
$N$ occur across diagonal parts. The odd rows of the top left part
are obtained by subtracting the entries from the corresponding cells
of the bottom right part, from $N$. The odd rows of the bottom left
part are obtained by subtracting the entries in  the corresponding
cells in the top right part, from $N$.  Similarly, the even rows of the
top right part is obtained by subtracting the entries in the
corresponding cells in the bottom left part, from $N$. The even rows of
the bottom right part is obtained by subtracting the entries in the
corresponding cells in the top left part, from $N$.

\begin{exm}

  In Table \ref{tablenarayansteps}, the steps of filling the numbers $1$
  to $32$ in the $8 \times 8$ Narayana square N1 are shown. We start
  with the first columns of the left and right side.  The distance of
  this pair from their respective left edges, is zero.  Hence the
  starting number is $nd+1 = 8\times 0 + 1 = 1$.  Here $C_l$ is column 1,
  and $C_r$ is column $5$ of the square. The sequence of operation is
  \[
  \begin{array}{|c|cc|} \hline
    \mbox{Part}  & \mbox{Top} & \mbox{Bottom}\\ \hline
    \mbox{Operation} &  \mbox{Up} & \mbox{Down} \\ \hline
    \mbox{Starting Number} & 1 &  5 \\ \hline
  \end{array}
  \]

  That is, we enter the numbers from $1$ to $4$ in the top part,
  starting from $C_l$, using the Up operation.  Next we enter the
  numbers from $5$ to $8$ in the bottom part, starting from $C_r$, using
  the Down operation. See Step 1 of Table \ref{tablenarayansteps}.

  In Step 2 of Table \ref{tablenarayansteps}, we consider equidistant
  columns of distance $1$. The starting number is $nd+1
  = 8 \times 1 + 1 = 9$.  Since the distance is odd, the sequence of
  operation is
  \[
  \begin{array}{|c|cc|} \hline
    \mbox{Part}  &  \mbox{Bottom}  & \mbox{Top}  \\ \hline
    \mbox{Operation} &  \mbox{Up} & \mbox{Down}  \\ \hline
    \mbox{Starting Number} & 9 &  13  \\ \hline
  \end{array}
  \]
  Thus, the numbers $9$ to $16$ are placed in the two columns using the
  above sequence of operations. Steps 3 and 4 demonstrate the placement
  of numbers from $17$ to $32$ in the rest of the columns of the square
  N1.  See Table \ref{tablenarayan16partial} for the placement of the
  numbers from $1$ to $128$ in the $16 \times 16$ Narayana square N2.

  The filling of the empty cells in square N1 with $N-i$, where $i$ is
  the entry in the corresponding cell in a diagonal part, is given in
  Table \ref{tablenarayanstepfinal}. The final step of filling empty
  cells for the square N2 is given in Table
  \ref{tablenarayan16subtractfromN}.
\end{exm}

\begin{table}{\scriptsize
    \caption{Construction of $8 \times 8$ Narayana square N1: filling the numbers from $1$ to $32$.} 
    \label{tablenarayansteps} 

    \[\begin{array}{ccccccccccccccccccccccccccccc}
    \mbox{Step 1} &&&& \mbox{Step 2}\\ \\

    \begin{array}{cc} 
      \begin{array}{|c|c|c|c|} \hline
        - & - & - & -   \\ \hline
        3 & - & - & -   \\ \hline
        - & - & - & - \\ \hline
        1 & - & - & -   \\ \hline
      \end{array} & \hspace{.13in}
      \begin{array}{|c|c|c|c|} \hline
        4 & - & - & -   \\ \hline
        - & - & - & - \\ \hline
        2 & - & - & -   \\ \hline
        - & - & - & -\\  \hline 
      \end{array}  \\ \\

      \begin{array}{|c|c|c|c|} \hline
        - & - & - & -  \\ \hline
        6 & - & - & -   \\ \hline
        - & - & - & - \\ \hline
        8 & - & - & -   \\ \hline
      \end{array} & \hspace{.13in}
      \begin{array}{|c|c|c|c|} \hline
        5 & - & - & -   \\ \hline
        - & - & - & -   \\ \hline
        7 & - & - & - \\ \hline
        - & - & - & - \\ \hline
      \end{array}
    \end{array}

    & & & &  

    \begin{array}{cc} 
      \begin{array}{|c|c|c|c|} \hline
        - & - & - & -   \\ \hline
        3 & 14& - & -   \\ \hline
        - & - & - & - \\ \hline
        1 & 16 & - & -   \\ \hline
      \end{array} & \hspace{.13in}
      \begin{array}{|c|c|c|c|} \hline
        4 & 13 & - & -   \\ \hline
        - & - & - & - \\ \hline
        2 & 15 &- & -   \\ \hline
        - & - & - & -\\  \hline 
      \end{array}  \\ \\
      \begin{array}{|c|c|c|c|} \hline
        - & - & - & -  \\ \hline
        6 & 11 & - & -   \\ \hline

        - & - & - & - \\ \hline
        8 & 9 & - & -   \\ \hline
      \end{array} & \hspace{.13in}
      \begin{array}{|c|c|c|c|} \hline
        5 & 12 & - & -   \\ \hline
        - & - & - & -   \\ \hline
        7 & 10 & - & - \\ \hline
        - & - & - & - \\ \hline
      \end{array}

    \end{array} \\ \\ 

    \mbox{Step 3} &&&& \mbox{Step 4}\\ \\
    \begin{array}{cc} 
      \begin{array}{|c|c|c|c|} \hline
        - & - & - & -   \\ \hline
        3 & 14& 19 & -   \\ \hline
        - & - & - & - \\ \hline
        1 & 16 & 17 & -   \\ \hline

      \end{array} & \hspace{.13in}
      \begin{array}{|c|c|c|c|} \hline
        4 & 13 & 20 & -   \\ \hline
        - & - & - & - \\ \hline
        2 & 15 & 18 & -   \\ \hline
        - & - & - & -\\  \hline 
      \end{array}  \\ \\

      \begin{array}{|c|c|c|c|} \hline
        - & - & - & -  \\ \hline
        6 & 11 & 22 & -   \\ \hline

        - & - & - & - \\ \hline
        8 & 9 & 24 & -   \\ \hline

      \end{array} & \hspace{.13in}
      \begin{array}{|c|c|c|c|} \hline
        5 & 12 & 21 & -   \\ \hline
        - & - & - & -   \\ \hline
        7 & 10 & 23 & - \\ \hline
        - & - & - & - \\ \hline
      \end{array}
    \end{array}

    & & & &  
    \begin{array}{cc} 

      \begin{array}{|c|c|c|c|} \hline
        - & - & - & -   \\ \hline
        3 & 14 & 19 & 30  \\ \hline
        - & - & - & - \\ \hline
        1 & 16 & 17 & 32\\ \hline 

      \end{array} & \hspace{.13in}
      \begin{array}{|c|c|c|c|} \hline
        4 & 13 & 20 & 29 \\ \hline
        - & - & - & - \\ \hline
        2 & 15 & 18 & 31  \\ \hline
        - & - & - & -\\  \hline 
      \end{array}  \\ \\

      \begin{array}{|c|c|c|c|} \hline
        - & - & - & -  \\ \hline
        6 & 11 & 22 & 27 \\ \hline
        - & - & - & - \\ \hline
        8 & 9 & 24 & 25  \\ \hline
      \end{array} & \hspace{.13in}
      \begin{array}{|c|c|c|c|} \hline
        5 & 12 & 21 & 28 \\ \hline
        - & - & - & - \\ \hline
        7 & 10 & 23 & 26 \\ \hline
        - & - & - & - \\ \hline
      \end{array}
    \end{array}
    \end{array}
    \]
}  \end{table}

\begin{table} {\scriptsize
    \caption{Construction of Narayana square  N1.} 
    \label{tablenarayanstepfinal} 

    \[
    \begin{array}{cc} 

      \begin{array}{|c|c|c|c|} \hline
        N-5 & N-12 & N-21 & N-28   \\ \hline
        3 & 14 & 19 & 30  \\ \hline
        N-7 & N-10 & N-23 & N-26 \\ \hline
        1 & 16 & 17 & 32\\ \hline 

      \end{array} & \hspace{.13in}
      \begin{array}{|c|c|c|c|} \hline
        4 & 13 & 20 & 29 \\ \hline
        N-6 & N-11& N-22 & N-27 \\ \hline
        2 & 15 & 18 & 31  \\ \hline
        N-8 & N-9 & N-24 & N-25\\  \hline 
      \end{array}  \\ \\

      \begin{array}{|c|c|c|c|} \hline
        N-4 & N-13 & N-20 & N-29  \\ \hline
        6 & 11 & 22 & 27 \\ \hline
        N-2 & N-15 & N-18 & N-31 \\ \hline
        8 & 9 & 24 & 25  \\ \hline
      \end{array} & \hspace{.13in}
      \begin{array}{|c|c|c|c|} \hline
        5 & 12 & 21 & 28 \\ \hline
        N-3 & N-14 & N-19 & N-30 \\ \hline
        7 & 10 & 23 & 26 \\ \hline
        N-1 & N-16 & N-17 & N-32 \\ \hline
      \end{array}

    \end{array}
    \]
  }
\end{table}

\begin{table} {\scriptsize
    \caption{Partially filled rows of the $16 \times 16$ Narayana square N2.} 
    \label{tablenarayan16partial} 

    \[
    \begin{array}{cc} 

      \begin{array}{|c|c|c|c|c|c|c|c|} \hline
        - & - & - & - &-& - & -& - \\ \hline
        7 & 26 & 39 & 58 &71& 90 & 103& 122  \\ \hline
        - & - & - & - &-& - & -& -    \\ \hline
        5 & 28 & 37 & 60 &69& 92 & 101& 124  \\ \hline
        - & - & - & - &-& - & -& -    \\ \hline
        3 & 30 & 35 & 62 &67& 94 & 99&126  \\ \hline
        - & - & - & - &-& - & -& -   \\ \hline
        1 & 32 & 33 & 64 &65& 96 & 97&128  \\ \hline
      \end{array} &
      \begin{array}{|c|c|c|c|c|c|c|c|} \hline
        8& 25 & 40 & 57 &72& 89 & 104&121  \\ \hline
        - & - & - & - &-& - & -&-  \\ \hline
        6 & 27 & 38 & 59 &70& 91 & 102&123  \\ \hline
        - & - & - & - &-& - & -&-  \\ \hline
        4 & 29 & 36 & 61 &68& 93 & 100&125  \\ \hline
        - & - & - & - &-& - & -&-  \\ \hline
        2 & 31 & 34 & 63&66& 95 & 98&127  \\ \hline
        - & - & - & - &-& - & -&-  \\ \hline
      \end{array}  \\ \\
      \begin{array}{|c|c|c|c|c|c|c|c|} \hline
        - & - & - & - &-& - & -&-  \\ \hline
        10 & 23 & 42 & 55 &74& 87 & 106&119 \\ \hline
        - & - & - & - &-& - & -&-  \\ \hline
        12 & 21 & 44& 53 &76& 85 & 108&117 \\ \hline
        - & - & - & - &-& - & -&-  \\ \hline
        14& 19& 46 & 51 &78& 83 & 110&115  \\ \hline
        - & - & - & - &-& - & -&-  \\ \hline
        16& 17 & 48 & 49 &80& 81 & 112&113  \\ \hline

      \end{array} & 
      \begin{array}{|c|c|c|c|c|c|c|c|} \hline
        9 & 24& 41 & 56 &73& 88 & 105& 120  \\ \hline
        - & - & - & - &-& - & -&-  \\ \hline
        11 & 22& 43 & 54 &75& 86 & 107&118 \\ \hline
        - & - & - & - &-& - & -&-  \\ \hline
        13 & 20 & 45 & 52 &77& 84 & 109&116 \\ \hline
        - & - & - & - &-& - & -&-  \\ \hline
        15& 18 & 47 & 50 &79&  82& 111&114  \\ \hline
        - & - & - & - &-& - & -&-  \\ \hline
      \end{array} 
    \end{array}
    \]
  }
\end{table}

\begin{table}  { \scriptsize
    \caption{Final step in the construction of the $16 \times 16$ Narayana
      square N2.}
    \label{tablenarayan16subtractfromN} 
    
    \begin{adjustwidth}{-3cm}{-1cm} 
      \[
      \begin{array}{cc} 
        \begin{array}{|c|c|c|c|c|c|c|c|} \hline
          N-9 & N-24& N-41 & N-56 &N-73& N-88 & N-105& N-120  \\ \hline
          7 & 26 & 39 & 58 &71& 90 & 103& 122  \\ \hline
          N-11 & N-22& N-43 & N-54 &N-75& N-86 & N-107&N-118 \\ \hline
          5 & 28 & 37 & 60 &69& 92 & 101& 124  \\ \hline
          N-13 & N-20 & N-45 & N-52 &N-77& N-84 & N-109& N-116 \\ \hline
          3 & 30 & 35 & 62 &67& 94 & 99&126  \\ \hline
          N-15& N-18 & N-47 & N-50 & N-79&  N-82& N-111& N-114  \\ \hline
          1 & 32 & 33 & 64 &65& 96 & 97&128  \\ \hline
        \end{array} & \hspace{.13in}
        \begin{array}{|c|c|c|c|c|c|c|c|} \hline
          8& 25 & 40 & 57 &72& 89 & 104&121  \\ \hline
          N-10 & N-23 & N-42 & N-55 &N-74& N-87 & N-106& N-119 \\ \hline
          6 & 27 & 38 & 59 &70& 91 & 102&123  \\ \hline
          N-12 & N-21 & N-44& N-53 &N-76& N-85 & N-108& N-117 \\ \hline
          4 & 29 & 36 & 61 &68& 93 & 100&125  \\ \hline
          N-14& N-19& N-46 & N-51 &N-78& N-83 & N-110& N-115  \\ \hline
          2 & 31 & 34 & 63&66& 95 & 98&127  \\ \hline
          N-16& N-17 & N-48 & N-49 & N-80& N-81 & N-112& N-113  \\ \hline
        \end{array}  \\ \\
        \begin{array}{|c|c|c|c|c|c|c|c|} \hline
          N-8& N-25 & N-40 & N-57 & N-72& N-89 & N-104& N-121  \\ \hline
          10 & 23 & 42 & 55 &74& 87 & 106&119 \\ \hline
          N-6 & N-27 & N-38 & N-59 & N-70&  N-91 & N-102& N-123  \\ \hline
          12 & 21 & 44& 53 &76& 85 & 108&117 \\ \hline
          N-4 & N-29 & N-36 & N-61 &N-68& N-93 & N-100& N-125  \\ \hline
          14& 19& 46 & 51 &78& 83 & 110&115  \\ \hline
          N-2 & N-31 & N-34 & N-63& N-66& N-95 & N-98& N-127  \\ \hline
          16& 17 & 48 & 49 &80& 81 & 112&113  \\ \hline

        \end{array} & \hspace{.13in}
        \begin{array}{|c|c|c|c|c|c|c|c|} \hline
          9 & 24& 41 & 56 &73& 88 & 105& 120  \\ \hline
          N-7 & N-26 & N-39 & N-58 &N-71& N-90 & N-103& N-122  \\ \hline
          11 & 22& 43 & 54 &75& 86 & 107&118 \\ \hline
          N-5 & N-28 & N-37 & N-60 &N-69& N-92 & N-101& N-124  \\ \hline
          13 & 20 & 45 & 52 &77& 84 & 109&116 \\ \hline
          N-3 & N-30 & N-35 & N-62 &N-67& N-94 & N-99& N-126  \\ \hline
          15& 18 & 47 & 50 &79&  82& 111&114  \\ \hline
          N-1 & N-32 & N-33 & N-64 &N-65& N-96 & N-97& N-128  \\ \hline

        \end{array} 
      \end{array}
      \]
    \end{adjustwidth}
  }
\end{table}

Summarizing, we derive the following algorithm for constructing Narayana squares.

\begin{algo}(Constructing $n \times n$ Narayana squares.) \label{narayanacosntructalgo1}
  \begin{enumerate}
  \item Partial filling of the left side.

    Start with the first two columns of each side, and then work
    outwards two equidistant columns at a time. Fill $n$ numbers at
    every step. We follow the sequence of operations according to the
    parity of the distance of the columns, as explained above. 

  \item \label{subtractnarastep}  Subtractions from $N$ to complete the square.

    Let $a_{i,j}$ denote the entries of the square.  Let $1 \leq i,j
    \leq n/2$.
    
    The empty cells in the top part and left side of the square, that is,
    when $i$ is odd, are given by
    \[
    a_{i,j} = N-a_{\frac{n}{2}+i, \frac{n}{2}+j}.
    \]

    The empty cells in the bottom part and left side of the square, that is,
    when $i$ is odd, are given by
    \[
    a_{n/2+i,j} = N-a_{i, \frac{n}{2}+j}.
    \]

    The empty cells in the top part and right side of the square, that is,
    when $i$ is even, are given by
    \[
    a_{i,n/2+j} = N-a_{\frac{n}{2}+i, j}.
    \]

    The empty cells in the bottom  part and right side of the square, that is,
    when $i$ is even, are given by
    
    \[
    a_{\frac{n}{2}+i, \frac{n}{2}+j} = N- a_{i,j}.
    \]
    
  \end{enumerate}

\end{algo} 

It will be established, soon, that Algorithm
\ref{narayanacosntructalgo1} is the same as Narayana's original
Chadya-Chadaka method of construction. Since this is an ancient, well
known method, Algorithm \ref{narayanacosntructalgo1} needs no
proof. However, proving that the Algorithm works, gives us an
opportunity to explore many interesting properties of the square.

\begin{lemma} \label{lemmanarayproof}  Let $a_{i,j}$ denote the entries of a $n \times n$ Narayana square.  If the number
  of rows or columns of two entries, from a given axis, is the same,
  the entries are called equidistant.
  \begin{enumerate}

  \item  \label{consecuticeincolnaritempart} Pair of entries of adjacent rows in a column add to $N \pm 2i$
    except for the rows $i=n/2$ and $i=n$, as follows.

    Consider a row  $i \in \{1,2, \dots, n\} \setminus
    \{n/2,  n\}$.  In  the  top part of  the
    square, that is, when $1 \leq i < n/2$,  we have
    \[
    a_{i,j}+a_{i+1,j} = \left \{
    \begin{array}{llllllllll}
      N-2i, & \mbox{if $j$ is odd,} \\
      N+2i, & \mbox{if $j$ is even.} \\
    \end{array}
    \right .
    \]

    For the bottom part, the situation is exactly opposite. That is, when
    $n/2 < i < n$, we have
    \[
    a_{i,j}+a_{i+1,j} =  \left \{ \begin{array}{llllllllll}
      N-2i, & \mbox{if $j$ is even,} \\
      N+2i, & \mbox{if $j$ is odd.} \\
    \end{array}
    \right.
    \]
    Finally, we consider rows $n/2$ and $n$.
    \[
    a_{n/2,j}+a_{n/2+1,j} =  \left \{ \begin{array}{llllllllll}
      N-(n/2-1),  & \mbox{if $j$ is odd,} \\
      N+(n/2-1),  & \mbox{if $j$ is even.} \\
    \end{array}
    \right.
    \]

    \[
    a_{n,j}+a_{1,j}
    =  \left \{ \begin{array}{llllllllll}
      N-(n/2-1),  & \mbox{if $j$ is even,} \\
      N+(n/2-1),  & \mbox{if $j$ is odd.} \\
    \end{array}
    \right.
    \]

  \item \label{halfrowitemnara} Let $m=n^2/2+1$, then equidistant entries from center in half rows add to $m$ or
    $2N-m$ as described below.  Consider the left side, that is $1
    \leq j \leq n/2$,
    \[
    a_{i,j} + a_{i,n/2+1-j} =  \left \{ \begin{array}{llll} m,  & \mbox{ if $i$ is even,} \\
      2N-m,  & \mbox{ if $i$ is odd.}

    \end{array} \right.
    \]

    For the right side, where  $n/2 < j \leq n$,
    \[
    a_{i,j} + a_{i,n+1-j} =  \left \{ \begin{array}{llll} m,  & \mbox{ if $i$ is odd,} \\
      2N-m,   & \mbox{ if $i$ is even.}

    \end{array} \right.
    \]

  \item \label{halfcolitemnara}  Equidistant entries from center in half columns add to $N \pm n/2$ as follows. 
    For top part, that is, $1 \leq i \leq n/2$,
    \[
    a_{i,j} + a_{\frac{n}{2}+1-i, j} =  \left \{ \begin{array}{llll} N - \frac{n}{2},  & \mbox{ if $j$ is odd,} \\ \\
      N + \frac{n}{2}, & \mbox{ if $j$ is even.}
    \end{array} \right.
    \]

    For bottom part, that is, $n/2 < i \leq n$,
    \[
    a_{i,j} + a_{\frac{n}{2}+1-i, j} =  \left \{ \begin{array}{llll} N - \frac{n}{2},  & \mbox{ if $j$ is even,} \\ \\
      N + \frac{n}{2}, & \mbox{ if $j$ is odd.}
    \end{array} \right.
    \]

  \item \label{partequhorinaara}  Let $m_i = n/2-1-2(i-1)$,  for  $1 \leq i \leq n/2$.

    Equidistant entries across the horizontal axis add to $N \pm m_i$ as follows.

    When $1 \leq i \leq n/4$,

    \[a_{i,j} + a_{n+1-i,j}  = \left \{ \begin{array}{llllllllll} N + m_i,  & \mbox{ for $j$ odd,  }  \\ \\
      N - m_i, & \mbox{ for $j$ even.   }
    \end{array} \right.
    \]

    When $n/4+1 \leq i \leq  n/2$,

    \[a_{i,j} + a_{n+1-i,j}  = \left \{ \begin{array}{llllllllll} N - m_i,  & \mbox{ for $j$ odd,  }  \\ \\
      N + m_i, & \mbox{ for $j$ even.   }
    \end{array} \right.
    \]

  \item  \label{partequvertinara} Let $m_j= n/2-1-2(j-1)$, where
    $1 \leq j \leq n/2$. Equidistant entries across the vertical axis add to $N \pm m_jn$, as shown below.

    When   $1 \leq j \leq n/4$,
    \[a_{i,j} + a_{i,n+1-j} = \left \{ \begin{array}{llllllllll} N + m_jn,  & \mbox{ for $i$ odd,  }  \\ \\
      N - m_jn, & \mbox{ for $i$ even.   }
    \end{array} \right.
    \]

    When $n/4+1 \leq  j \leq n/2$,

    \[a_{i,j} + a_{i,n+1-j} = \left \{ \begin{array}{llllllllll} N - m_jn,  & \mbox{ for $i$ odd,  }  \\ \\
      N + m_jn, & \mbox{ for $i$ even.   }
    \end{array} \right.
    \]
  \end{enumerate}
\end{lemma}

\begin{proof}
  The square inherits these properties by construction.
\end{proof} 

\begin{corollary} \label{halfrowsumcor}
  Consider an $n \times n$ Narayana square.  Let $M$ denote the magic
  sum and let $m=n^2/2+1$. Then,
  \begin{enumerate}
  \item  Half row sums add either to $(n/4)m$ or $M-(n/4)m$.
  \item  Half column sums add to  $M/2 \pm n^2/8$.
  \end{enumerate}

\end{corollary}

\begin{proof}
  \begin{enumerate}

  \item By Part \ref{halfrowitemnara} of Lemma \ref{lemmanarayproof},
    adding the $n/4$ equidistant pairs in half rows, we get
    \[ \begin{array}{lllllllllll}
      a_{i,1} + a_{i,2} + \cdots a_{i,\frac{n}{2}} = \left \{ \begin{array}{llllllllllll}
        M-(\frac{n}{4})m, & \mbox{ if $i$ is odd,} \\ \\
        (\frac{n}{4})m, & \mbox{ if $i$ is even.}
      \end{array} 
      \right . \\ \\
      a_{i,\frac{n}{2}+1} + a_{i,\frac{n}{2}+2} + \cdots a_{i,n} = \left \{ \begin{array}{llllllllllll}
        M-(\frac{n}{4})m, & \mbox{ if $i$ is even,} \\ \\
        (\frac{n}{4})m, & \mbox{ if $i$ is odd.}
      \end{array} 
      \right .
    \end{array} \]

    Consequently, for odd rows, left half row sums add  to  $M-(n/4)m$,  and right half row sums add  to $(n/4)m$.
    On the other hand, for even rows, left half row sums add  to $(n/4)m$,  and right half row sums add  to $M-(n/4)m$.

  \item By Part \ref{halfcolitemnara} of Lemma \ref{lemmanarayproof}, when $j$ is odd, and $1 \leq i \leq n/2$,
    \[ a_{i,j} + a_{\frac{n}{2}+1-i, j} = \left \{ \begin{array}{llllll} 
      N - \frac{n}{2},  & 1 \leq i \leq \frac{n}{2} \\ \\
      N + \frac{n}{2}, &  \frac{n}{2} < i \leq n.
    \end{array}
    \right .
    \]

    Therefore, for odd $j$, the top half columns add  to 
    \[
    \frac{n}{4} \left ( N - \frac{n}{2} \right ) =  \frac{M}{2} - \frac{n^2}{8},
    \]

    and  the bottom  half columns add  to $M/2 + n^2/8$.  

    Similarly, by Part \ref{halfcolitemnara} of Lemma
    \ref{lemmanarayproof}, we get, when $j$ is even, the top half columns
    add to $M/2 + n^2/8$, and the bottom half columns add to $M/2 -
    n^2/8$.

    Thus, half column sums add to $M/2 \pm n^2/8$.

  \end{enumerate}
\end{proof}

\begin{prop}
  Algorithm \ref{narayanacosntructalgo1} produces a Narayana square.
\end{prop}

\begin{proof}
  Let $a_{i,j}$ denote the entries of a $n \times n$ square constructed by
  Algorithm \ref{narayanacosntructalgo1}.
  \begin{enumerate}

  \item $2 \times 2$ sub-square sums.

    Consider a row $i \in \{1,2, \dots, n\} \setminus \{n/2, n\}$.

    By Part \ref{consecuticeincolnaritempart} of Lemma
    \ref{lemmanarayproof},
    \[
    \mbox{ if } a_{i,j} + a_{i+1,j} = N+2i, \mbox{ then } a_{i,j+1} + a_{i+1,j+1} = N-2i.
    \]
    On the other hand, 
    \[
    \mbox{ if } a_{i,j} + a_{i+1,j} = N-2i, \mbox{ then } a_{i,j+1} + a_{i+1,j+1} = N+2i.
    \]
    Consequently, for all $i \in \{1,2, \dots, n\} \setminus \{n/2, n\}$ and all $j$, 
    \[a_{i,j}+a_{i+1,j} + a_{i+1,j}+ a_{i+1,j+1} = 2N.\]

    Now we consider the row $n/2$.   By Part \ref{consecuticeincolnaritempart} of Lemma
    \ref{lemmanarayproof}, we get

    \[ \begin{array}{lllllllll}
      \mbox{ if } a_{\frac{n}{2},j} + a_{\frac{n}{2}+1,j} = N+(\frac{n}{2}-1), \mbox{ then } a_{\frac{n}{2},j+1} + a_{\frac{n}{2},j+1} =  N-(\frac{n}{2}-1),\mbox{ and } \\ \\
      \mbox{ if } a_{\frac{n}{2},j} + a_{\frac{n}{2}+1,j} = N-(\frac{n}{2}-1), \mbox{ then } a_{\frac{n}{2},j+1} + a_{\frac{n}{2},j+1} =  N+(\frac{n}{2}-1)  \end{array} 
    \]
    
    Consequently, all the $2 \times 2$ sub-squares, within the Narayana
    square, add to $2N$. Next, we verify the continuity of this property.

    By Part \ref{partequhorinaara} of Lemma
    \ref{lemmanarayproof},
    \[\begin{array}{lllllllll}
    \mbox{ if } a_{n,j} + a_{1,j} = N+(\frac{n}{2}-1) \mbox{ then } a_{n,j+1} + a_{n,j+1} = N-(\frac{n}{2}-1) \mbox{ and } \\ \\
    \mbox{ if } a_{n,j} + a_{1,j} = N-(\frac{n}{2}-1) \mbox{ then } a_{n,j+1} + a_{n,j+1} = N+(\frac{n}{2}-1).
    \end{array} 
    \]

    Consequently, the $2 \times 2$ sub-squares formed by rows $1$ and
    $n$ add to $2N$.  Part \ref{partequvertinara} of Lemma
    \ref{lemmanarayproof} implies

    \[ a_{i,1}+a_{i,n} =  \left \{ \begin{array}{llllllll}
      N+(\frac{n}{2}-1)n, & \mbox{ if $i$ is even, } \\ \\
      N-(\frac{n}{2}-1)n,  & \mbox{ if $i$ is odd. }
    \end{array}
    \right.
    \]
    Thus, the $2 \times 2$ sub-squares formed by columns $1$ and $n$ add
    to $2N$. This proves that the continuity property of $2 \times 2$
    sub-squares hold for squares constructed by Algorithm
    \ref{narayanacosntructalgo1}.

  \item Row and column sums. 

    By Corollary \ref{halfrowsumcor}, when $i$ is odd, left half row
    sums add to $M-(n/4)m$, and right half row sums add to
    $(n/4)m$. Consequently, when $i$ is odd the $i$-th row sum is
    $M$. Similarly, applying Corollary \ref{halfrowsumcor}, we see
    that necessary cancellations happen, when we add half row sums and
    half column sums, to form row sums and column sums,
    respectively. Consequently, row and column sums add to $M$.

  \item Pandiagonal sums.

    Let $a_{i,j}$ belong to a pandiagonal. Then, as we saw in Lemma \ref{franklinpandiaohlemma}, if $1 \leq j
    \leq n/2$, then $a_{n/2+i, n/2+j}$ also belongs to the pandiagonal.
    On the other hand if $j > n/2$, then $a_{n/2+i, j-n/2}$ belongs to
    the pandiagonal.   That is, every pandiagonal is made up of $n/2$ paired entries.

    Consider $1 \leq i, j \leq n/2$.  By Step
    \ref{subtractnarastep} of Algorithm \ref{narayanacosntructalgo1}, we
    get 
    \[ \begin{array}{llllllllll}
      a_{i,j} + a_{\frac{n}{2}+i, \frac{n}{2}+j} = N, \\
      a_{i, \frac{n}{2}+j}+ a_{\frac{n}{2}+i,j} = N. 
    \end{array}
    \]
 
    Consequently,  every pandiagonal sum  adds to $(n/2)N = M$.

  \end{enumerate}
  
Thus, Algorithm \ref{narayanacosntructalgo1} produces a Narayana square.

\end{proof}

\begin{lemma}
  Left and right  bend diagonal sums add to  $ M \pm n/2$. Top and bottom  bend diagonal sums add to  $ M \pm n^2/2$.
\end{lemma}
\begin{proof}

  Let $1 \leq j \leq n/2+1$, then, if $j$ is odd, then  by  Part \ref{partequhorinaara}  of Lemma
  \ref{lemmanarayproof},   the left
  bend diagonal sum starting with row $1$ and column $j$ adds as
  follows.
  \[ \begin{array}{llllllllllllll}
    \left [ (a_{1,j} + a_{n,j}) + (a_{2, j+1} + a_{n-1, j+1}) +  \cdots + (a_{\frac{n}{4}, j+\frac{n}{4}-1} + a_{\frac{3n}{4}+1, j+\frac{n}{4}-1}) \right ]\\ \\
    + \left [(a_{\frac{n}{4}+1, j+\frac{n}{4}}+a_{\frac{3n}{4}, j+\frac{n}{4}}) + \cdots + (a_{n/2,j+n/2-1}  + a_{n/2+1,j+n/2-1}) \right ]\\ \\
    =   \left [\left( N+(\frac{n}{2} - 1 ) \right ) +  \left ( N -  (\frac{n}{2} - 3 ) \right ) + \left  ( N +  (\frac{n}{2} - 5 ) \right ) + \cdots + (N+3) +  (N-1)\right ] \\ \\
    + \left [ (N-1) + (N+3) + \cdots +  \left (N -  (\frac{n}{2} - 3)\right ) + \left ( N+(\frac{n}{2} - 1 ) \right )
      \right ] \\ \\
    = 2  \left [(N+(\frac{n}{2} - 1 )) +  (N -  (\frac{n}{2} - 3 )) +  (N+(\frac{n}{2} - 5 )) + \cdots + (N+3) +  (N-1)\right ] \\ \\
    = 2 \left [ \frac{n}{4} N  +  ( \frac{n}{2} - 1 -  \frac{n}{2} + 3  + \frac{n}{2} - 5  + \cdots + 3  -1)  \right ] \\ \\
    = 2 \left [ \frac{n}{4} N + \frac{n}{8} \times 2 \right ]

    \\ \\  = \frac{n}{2}N + \frac{n}{2}  =  M + \frac{n}{2}. 
  \end{array} 
  \]

  A similar argument gives us that when $j$ is even, the  left  bend diagonals add to $M - n/2$.

  Now,  let  $n/2+1 < j \leq n$, the left bend diagonal sums are 
  \[ \begin{array}{llllllllllllll}
    (a_{1,j} + a_{n,j}) +  (a_{2,j} + a_{n-1, j+1})+ \cdots +  (a_{n-j+1, n} + a_{j,n}) \\ \\
    + (a_{n-j+2, 1} + a_{j-1,1}) + (a_{n-j+3, 2} + a_{j-2,2})+ \cdots +  (a_{\frac{n}{2}, j-\frac{n}{2}-1}+a_{\frac{n}{2}+1, j-\frac{n}{2}-1} ).
  \end{array} 
  \]
  Check that,  by  Part \ref{partequhorinaara}  of Lemma
  \ref{lemmanarayproof}, we get that these sums also add to $M \pm n/2$.

  For example, in the case of of N1 (see Table  \ref{tablenarayanstepfinal}), the seventh bend diagonal sum is 
  \[\begin{array}{llllllllllllll}
  (a_{1,7}+a_{8,7}) + (a_{2,8}+a_{7,8})+(a_{3,1}+a_{6,1})+ (a_{4,2}+a_{5,2}) \\ \\
  = (N+3)+(N-1)+(N-1) + (N+3) = 4N+4  = M + \frac{n}{2}.
  \end{array} 
  \]

  Thus, the left bend diagonals add to $M \pm n/2$, continuously. The
  proof that all right bend diagonals add to $M \pm n/2$, is similar
  to the case of left bend diagonals. The proof depends, mainly, on
  the fact that equidistant entries across the horizontal axis add to
  either $N+m_i$ or $N -m_i$, and all $m_i$ cancel in the final sum.

  Equidistant entries across the vertical axis is used to prove that
  the top and bottom bend diagonals add to magic sum. By Part \ref{partequvertinara}
  of Lemma \ref{lemmanarayproof}, pairs of equidistant
  entries across the vertical axis add to $N-m_in$ or $N+m_in$.

  For $1 \leq i \leq n/2+1$, let $i$ be odd, then the $i$-th top bend diagonal sum is given below.
  \[\begin{array}{llllllllll}  
  \left [ (a_{i,1} + a_{i,n})+ (a_{i+1,2} + a_{i+1,n-1})+ \cdots + (a_{i+\frac{n}{4}-1, \frac{n}{4}}+a_{i+\frac{n}{4}-1, \frac{3n}{4}+1}) \right ]  \\ \\
  + \left [ (a_{i+\frac{n}{4}, \frac{n}{4}+1} + a_{i+\frac{n}{4}, \frac{3n}{4}} ) + \cdots + (a_{i+\frac{n}{2}-1, \frac{n}{2}} + a_{i+\frac{n}{2}, \frac{n}{2}+1 }) \right ]  \\ \\
  = \left [(N + (\frac{n}{2}-1)n) + (N -(\frac{n}{2}-3)n)   + \cdots  (N+3n) + (N-n) \right ]  \\ \\
  + \left [ (N-n) + (N+3n) + (N -(\frac{n}{2}-3)n) + (\frac{n}{2}-1)n) + (N + (\frac{n}{2}-1)n)\right ]\\ \\
  = 2 \left [ \frac{n}{4} N + \frac{n}{8}(2n) \right ] \\ \\
  = \frac{n}{2}N  + \frac{n^2}{2} = M + \frac{n^2}{2}.
  \end{array}
  \]
  When $i$ is even, it can be checked that the the $i$-th top bend diagonal sum is $M - n^2/2$.
  Thus, the top bend diagonal sums add to $M \pm n^2/2$ when $1 \leq i \leq
  n/2+1$.

  For $n/2 + 1 < i \leq n$, the top bed diagonal sums are 
  \[\begin{array}{llllllllll}  
  (a_{i,1} + a_{i,n})+ (a_{i+1,2} + a_{i+1,n-1})+ \cdots (a_{n,n-i+1} + a_{n,i} )  \\ \\
  +(a_{1, n-i+2}+a_{1, i-1}) + (a_{2, n-i+3}+a_{2, i-2})+ \cdots +(a_{i-\frac{n}{2}-1, \frac{n}{2}}+ a_{i-\frac{n}{2}-1, \frac{n}{2}+1})
  \end{array}
  \]

  Again, by Part \ref{partequvertinara}
  of Lemma \ref{lemmanarayproof}, it
  can be checked that the top bed diagonal sums add to $M \pm n^2/2$.

  For example, in the case of N1 (see Table
  \ref{tablenarayanstepfinal}), the seventh top bend diagonal sum is
  \[\begin{array}{llllllllll}  
  (a_{7,1}+a_{7,8}) + (a_{8,2}+a_{8,7}) + (a_{1,3}+a_{1,6}) + (a_{2,4}+a_{2,5}) \\ \\
  = (N+3n)+ (N-n) + (N-n)+ (N+3n) = 4N + 4n = M+n^2/2. 
  \end{array}
  \]

  Consequently, the top bend diagonals add to $ M \pm n^2/2$,
  continuously. A similar proof, applying Part \ref{partequvertinara}
  of Lemma \ref{lemmanarayproof}, shows that the bottom bend diagonals
  add to $M \pm n^/2$, continuously.

\end{proof}

We proceed to show that the $N-i$ method is the same as Narayana
Pandit's Chadya-Chadaka method described in Section
\ref{introsection}.  We use the example of the Narayana square N2 to
demonstrate our derivation. This process is very similar to the
derivation of the Chadya and Chadaka squares of Franklin squares in
Section \ref{franklinsection}. Since $16$ numbers are filled at
every step, Table \ref{tablenarayan16partial} can be rewritten in
terms of multiples of $16$ as shown below.

{\scriptsize
  \begin{adjustwidth}{-3cm}{-1cm} 
    \[
    \begin{array}{cc} 

      \begin{array}{|c|c|c|c|c|c|c|c|} \hline
        - & - & - & - &-& - & -& - \\ \hline
        7 & 10+16 & 7+32 & 10+48 &7+64& 10+80 & 7+96& 10+112  \\ \hline
        - & - & - & - &-& - & -& -    \\ \hline
        5 & 12+16 & 5+32 & 12+48 &5+64& 12+80 & 5+96& 12+112  \\ \hline
        - & - & - & - &-& - & -& -    \\ \hline
        3 & 14+16 & 3+32 & 14+48 &3+64& 14+80 & 3+96& 14+112  \\ \hline
        - & - & - & - &-& - & -& -   \\ \hline
        1 & 16+16 & 1+32 & 16+48 &1+64& 16+80 & 1+96& 16+112  \\ \hline
        
      \end{array} & \hspace{.13in}
      \begin{array}{|c|c|c|c|c|c|c|c|} \hline
        8 & 9+16 & 8+32 & 9+48 &8+64& 9+80 & 8+96& 9+112  \\ \hline
        - & - & - & - &-& - & -&-  \\ \hline
        6 & 11+16 & 6+32 & 11+48 &6+64& 11+80 & 6+96& 11+112  \\ \hline
        - & - & - & - &-& - & -&-  \\ \hline
        4 & 13+16 & 4+32 & 13+48 &4+64& 13+80 & 4+96& 13+112  \\ \hline
        - & - & - & - &-& - & -&-  \\ \hline
        2 & 15+16 & 2+32 & 13+48 &2+64& 15+80 & 2+96& 15+112  \\ \hline
        - & - & - & - &-& - & -&-  \\ \hline
      \end{array}  \\ \\
      \begin{array}{|c|c|c|c|c|c|c|c|} \hline
        - & - & - & - &-& - & -&-  \\ \hline
        10 & 7+16 & 10+32 & 7+48 &10+64& 7+80 & 10+96& 7+112  \\ \hline
        - & - & - & - &-& - & -&-  \\ \hline
        12 & 5+16 & 12+32 & +5+48 &12+64& 5+80 & 12+96& 5+112  \\ \hline    
        - & - & - & - &-& - & -&-  \\ \hline
        14 & 3+16 & 14+32 & 3+48 &14+64& 3+80 & 14+96& 3+112  \\ \hline
        - & - & - & - &-& - & -&-  \\ \hline
        16 & 1+16 & 16+32 & 1+48 &16+64& 1+80 & 16+96& 1+112  \\ \hline
      \end{array} & \hspace{.13in}

      \begin{array}{|c|c|c|c|c|c|c|c|} \hline
        9 & 8+16 & 9+32 & 8+48 &9+64& 8+80 & 9+96& 8+112  \\ \hline
        - & - & - & - &-& - & -&-  \\ \hline
        11 & 6+16 & 11+32 & 6+ 48 & 11+64&  6+80 & 11+96& 6+112  \\ \hline    
        - & - & - & - &-& - & -&-  \\ \hline
        13 & 4+16 & 13+32 & 4+48 &13+64& 4+80 & 13+96& 4+112  \\ \hline    
        - & - & - & - &-& - & -&-  \\ \hline
        15 & 2+16 & 15+32 & 2+48 &15+64& 2+80 & 15+96& 2+112  \\ \hline 
        - & - & - & - &-& - & -&-  \\ \hline
      \end{array} 
    \end{array}
    \]
  \end{adjustwidth}
}

Since $N-i-rn = (n^2+1)-rn-i = (n^2-(
r+1)n) +(n+1-i)$, Table  \ref{tablenarayan16subtractfromN} becomes 

{\scriptsize
  \begin{adjustwidth}{-2.8cm}{-1cm} 
    \[
    \begin{array}{cc} 

      \begin{array}{|c|c|c|c|c|c|c|c|} \hline
        8+240 & 9+224 & 8+208 & 9+192 &8+176& 9+160 & 8+144& 9+128 \\ \hline
        7 & 10+16 & 7+32 & 10+48 &7+64& 10+80 & 7+96& 10+112  \\ \hline
        6+240 & 11+224 & 6+208 & 11+192 &6+176& 11+160 & 6+144& 6+128 \\ \hline 
        5 & 12+16 & 5+32 & 12+48 &5+64& 12+80 & 5+96& 12+112  \\ \hline
        4+240 & 13+224 & 4+208 & 13+192 &4+176& 13+160 & 4+144& 13+128\\ \hline
        3 & 14+16 & 3+32 & 14+48 &3+64& 14+80 & 3+96& 14+112  \\ \hline
        2+240 & 15 +224 & 2+208 & 15+192 &2+176& 15+160 & 2+144& 15+128 \\ \hline
        1 & 16+16 & 1+32 & 16+48 &1+64& 16+80 & 1+96& 16+112  \\ \hline
      \end{array} &
      \begin{array}{|c|c|c|c|c|c|c|c|} \hline
        8 & 9+16 & 8+32 & 9+48 &8+64& 9+80 & 8+96& 9+112  \\ \hline
        7+240 & 10+224 & 7+208 & 10+192 &7+176& 10+160 & 7+144& 10+128 \\ \hline
        6 & 11+16 & 6+32 & 11+48 &6+64& 11+80 & 6+96& 11+112  \\ \hline
        5+240 & 12+224 & 5+208 & 12+192 &5+176& 12+160 & 5+144& 12+128 \\ \hline
        4 & 13+16 & 4+32 & 13+48 &4+64& 13+80 & 4+96& 13+112  \\ \hline
        3+240 & 14+224 & 3+208 & 14+192 &3+176& 14+160 & 3+144& 14+128 \\ \hline
        2 & 15+16 & 2+32 & 13+48 &2+64& 15+80 & 2+96& 15+112  \\ \hline
        1+240 & 16+224 & 1+208 & 16+192 &1+176& 16+160 & 1+144& 16+128 \\ \hline
      \end{array}  \\ \\
      \begin{array}{|c|c|c|c|c|c|c|c|} \hline
        9+240 & 8+224 & 9+208 & 8+192 &9+176& 8+160 & 9+144& 8+128 \\ \hline
        10 & 7+16 & 10+32 & 7+48 &10+64& 7+80 & 10+96& 7+112  \\ \hline
        11+240 & 6+224 & 11+208 & 6+192 &11+176& 6+160 & 11+144& 6+128 \\ \hline
        12 & 5+16 & 12+32 & +5+48 &12+64& 5+80 & 12+96& 5+112  \\ \hline    
        13+240 & 4+224 & 13+208 & 4+192 &13+176& 4+160 & 13+144& 4+128 \\ \hline
        14 & 3+16 & 14+32 & 3+48 &14+64& 3+80 & 14+96& 3+112  \\ \hline
        15+240 & 2+224 & 15+208 & 2+192 &15+176& 2+160 & 15+144& 2+128 \\ \hline
        16 & 1+16 & 16+32 & 1+48 &16+64& 1+80 & 16+96& 1+112  \\ \hline
      \end{array} & 

      \begin{array}{|c|c|c|c|c|c|c|c|} \hline
        9 & 8+16 & 9+32 & 8+48 &9+64& 8+80 & 9+96& 8+112  \\ \hline
        10+240 & 7+224 & 10+208 & 7+192 &10+176& 7+160 & 10+144& 7+128 \\ \hline
        11 & 6+16 & 11+32 & 6+ 48 & 11+64&  6+80 & 11+96& 6+112  \\ \hline    
        12+240 & 5+224 & 12+208 & 5+192 &12+176& 5+160 & 12+144& 5+128 \\ \hline
        13 & 4+16 & 13+32 & 4+48 &13+64& 4+80 & 13+96& 4+112  \\ \hline    
        14+240 & 3+224 & 14+208 & 3+192 &14+176& 3+160 & 14+144& 3+128 \\ \hline
        15 & 2+16 & 15+32 & 2+48 &15+64& 2+80 & 15+96& 2+112  \\ \hline 
        16+240 & 1+224 & 16+208 & 1+192 &16+176& 1+160 & 16+144& 1+128 \\ \hline
      \end{array} 
    \end{array}
    \]
  \end{adjustwidth}
}

Consequently, the square can be split as the Chadya and flipped
Chadaka of N2.

{\scriptsize
  \begin{adjustwidth}{-3cm}{-1cm} 
    \[    
    \begin{array}{cccccccc}
      \mbox{Chadya} && \mbox{Flipped Chadaka} \\ \\
      \begin{array}{|c|c|c|c|c|c|c|c|c|c|c|c|c|c|c|c|} \hline
        8 &  9 & 8 & 9 & 8 & 9 & 8 & 9 & 8 & 9 & 8 & 9 & 8 & 9 & 8 & 9\\  \hline
        7 & 10 & 7 & 10 &7 & 10 &7 & 10 &7 & 10 &7 & 10 &7 & 10 & 7 & 10 \\ \hline
        6 & 11 &6 & 11 &6 & 11 &6 & 11 &6 & 11 &6 & 11 &6 & 11 &6 & 11  \\ \hline
        5 & 12 &5 & 12 &5 & 12 &5 & 12 &5 & 12 &5 & 12 &5 & 12 &5 & 12  \\ \hline
        4 & 13 &4 & 13 &4 & 13 &4 & 13 &4 & 13 &4 & 13 &4 & 13 &4 & 13  \\ \hline
        3 & 14 &3 & 14 &3 & 14 &3 & 14 &3 & 14 &3 & 14 &3 & 14 &3 & 14  \\ \hline
        2 & 15 &2 & 15 & 2 & 15 &2 & 15 & 2 & 15 &2 & 15 & 2 & 15 &2 & 15 \\ \hline
        1 & 16 &1 & 16 &1 & 16 &1 & 16 &1 & 16 &1 & 16 &1 & 16 &1 & 16  \\ \hline
        9 &  8 & 9 &  8 & 9 &  8 & 9 &  8 & 9 &  8 & 9 &  8 &9 &  8 & 9 & 8 \\ \hline
        10 & 7 &10 & 7 &10 & 7 &10 & 7 &10 & 7 &10 & 7 & 10 & 7 & 10 & 7  \\ \hline
        11 & 6 & 11 & 6 &11 & 6 & 11 & 6 &11 & 6 & 11 & 6 &11 & 6 & 11 & 6   \\ \hline
        12& 5 &12& 5 &12& 5 &12& 5 &12& 5 &12& 5 &12& 5 &12& 5  \\ \hline
        13& 4 &13& 4 &13& 4 &13& 4 &13& 4 & 13& 4 &13& 4 &13& 4  \\ \hline
        14& 3  &14& 3  &14& 3  &14& 3  &14& 3  &14& 3  &14& 3  &14& 3   \\ \hline
        15& 2 &15& 2 &15& 2 &15& 2 &15& 2 &15& 2 & 15& 2 &15& 2 \\ \hline
        16 & 1 & 16 & 1 &16 & 1 & 16 & 1 &16 & 1 & 16 & 1 &16 & 1 & 16 & 1  \\ \hline
      \end{array} 
      &+&
      
      \begin{array}{|c|c|c|c|c|c|c|c|c|c|c|c|c|c|c|c|} \hline
        240 & 224 & 208 & 192 & 176 & 160 & 144 & 128 & 0 & 16 & 32 & 48 & 64 & 80 & 96 & 112 \\ \hline
        0 & 16 & 32 & 48 & 64 & 80 & 96 & 112 & 240 & 224 & 208 & 192 & 176 & 160 & 144 & 128 \\ \hline
        240 & 224 & 208 & 192 & 176 & 160 & 144 & 128 & 0 & 16 & 32 & 48 & 64 & 80 & 96 & 112 \\ \hline
        0 & 16 & 32 & 48 & 64 & 80 & 96 & 112 & 240 & 224 & 208 & 192 & 176 & 160 & 144 & 128 \\ \hline
        240 & 224 & 208 & 192 & 176 & 160 & 144 & 128 & 0 & 16 & 32 & 48 & 64 & 80 & 96 & 112 \\ \hline
        0 & 16 & 32 & 48 & 64 & 80 & 96 & 112 & 240 & 224 & 208 & 192 & 176 & 160 & 144 & 128 \\ \hline
        240 & 224 & 208 & 192 & 176 & 160 & 144 & 128 & 0 & 16 & 32 & 48 & 64 & 80 & 96 & 112 \\ \hline
        0 & 16 & 32 & 48 & 64 & 80 & 96 & 112 & 240 & 224 & 208 & 192 & 176 & 160 & 144 & 128 \\ \hline
        240 & 224 & 208 & 192 & 176 & 160 & 144 & 128 & 0 & 16 & 32 & 48 & 64 & 80 & 96 & 112 \\ \hline
        0 & 16 & 32 & 48 & 64 & 80 & 96 & 112 & 240 & 224 & 208 & 192 & 176 & 160 & 144 & 128 \\ \hline
        240 & 224 & 208 & 192 & 176 & 160 & 144 & 128 & 0 & 16 & 32 & 48 & 64 & 80 & 96 & 112 \\ \hline
        0 & 16 & 32 & 48 & 64 & 80 & 96 & 112 & 240 & 224 & 208 & 192 & 176 & 160 & 144 & 128 \\ \hline
        240 & 224 & 208 & 192 & 176 & 160 & 144 & 128 & 0 & 16 & 32 & 48 & 64 & 80 & 96 & 112 \\ \hline
        0 & 16 & 32 & 48 & 64 & 80 & 96 & 112 & 240 & 224 & 208 & 192 & 176 & 160 & 144 & 128 \\ \hline
        240 & 224 & 208 & 192 & 176 & 160 & 144 & 128 & 0 & 16 & 32 & 48 & 64 & 80 & 96 & 112 \\ \hline
        0 & 16 & 32 & 48 & 64 & 80 & 96 & 112 & 240 & 224 & 208 & 192 & 176 & 160 & 144 & 128 \\ \hline
      \end{array} 

    \end{array}
    \]
  \end{adjustwidth}
}

Thus, the $N-i$ method is the same as the original Chadya-Chadaka
method of Narayana (see Example \ref{16x16narayannamaste}). However, the $N-i$ method can be easily modified
to create new Narayana squares.

\begin{exm} Constructing New Narayana Square.

  We first enter the numbers $1$ to $32$ as shown below. We begin in the
  second column of the last row of the square.  This step is slightly
  different from Step 1 in Algorithm \ref{narayanacosntructalgo1}.

  {\scriptsize
    \[
    \begin{array}{cc} 

      \begin{array}{|c|c|c|c|} \hline
        - & - & - & -   \\ \hline
        27 & 6 & 11 & 22   \\ \hline
        - & - & - & - \\ \hline
        25 & 8 & 9 & 24   \\ \hline

      \end{array} & \hspace{.13in}
      \begin{array}{|c|c|c|c|} \hline
        28& 5 & 12& 21   \\ \hline
        - & - & - & - \\ \hline
        26& 7 & 10 & 23   \\ \hline
        - & - & - & -\\  \hline 
      \end{array}  \\ \\

      \begin{array}{|c|c|c|c|} \hline
        - & - & - & -  \\ \hline
        30 & 3 & 14& 19  \\ \hline

        - & - & - & - \\ \hline
        32 & 1 & 16 & 17   \\ \hline

      \end{array} & \hspace{.13in}
      \begin{array}{|c|c|c|c|} \hline
        29 & 4 & 13& 20   \\ \hline
        - & - & - & -   \\ \hline
        31 & 2 & 15 & 18\\ \hline
        - & - & - & - \\ \hline
      \end{array}

    \end{array}
    \]
  }

  Next we do the necessary subtractions from $N$. This  step is the same as in Algorithm \ref{narayanacosntructalgo1}.

  {\scriptsize
    \[
    \begin{array}{cc} 

      \begin{array}{|c|c|c|c|} \hline
        N-29 & N-4 & N-13& N-20   \\ \hline
        27 & 6 & 11 & 22   \\ \hline
        N-31 & N-2 & N-15 & N-18\\ \hline
        25 & 8 & 9 & 24   \\ \hline

      \end{array} & \hspace{.13in}
      \begin{array}{|c|c|c|c|} \hline
        28& 5 & 12& 21   \\ \hline
        N-30 & N-3 & N-14& N-19  \\ \hline
        26& 7 & 10 & 23   \\ \hline
        N-32 & N-1 & N-16 & N-17   \\ \hline
        
      \end{array}  \\ \\

      \begin{array}{|c|c|c|c|} \hline
        N-28& N-5 & N-12& N-21   \\ \hline
        30 & 3 & 14& 19  \\ \hline
        N-26& N-7 & N-10 & N-23   \\ \hline
        32 & 1 & 16 & 17   \\ \hline

      \end{array} & \hspace{.13in}
      \begin{array}{|c|c|c|c|} \hline
        29 & 4 & 13& 20   \\ \hline
        N-27 & N-6 & N-11 & N-22   \\ \hline
        31 & 2 & 15 & 18\\ \hline
        N-25 & N-8 & N-9 & N-24   \\ \hline
      \end{array}
    \end{array}
    \]
  }

  Thus, we get a new square which can be checked to be a new Narayana square.
  {\scriptsize
    \[
    \begin{array}{ccccccc}
      \mbox{New Narayana square} \\ \\
      \begin{array}{|c|c|c|c|c|c|c|c|} \hline
        36 & 61 & 52 & 45 & 28 & 5 & 12 & 21  \\ \hline
        27 & 6 &  11 &  22   & 35 & 62 & 51 & 46  \\ \hline
        34 & 63 & 50 & 47 & 26 & 7 & 10 & 23  \\ \hline
        25 & 8 & 9 & 24 & 33 & 64 & 49 & 48  \\ \hline
        37 & 60 & 53 & 44 & 29 & 4 & 13 & 20  \\ \hline
        30 & 3 & 14 & 19 & 38 & 59 & 54 & 43  \\ \hline
        39 & 58 & 55 & 42 & 31 & 2 & 15 & 18  \\ \hline
        32 & 1 & 16 &  17 & 40 & 57 & 56 & 41  \\ \hline
      \end{array} 
    \end{array}
    \]
  }

  Also verify that the new Narayana square has the following
  additional properties: Half row sums add either to $(n/4)m$ or
  $M-(n/4)m$; Half column sums add to $M/2 \pm n^2/8$; Left and right
  bend diagonal sums add to $ M \pm n/2$; Top and bottom bend diagonal
  sums add to $ M \pm n^2/2$.

\end{exm}

\end{document}